\documentclass[11pt]{amsart}
\usepackage{amsmath,amsthm,amscd,amssymb,latexsym,upref,hyperref,enumerate,enumitem,xfrac,a4wide}
\usepackage[utf8]{inputenc} 

\usepackage{color}
\definecolor{darkgreen}{rgb}{0,0.6,0}

\definecolor{darkblue}{rgb}{0,0,0.5}
\definecolor{darkgrey}{rgb}{0.3,0.3,0.3}

\newcommand{\Deletable}[1]{{\color{darkgrey}#1}}

\newcommand*{\mailto}[1]{\href{mailto:#1}{\nolinkurl{#1}}}

\newtheorem{theorem}{Theorem}[section]
\newtheorem{corollary}[theorem]{Corollary}

\newtheorem{lemma}[theorem]{Lemma}
\newtheorem{proposition}[theorem]{Proposition}

\theoremstyle{definition}

\newtheorem{remark}[theorem]{Remark}

\newcommand{\calB}{\mathcal{B}}
\newcommand{\bbE}{\mathbb{E}}
\newcommand{\calE}{\mathcal{E}}

\newcommand{\del}{\partial}
\newcommand{\Exp}{\mathrm{e}}

\newcommand{\bbG}{\mathbb{G}}

\newcommand{\calHT}{\mathcal{HT}}

\newcommand{\re}{\mathrm{Re}}

\newcommand{\abs}[1]{| #1 |}
\newcommand{\norm}[1]{\| #1 \|}

\renewcommand{\(}{\left(}
\renewcommand{\)}{\right)}

\newcommand{\supp}{\operatorname{supp}}
\newcommand{\sfT}{\mathsf T}
\newcommand{\eps}{\epsilon}

\newcommand{\lam}{\lambda}
\newcommand{\om}{\omega}
\newcommand{\Om}{\Omega}

\newcommand{\bbR}{{\mathbb R}}
\newcommand{\R}{{\mathbb R}}
\newcommand{\bbN}{{\mathbb N}}
\newcommand{\N}{{\mathbb N}}

\newcommand{\C}{{\mathbb C}}
\newcommand{\K}{{\mathbb K}}

\newcommand{\calH}{\mathcal{H}}
\newcommand{\D}{\mathcal{D}}

\newcommand{\E}{{\mathbb E}}
\newcommand{\F}{{\mathbb F}}

\newcommand{\Hb}{{\mathbb H}}
\newcommand{\W}{{\mathbb W}}

\newcommand{\0}{\mspace{0mu}_0}

\numberwithin{equation}{section}

\begin{document}
\title[The Blackstock--Crighton equation in $L_p$-spaces]
{Optimal regularity and exponential stability for the Blackstock--Crighton equation in $L_p$-spaces with Dirichlet and Neumann boundary conditions}

\author[R.\ Brunnhuber]{Rainer Brunnhuber}
\address{Institut f\"ur Mathematik \\ Universit\"at Klagenfurt\\
Universit\"atsstra{\ss}e 65-57\\ 9020 Klagenfurt am W\"orthersee\\ Austria}
\email{\mailto{rainer.brunnhuber@aau.at}}

\author[S.\ Meyer]{Stefan Meyer}
\address{Institut f\"ur Mathematik - Naturwissenschaftliche Fakult\"at II \\ Martin Luther Universit\"at Halle-Wittenberg\\ 
Theodor-Lieser-Stra{\ss}e 5\\ 06120 Halle (Saale) \\ Germany}
\email{\mailto{stefan.meyer@mathematik.uni-halle.de}}


\begin{abstract}
The Blackstock--Crighton equation models nonlinear acoustic wave propagation in thermo-viscous fluids.  In the present work we investigate the associated 
inhomogeneous Dirichlet and Neumann boundary value problems in a bounded domain and prove long-time well-posedness and exponential stability for 
sufficiently small data.  The solution depends analytically on the data.  In the Dirichlet case, the solution decays to zero and the same holds for Neumann 
conditions if the data have zero mean. 
  
We choose an optimal $L_p$-setting, where the regularity of the initial and boundary data are necessary and sufficient for existence, uniqueness and 
regularity of the solution.  The linearized model with homogeneous boundary conditions is represented as an abstract evolution equation for which 
we show maximal $L_p$-regularity.  In order to eliminate inhomogeneous boundary conditions, we establish a general higher regularity result for the heat equation.  
We conclude that the linearized model induces a topological linear isomorphism and then solve the nonlinear problem by means of the implicit function theorem.
\end{abstract}

\thanks{R.B. was supported by the Austrian Science Fund (FWF): P24970}
\maketitle

\setcounter{tocdepth}{1}
\vspace*{-0.5cm}
\tableofcontents
\vspace*{-0.5cm}


\section{Introduction}
\label{sec:introduction}
An acoustic wave propagates through a medium as a local pressure change. Nonlinear effects typically occur in case of acoustic waves of high amplitude 
which are used for several medical and industrial purposes such as lithotripsy, thermotherapy, ultrasound cleaning or welding and sonochemistry.
Research on mathematical aspects of nonlinear acoustic wave propagation is therefore not only interesting 
from a mathematicians point of view. In fact, in case of medical applications, enhancement of the mathematical 
understanding of the underlying models should lead to a considerable reduction of complication risks.

The present work aims to provide a mathematical analysis of the Blackstock--Crighton--Kuznetsov equation
\begin{equation}
\label{BCK:psi}
\left(a\Delta - \partial_t \right)
\left(u_{tt}-c^2\Delta  u - b\Delta u_t \right) =
\big(\tfrac{1}{c^2} \tfrac{B}{2A}( u_t)^2 + \left|\nabla u \right|^2 \big)_{tt}
\end{equation}
and the Blackstock--Crighton--Westervelt equation
\begin{equation}
\label{BCW:psi}
\left(a\Delta - \partial_t \right)
\left( u_{tt}-c^2\Delta u - b\Delta u_t \right) =
\big(\tfrac{1}{c^2}\left(1+\tfrac{B}{2A}\right)( u_t)^2\big)_{tt}
\end{equation}
for the acoustic velocity potential $u$, where $c$ is the speed of sound, $b$ is the diffusivity of sound and $a$ is the heat conductivity of the fluid. Note that $a = \nu \mathrm{Pr}$, where $\nu$ is the is kinematic viscosity and $\mathrm{Pr}$ denotes the
Prandtl number.  
Alternatively, \eqref{BCK:psi} and \eqref{BCW:psi} can be expressed in terms of the acoustic pressure $p$ via the pressure density relation $\rho u_t=p$, where $\rho$ denotes the mass density. The quantity $B/A$ is known as the parameter of nonlinearity and is proportional to the ratio of the coefficients of the quadratic and linear terms in the Taylor series expansion of the variations of the pressure in a medium in terms of variations of the density. Note that \eqref{BCW:psi} is obtained from \eqref{BCK:psi} by neglecting local nonlinear effects in the sense that the expression $c^2|\nabla u|^2-( u_t)^2$ is sufficiently small. For a detailed introduction to the theory and applications of nonlinear acoustics we refer to \cite{HaBl98}.

Equations \eqref{BCK:psi} and \eqref{BCW:psi} result from two evolution equations of fourth order governing 
finite-amplitude sound in thermoviscous relaxing fluids, namely
\begin{align}
\label{Crighton11}
-c^2 a\Delta^2 u+\left(a+b\right)\Delta u_{tt}+\left(c^2\Delta u- u_{ttt}\right)
&=\big(\left|\nabla u\right|^2_t+\tfrac{B}{2A} u_t\Delta u\big)_{t},\\
 \label{Crighton13}
\left(a\Delta - \partial_t \right)\left( u_{tt}-c^2\Delta u \right) 
&=\big(\left|\nabla u\right|^2_t+\tfrac{B}{2A} u_t\Delta u\big)_{t},
\end{align}
which have been derived by Blackstock \cite{Bla63} from the basic equations describing the general motion of thermally relaxing, viscous fluids 
(continuity equation, momentum equation, entropy equation and an arbitrary equation of state) and also 
appear as equations (11) and (13) in Crighton's work \cite{Cri79} on nonlinear acoustic models. 
We replace $\Delta  u$ in the last term of \eqref{Crighton11} and \eqref{Crighton13} by $\frac{1}{c^2} u_{tt}$, which can be justified by the main 
part of the differential operator corresponding to the wave equation $ u_{tt}-c^2 \Delta  u =0$. 
Moreover, in \eqref{Crighton13} we consider potential diffusivity as in \eqref{Crighton11}. 
Therewith, we arrive at equation \eqref{BCK:psi} for which in \cite{Bru14} the name Blackstock--Crighton--Kuznetsov equation has 
been introduced. For a more rigorous derivation of \eqref{BCK:psi} we refer to Section 2 in \cite{Bru14}.

While \eqref{BCK:psi} and \eqref{BCW:psi} are enhanced models in nonlinear acoustics, the Kuznetsov
\begin{equation}
\label{Kuznetsov:psi}
 u_{tt}-b\Delta u_t - c^2\Delta u= \left(\tfrac{1}{c^2}\tfrac{B}{2A} ( u_t)^2 + | \nabla u|^2\right)_t
\end{equation}
and the Westervelt equation 
\begin{equation}
\label{Westervelt:psi}
 u_{tt}-b\Delta u_t - c^2\Delta u=\left(\tfrac{1}{c^2}\left(1+\tfrac{B}{2A}\right) ( u_t)^2\right)_t,
\end{equation}
are classical, well-accepted and widely used models governing sound propagation in fluids. 
As \eqref{BCK:psi} and \eqref{BCW:psi}, they are derived from the basic equations in fluid mechanics. 
The Kuznetsov equation is the more general one of these classical models, in particular the Westervelt equation is obtained 
from the Kuznetsov equation by neglecting local nonlinear effects.
Moreover, for a small ratio of $\nu$ and $\mbox{Pr}$, that is, for small heat conductivity, \eqref{Kuznetsov:psi} and \eqref{Westervelt:psi} 
can be regarded as simplifications of \eqref{BCK:psi} and \eqref{BCW:psi}, respectively.

The classical models \eqref{Kuznetsov:psi} and \eqref{Westervelt:psi} have recently been extensively investigated. In particular, results on well-posedness for the Kuznetsov and the Westervelt equation with homogeneous Dirichlet \cite{KaLa09} and inhomogeneous Dirichlet \cite{KLV11}, \cite{KaLa12} and Neumann \cite{KaLa11} boundary conditions have recently been shown in an $L_2(\Omega)$-setting on spatial domains $\Omega \subset \R^n$ of dimension $n\in \{1,2,3\}$. Moreover, there are results on optimal regularity and long-time behavior of solutions for the Westervelt equation with homogeneous Dirichlet \cite{MeWi11} and for the Kuznetsov equation with inhomogeneous Dirichlet \cite{MeWi13} boundary conditions in $L_p(\Omega)$-spaces where the spatial domain $\Omega \subset \R^n$
is of arbitrary dimension.

On the contrary, mathematical research on higher order partial differential equations arising in nonlinear acoustics is still in an early stage. Well-posedness and exponential decay results for the homogeneous Dirichlet boundary value problems associated with \eqref{BCK:psi} and \eqref{BCW:psi} in an $L_2(\Omega)$-setting where $\Omega \subset \R^n$, $n\in\{1,2,3\}$, have been shown in \cite{Bru14} and \cite{BrKa14}, respectively. In the present work we consider \eqref{BCK:psi} and \eqref{BCW:psi} with inhomogeneous Dirichlet and Neumann boundary conditions in $L_p(\Omega)$-spaces where the spatial domain $\Omega$ is of dimension $n \in \N$. We show global well-posedness and long-time behavior of solutions in an optimal functional analytic setting in the sense that the regularity of the solution is necessary and sufficient for the regularity of the initial and boundary data. 
While in \cite{Bru14} and \cite{BrKa14} the results were proved by means of appropriate energy estimates and the Banach fixed-point theorem, the techniques used in the present paper are based on maximal $L_p$-regularity for parabolic problems and the implicit function theorem in Banach spaces. 

We suppose that $\Omega \subset \R^n$, $n\in\N$, is a bounded domain, i.e., an open, connected and bounded subset of the $n$-dimensional Euclidean space, with smooth boundary $\Gamma$. 
Let $J=(0,T)$ for some finite $T>0$ or $J=\R_+= (0,\infty)$. We consider the inhomogeneous Dirichlet boundary value problem
\begin{equation}
\label{IBVP:Dirichlet}
\begin{cases}
\begin{aligned}
(a\Delta - \partial_t)(u_{tt}-b\Delta u_t - c^2 \Delta u)&= (k (u_t)^2+s|\nabla u|^2)_{tt}		&& \text{in } J \times \Omega,\\
(u, \Delta u)&= (g,h)														&& \text{on }J \times \Gamma,\\
(u,u_t,u_{tt})&=(u_0, u_1, u_2)												&& \text{on } \{t=0\}\times\Omega,
\end{aligned}
\end{cases}
\end{equation}
and the inhomogeneous Neumann boundary value problem
\begin{equation}
\label{IBVP:Neumann}
\begin{cases}
\begin{aligned}
(a\Delta - \partial_t)(u_{tt}-b\Delta u_t - c^2 \Delta u)&= (k (u_t)^2+s|\nabla u|^2)_{tt}		&& \text{in } J \times \Omega,\\
(\partial_\nu u, \partial_\nu \Delta u)&= (g,h)									&& \text{on }J \times \Gamma,\\
(u,u_t,u_{tt})&=(u_0, u_1, u_2)												&& \text{on } \{t=0\}\times\Omega,
\end{aligned}
\end{cases}
\end{equation}
where $u_0, u_1, u_2 \colon \Omega \rightarrow \R$ and $g, h: J \times \Gamma \rightarrow \R$ are given, $u: J \times \Omega \rightarrow \R$ is the unknown, 
$u(t,x)$, and $a$, $b$, $c$ and $k$ are positive constants. Moreover, $\partial_\nu u = \nu \cdot \nabla u|_\Gamma$ where $\nu$ is the outer normal unit vector denotes the normal derivative of $u$. 
The parameter $s\in \{0,1\}$ allows us to switch between \eqref{BCK:psi} and \eqref{BCW:psi}.

We point out that the present work extends the results from \cite{Bru14} in several ways.
First, while in \cite{Bru14} the Blackstock--Crighton equation was considered with homogeneous Dirichlet boundary 
conditions, we also allow for inhomogeneous Dirichlet as well as Neumann boundary conditions.
We are able to remove the restriction $n\in \{1,2,3\}$ on the dimension of the spatial domain $\Om$. Instead of $L_2(\Om)$, we consider \eqref{BCK:psi} and \eqref{BCW:psi} in $L_p(\Om)$ 
where $p\in (1,\infty)$ in case of the linearized equation and $p > \max \{n/4 +1/2, n/3\}$ in case of the nonlinear equations \eqref{IBVP:Dirichlet} and \eqref{IBVP:Neumann}.
In particular, we require $p\in(5/4,\infty)$ in case $n=3$ and then $p=2$ is admissible. 
Moreover, most notably, our conditions on the regularity of the data $(g,h,u_0, u_1, u_2)$ are necessary and sufficient for the
existence of a unique solution of the Blackstock--Crighton equation (within a certain regularity class/a certain subspace of $L_p(J\times\Omega)$).

Our strategy for solving \eqref{IBVP:Dirichlet} and \eqref{IBVP:Neumann} is to prove that their linearizations induce isomorphisms between suitable Banach spaces and 
to apply the implicit function theorem. In some sense these linearizations can be considered as a composition of a heat problem and another linearized problem for the Westervelt equation. 
While the linearized Westervelt equation can be handled similar as in \cite{MeWi11,MeWi13}, the heat equation has to be solved with higher regularity conditions.

The paper is organized as follows. The purpose of Section \ref{sec:prerequisites} is to recall several facts we need on our way to global well-posedness and exponential 
stability of \eqref{IBVP:Dirichlet} and \eqref{IBVP:Neumann}. In particular, we mention all function spaces we use, provide facts about the homogeneous Dirichlet and 
Neumann Laplace operator and list some important embeddings and traces. 
We also give a short review of the concept of maximal $L_p$-regularity for parabolic problems. Furthermore, we recall respectively prove optimal regularity results 
for the heat equation and the linearized Westervelt equation with inhomogeneous Dirichlet and Neumann boundary conditions. 

Section \ref{sec:Dirichlet} is devoted to the inhomogeneous Dirichlet boundary value problem \eqref{IBVP:Dirichlet}. First of all we consider the corresponding linear problem 
and represent it as an abstract parabolic evolution equation for which we show maximal $L_p$-regularity. This gives us optimal regularity for the homogeneous linear version 
of \eqref{IBVP:Dirichlet}. Based on the optimal regularity results for the heat and the linearized Westervelt equation from Section \ref{sec:prerequisites} we prove optimal 
regularity for the linear inhomogeneous Dirichlet boundary value problem in Proposition \ref{prop:linearinhom}. The main result in this section is Theorem \ref{thm:global:Dirichlet} 
which states global well-posedness of \eqref{IBVP:Dirichlet} and immediately implies exponential stability (Theorem \ref{thm:decay:Dirichlet}).

In Section \ref{sec:Neumann} we treat the inhomogeneous Neumann boundary value problem \eqref{IBVP:Neumann}. Here, we proceed analogously to Section \ref{sec:Dirichlet}. 
Theorem \ref{thm:local:Neumann} provides local well-posedness for \eqref{IBVP:Neumann}. In the Neumann case, global well-posedness is shown for data having zero mean 
(Theorem \ref{thm:global:Neumann}). Moreover, Theorem \ref{thm:decay:Neumann} states long-time behavior of solutions.

In Appendix \ref{app:neumann} we collect several facts about the homogeneous Neumann Laplacian. We outline how one finds a realization of the Laplacian with homogeneous 
Neumann boundary conditions such that its spectrum is contained in the positive half-line.

In Appendix \ref{sec:trac-mixed-deriv} we first study the temporal trace operator acting on a class of anisotropic Sobolev spaces. We present its mapping properties, 
provide a right-inverse and thus obtain its precise range space.  Moreover, we construct functions with prescribed higher-order initial data.  Second, we prove some so-called mixed 
derivative embeddings which are often used for checking the continuity of differential operators acting on anisotropic spaces.

In Appendix \ref{app:heathigher} we prove some higher regularity results for the heat equation with inhomogeneous Dirichlet or Neumann boundary conditions and inhomogeneous 
initial conditions in a far more general framework than needed in the main text.  In particular, we state explicitly all necessary compatibility conditions between initial and boundary 
data and show how they are used to contruct a solution with high regularity.

\section{Preliminaries} 
\label{sec:prerequisites}
The purpose of this section is to introduce the notation we are going to use throughout the paper and to recall several important facts and results we need to prove global well-posedness and long-time behavior of solutions for \eqref{IBVP:Dirichlet} and \eqref{IBVP:Neumann}.
As already mentioned in Section \ref{sec:introduction}, we always assume that the spatial domain 
$\Omega \subset \R^n$, $n\in\N$, is bounded and has smooth boundary $\Gamma=\partial\Omega$. We write $J$ for a 
time interval and consider either $J=(0,T)$ for some finite time horizon $T>0$ or $J=\R_+=(0,\infty)$.

\subsection{Function spaces, operators, embeddings and traces} 
The space $BUC^k(\Om)$ contains all $k$-times Fr\'echet differentiable functions $\Om \rightarrow \R$, whose derivatives up to order $k$ are bounded and uniformly continuous. For $p\in(1,\infty)$, let $L_p(\Omega)$ denote the space of (equivalence classes of) Lebesgue measurable $p$-integrable functions $\Omega \rightarrow \R$. 
We write $W_p^m(\Omega)$ for the Sobolev--Slobodeckij space and $H_p^m(\Om)$ for the Bessel potential space of order $m\in[0,\infty)$, where we have $W_p^m(\Om) = H_p^m(\Om)$ if $m \in \N_0$. Moreover, $W^m_p(\Omega;X)$ and $H^m_p(\Omega;X)$ denote the vector valued versions and $_0W_p^1(J;X)$ denotes the space of all functions $u \in W_p^1(J;X)$ with $u(0)=0$.
For $p \in [1,\infty)$, $q\in[1,\infty]$, $s\in \R_+$,  the Besov space $B_{p,q}^s(\Omega)$ is defined as 
$(L_p(\Omega), W_p^m(\Omega))_{s/m,q}$ where $m=\lceil s \rceil$ and $(\cdot , \cdot )_{s/m,q}$ indicates real interpolation. 
It holds that $B_{p,p}^s (\Omega) = W_p^s(\Omega)$ if $s \in \R_+\setminus \N$ and $B_{p,q}^s(\Omega) = W_p^s(\Omega)$ 
if $p=q=2$. Moreover,
\begin{equation}
\label{Besov:interpolation}
B_{p,q}^s(\Omega)=(W_p^k(\Omega), W_p^m(\Omega))_{\Theta,q}, 
\qquad \text{where } 0 \leq k <s<m \text{ and } s= (1-\Theta)k + \Theta m.
\end{equation}
We always write $X \hookrightarrow Y$ if the Banach space $X$ is continuously embedded into the 
Banach space $Y$. Moreover, let $L(X,Y)$ be the space of all bounded linear operators between $X$ and $Y$. A linear operator $A \colon X\to Y$ is called an isomorphism if it is bounded and bijective. Then the closed graph theorem implies that $A^{-1} \colon Y\to X$ is also bounded and therefore $A\colon X\to Y$ is a homeomorphism.
Now, let $X$ and $\mathbb{X}$ be Banach spaces such that $\mathbb{X} \hookrightarrow L_{1,loc}(J; X)$ where $L_{1,loc}(J;X)$ is the space of locally integrable functions $J\rightarrow X$. 
For any $\omega \in \R$ we define the exponentially weighted space
\begin{equation*}
\mathrm{e}^\omega \mathbb{X} = \{u \in L_{1,loc}(J;X)\colon \mathrm{e}^{-\omega t}u \in \mathbb{X} \},
\end{equation*}
equipped with the norm $\|u\|_{\mathrm{e}^\omega \mathbb{X}} = \|\mathrm{e}^{-\omega t} u\|_{\mathbb{X}}$ where $\mathrm{e}^{-\omega t} u$ denotes the mapping $[t \mapsto \mathrm{e}^{-\omega t} u(t)]$.

Let $-\Delta_D\colon \D(\Delta_D) \rightarrow L_p(\Omega)$, $u \mapsto -\Delta u$ denote the negative Dirichlet Laplacian with domain $\D(\Delta_D) = \{u \in W_p^2(\Omega)\colon u=0 \text{ on }\Gamma\}$.  We recall that the spectrum $\sigma(-\Delta_D)$ is a discrete subset of $(0,\infty)$ consisting only of eigenvalues $\lam_n^D=\lam_n(-\Delta_D)$, $n\in\N_0$ with finite multiplicity. We write $\lam_0^D>0$ for the smallest eigenvalue of $-\Delta_D$.
Moreover, the negative Neumann Laplacian $-\Delta_N: \D(\Delta_N) \rightarrow L_p(\Omega)$, $u\mapsto -\Delta u$ with domain $\D(\Delta_N)= \{u \in W_p^2(\Omega) \colon \partial_\nu u =0 \text{ on } \Gamma\}$ has a discrete spectrum $\sigma(-\Delta_N)\subset[0,\infty)$ which contains only eigenvalues $\lam_n^N= \lam_n(-\Delta_N)$, $n\in \N_0$, of finite multiplicity. Here, $\lam_0^N=0$ is an isolated point of $\sigma(-\Delta_N)$ which can be removed when introducing the space $L_{p,0}(\Om)=\{u \in L_p(\Om)\colon \int_\Om u \, dx=0\}$ and considering $-\Delta_{N,0}\colon \D(\Delta_{N,0}) \rightarrow L_{p,0}(\Om)$ , $u\mapsto -\Delta u$ with $\D(\Delta_{N,0}) = \D(\Delta_N) \cap L_{p,0}(\Omega)$. We then have $\sigma(-\Delta_{N,0}) \subset (0,\infty)$ where $\lam_1^N = \lam_1(-\Delta_N)>0$ is the smallest eigenvalue of $-\Delta_{N,0}$. For details we refer to Appendix \ref{app:neumann}.

We shall use the embeddings $W_p^1(J) \hookrightarrow BUC(J)$ and $W_p^s(\Omega) \hookrightarrow W_p^t(\Omega)$ for $s\geq t$. We always write $\gamma_D=  \cdot |_\Gamma$ and $\gamma_N= \del_\nu \cdot |_\Gamma = \nu \cdot (\nabla \cdot)|_\Gamma$
for the Dirichlet and the Neumann trace, respectively. Moreover, ${\gamma_t = }\cdot |_{t=0}$ denotes the temporal trace. 
Let $B\in \{D,N\}$, $j_D=0$, $j_N=1$. For $p\in (1,\infty)$, $k\in \N$ and $l\in \N_0$ the spatial trace
\begin{equation}
\label{trace:spatial}
\begin{aligned}
u \mapsto \gamma_B u \colon &  W_p^{k+l}(J; L_p(\Om))\cap W_p^k(J; W_p^{2l}(\Om)) \\
&\qquad\rightarrow W_p^{k+l-j_B/2-1/2p}(J; L_p(\Gamma)) \cap W_p^k(J; W_p^{2l-j_B-1/p}(\Gamma))
\end{aligned}
\end{equation}
is bounded, see Appendix \ref{sec:trac-mixed-deriv}.  Furthermore, the trace
\begin{equation}
\label{trace:spatial2}
u \mapsto \gamma_B u \colon W_p^s(\Om) \rightarrow B_{p,p}^{s-j_B-1/p}(\Gamma)
\end{equation}
is bounded for every $s\in (j_B+1/p, \infty)$, cf. \cite[Theorem 3.3.3]{Tri83}.
The temporal trace
\begin{equation}
\label{trace:temporal}
\gamma_t : u \mapsto u|_{t=0}\colon W_p^\alpha(J; W_p^s(\Om)) \cap L_p(J; W_p^{s+2\alpha}(\Om)) \rightarrow B_{p,p}^{s+2\alpha-2/p}(\Om)
\end{equation}
is bounded for $\alpha \in (1/p,1]$, $s\in [0,\infty)$ and $s+2\alpha \notin \N$ for $\alpha<1$. The same holds when the 
domain $\Om$ is replaced by its boundary $\Gamma$. 
Finally, we mention that for $p\in(1,\infty)$, $t,s\in\N_0$, $\tau,\sigma\in\N$, $\theta\in(0,1)$ we have the mixed derivative embedding
\begin{equation}
\label{emb:mixed}
W^{t+\tau}_{p}(J;W^s_{p}(\Omega)) \cap W^t_{p}(J;W^{s+\sigma}_{p}(\Omega)) \hookrightarrow W^{t+\theta\tau}_{p}(J;W^{s+(1-\theta)\sigma}_{p}(\Omega)),
\end{equation}
where again $\Om$ can be replaced by $\Gamma$. For more general embeddings of this form we refer to Appendix \ref{sec:trac-mixed-deriv}.

\subsection{Maximal \texorpdfstring{$L_p$}{Lp}-regularity} 
Let $J=(0,T)$ or $J=\R_+=(0,\infty)$ and assume $p\in (1,\infty)$.
We say that a closed linear operator $A:\mathcal{D}(A) \rightarrow X$ with dense domain 
$\mathcal{D}(A)$ in a Banach space $X$ 
admits maximal $L_p$-regularity on $J$ if for each $F\in L_p(J;X)$ the abstract Cauchy problem
\begin{equation}
\label{Cauchy:abstract}
v_t(t) + A v(t) = F(t), \,  t \in J,\qquad v(0)=v_0,
\end{equation}
admits a unique solution $u\in \E(J)=W_p^1(J;X) \cap L_p(J;\mathcal{D}(A))$ for $v_0=0$.

Furthermore, the abstract inhomogeneous Cauchy problem \eqref{Cauchy:abstract} is said to admit 
maximal $L_p$-regularity, if
\begin{equation}
(\del_t+A,\gamma_t) : \E(J) \rightarrow L_p(J;X) \times \mathrm{tr}\, \E(J), \qquad v \mapsto (F,v_0)
\end{equation}
is a homeomorphism. Then its inverse is the solution map
\begin{equation}
(\del_t+A,\gamma_t)^{-1}: L_p(J;X) \times \mathrm{tr}\, \E(J) \rightarrow \E(J), \qquad (F, v_0) \mapsto v.
\end{equation} 
If $A:\mathcal{D}(A) \rightarrow X$ has maximal $L_p$-regularity on $J$,
then the abstract Cauchy problem \eqref{Cauchy:abstract} has maximal $L_p$-regularity on $J$, cf. Section III.1.5 in 
\cite{Ama95}. The following result is very useful and will be used several times throughout this paper.
\begin{lemma}[cf. {\cite[Proposition III.1.5.3]{Ama95}}]
\label{lem:maximalexp}
Let $\alpha \in \R$. Suppose that 
\begin{equation*}
(\partial_t +\alpha + A, \gamma_t)\colon  \mathbb{E}(J) \rightarrow L_p(J;X) \times \mathrm{tr}\, \E(J)
\end{equation*}
is a homeomorphism. Then
\begin{equation*}
(\partial_t + A,\gamma_t)\colon  \Exp^\alpha \mathbb{E}(J) \rightarrow \Exp^{\alpha} L_p(J;X) \times \mathrm{tr}\, \E(J)
\end{equation*}
is a homeomorphism. 
\end{lemma}

\subsection{Optimal regularity results}
In order to prove our results on optimal regularity for the linearized versions of \eqref{IBVP:Dirichlet} and \eqref{IBVP:Neumann}, 
we need optimal regularity results for the heat equation and the linearized Westervelt equation. We always let $a,b,c\in(0,\infty)$.

\subsubsection{Dirichlet boundary conditions} Recall that $\lambda_0^D>0$ always denotes the smallest eigenvalue
of the negative Dirichlet Laplacian in $L_p(\Omega)$. 
\begin{lemma}[{\cite[Proposition 8]{LatPrSchn06}}]
\label{lem:heatinhom}
Let $p \in (1, \infty)$ and $\omega \in (0, a\lambda_0^D)$. Then the initial boundary 
value problem for the heat equation
\begin{equation}
\label{heat:linear}
\begin{cases}
\begin{aligned}
u_t-a\Delta u &= f 		&&\text{ in } \R_+ \times \Omega,\\
		   u&=g 		&&\text{ on } \R_+ \times \Gamma,\\ 
		   u&=u_0 		&&\text{ on } \{t=0\} \times \Omega,

\end{aligned}
\end{cases}
\end{equation}
has a unique solution 
\begin{equation*}
u \in \Exp^{-\omega} \Hb_u, \qquad \Hb_u= W_p^1(\R_+;L_p(\Omega)) \cap L_p(\R_+;W_p^2(\Omega)),
\end{equation*}
if and only if the given data $f$, $g$ and $u_0$ satisfy the regularity conditions
\begin{enumerate}[label=\emph{(\roman*)}, leftmargin=*]
\item $f \in \Exp^{-\omega} L_p(\R_+ \times \Omega)$, 
\item $u_0 \in W_p^{2-2/p}(\Omega)$,
\item $g \in \Exp^{-\omega} \Hb_\Gamma$, $\Hb_\Gamma= W_p^{1-1/2p}(\R_+,L_p(\Gamma)) \cap L_p(\R_+; W_p^{2-1/p}(\Gamma))$, 
\item $u_0 |_{\Gamma} = g |_{t=0}$ in the sense of traces.
\end{enumerate}
\end{lemma}

\begin{lemma}[{\cite[Lemma 5]{MeWi13}}]
\label{lem:Westerveltinhom} Suppose $p \in (1,\infty)$, $p\neq 3/2$\Deletable{, $n\in\N$} and define $\omega_0=\min\{b\lambda_0^D/2, c^2/b\}$. 
Then for every $\omega \in (0,\omega_0)$ there exists a unique solution
\begin{equation*}
u \in \Exp^{-\omega} \W_u, \qquad \W_u=W_p^2(\R_+; L_p(\Omega)) \cap W_p^1(\R_+;W_p^2(\Omega))
\end{equation*}
of the linear initial boundary value problem
\begin{equation}
\begin{cases}
\begin{aligned}
u_{tt}-b\Delta u_t - c^2 \Delta u &= f, 	      			&& \text{in } \R_+ \times \Omega,\\
					     u &= g, 				&& \text{on } \R_+ \times \Gamma,\\
				    (u, u_t) &= (u_0,u_1)		&& \text{on } \{t=0\} \times \Omega,			
\end{aligned}
\end{cases}
\end{equation}
if and only if the data satisfy the following conditions:
\begin{enumerate}[label=\emph{(\roman*)}, leftmargin=*]
\item $f \in \Exp^{-\omega} L_p(\R_+ \times \Omega)$,
\item $u_0 \in W_p^2(\Omega)$, $u_1 \in W_p^{2-2/p}(\Omega)$,
\item $g \in \Exp^{-\omega}\W_\Gamma$, $\W_\Gamma = W_p^{2-1/2p}(\R_+; L_p(\Gamma)) \cap W_p^1(\R_+; W_p^{2-1/p}(\Gamma))$,
\item $g |_{t=0}=u_0|_{\Gamma}$ and if $p>3/2$ also $g_t|_{t=0}=u_1|_{\Gamma}$ in the sense of traces.
\end{enumerate}
\end{lemma}

\subsubsection{Neumann boundary conditions} 
Now we prove optimal regularity results for the heat equation and the linearized Westervelt equation with Neumann boundary conditions.  
Recall that $\lam_1^N>0$ denotes the smallest eigenvalue of the negative homogeneous Neumann Laplacian in $L_{p,0}(\Omega)$.
Let $\bar{u} = \abs\Omega^{-1} \int_\Omega u\, dx$ denote the mean of a function $u:\Omega\to\bbR$ and let 
$\bar g = \abs\Gamma^{-1}\int_\Gamma g \, dS$ for $g:\Gamma\to\bbR$.

\begin{lemma}
\label{lem:heatNeumann}
  Let  $p \in (1,\infty)\setminus\{3\}$ and $\omega \in [0, a \lam_1^N)$. Then the inhomogeneous Neumann boundary value problem for the heat equation
  \begin{equation}\label{prob:inhom_Neumann_Laplacian}
    \begin{cases}
      \begin{aligned}
        u_t - a \Delta u &= f 	      		&& \text{in } \R_+ \times \Omega,\\   
        \partial_\nu u &= g 						&& \text{on } \R_+ \times \Gamma,\\	
        u &= u_0									&& \text{on } \{t=0\} \times \Omega,      
      \end{aligned}
    \end{cases}
  \end{equation}
  admits a unique solution of the form $u(t,x) = v(t,x) + w(t)$ with
  \begin{equation*}
    v \in \Exp^{-\omega} \Hb_{u,0}, \quad \Hb_{u,0} =W_p^1(\R_+; L_{p,0}(\Omega)) \cap L_p(\R_+; W_p^2(\Omega) \cap L_{p,0}(\Omega)), 
    \qquad w_t \in \Exp^{-\omega }L_p(\R_+),
  \end{equation*}
  if and only if the data satisfy the following conditions:
  \begin{enumerate}[label=\emph{(\roman*)}, leftmargin=*]
  \item $f \in \Exp^{-\omega} L_p(\R_+; L_p(\Omega))$,
  \item $u_0 \in W_p^{2-2/p}(\Omega)$,
  \item $g \in \Exp^{-\omega}\Hb_{\nu} $, $ \Hb_{\nu} = W_p^{1/2-1/2p}(\R_+; L_p(\Gamma)) \cap L_p(\R_+; W_p^{1-1/p}(\Gamma))$,
  \item $g |_{t=0}=\partial_\nu u_0|_{\Gamma}$ in the sense of traces if $p>3$.
  \end{enumerate}
  If in addition $f(t,\cdot)$, $u_0$, $g(t,\cdot)$ have mean value zero over $\Omega$ resp.\ $\Gamma$ for all $t$, then $w=0$.
\end{lemma}
\begin{proof}
We first let $\omega=0$.  By \cite[Theorem 8.2]{DHP03}, Lemma \ref{lem:spectrum_of_Neumann_Laplacian} and \cite[Theorem 2.4]{Dor93}, 
the Neumann Laplacian in $L_{p,0}(\Omega)$ with domain $\D(\Delta_{N,0}) = \D(\Delta_N) \cap L_{p,0}(\Om)$ 
has maximal regularity on $\R_+$. We therefore obtain a unique solution $u_3\in \Hb_{u,0}$ of the problem
\begin{equation*}
\partial_t u_3 - a \Delta u_3 = f_3 \text{ in } \R_+ \times \Omega,\qquad  \partial_\nu u_3 = 0 \text{ on } \R_+ \times \Gamma, \qquad u_3(0)=0 \text{ in }\Omega
\end{equation*}
for every given $f_3\in L_p(\R_+;L_{p,0}(\Omega))$.  Furthermore, problem \eqref{prob:inhom_Neumann_Laplacian} admits at most one solution.  
Indeed, let us construct it as $u = u_1+u_2+u_3$ where we first solve
\begin{equation*}
 \partial_t u_1 + \mu u_1 - a \Delta u_1 = 0 \text{ in }\R_+ \times \Omega, \qquad \partial_\nu u_1 = g \text{ on } \R_+ \times \Gamma, \qquad u_1(0) = u_0 \text{ in } \Omega,
\end{equation*}
for some sufficiently large $\mu>0$ with \cite[Theorem 2.1]{DHP07}. Next, we let $u_2$ solve the ordinary differential equation
  \begin{equation*}
    \partial_t u_2(t) = \bar f(t)+\mu \bar u_1(t), \qquad u_2(0) = 0,
  \end{equation*}
Finally, with $f_3 = f-\bar f+\mu (u_1-\bar u_1)$, we obtain $u_3$ as above.  It is easy to check that $v = u_1+u_3$ and $w = u_2$ satisfy the assertion.  
The case $\omega>0$ can be reduced to the previous one by multiplying the functions $u$, $f$, $g$ with $t\mapsto \Exp^{\omega t}$ and using 
that the spectrum of $-a\Delta_{N,0}+\omega$ is contained in $(0,\infty)$.
\end{proof} 

\begin{lemma}
\label{lem:Neumannint}
Let $p \in (1,\infty)\setminus\{3\}$ and $\omega \in (0, b \lam_1^N)$. Then the inhomogeneous Neumann boundary value problem 
  \begin{equation}\label{prob:inhom_Neumann2}
    \begin{cases}
      \begin{aligned}
        u_{tt} - b \Delta u_t &= f 	      		&& \text{in } \R_+ \times \Omega,\\        
        \partial_\nu u &= g 				&& \text{on } \R_+ \times \Gamma,\\		
        (u,u_t) &= (u_0,u_1)				&& \text{on } \{t=0\} \times \Omega,
     \end{aligned}
    \end{cases}
  \end{equation}
  admits a unique solution of the form $u(t,x) = v(t,x) + w(t)$ with
  \begin{align*}
   & v \in \Exp^{-\omega} \W_{u,0}, \quad 
     \W_{u,0} =W_p^2(\R_+; L_{p,0}(\Omega))\cap W_p^1(\R_+; W_p^2(\Omega)\cap L_{p,0}(\Omega)), \\ 
    &w\in L_{p,\text{loc}}([0,\infty)),\, w_{tt} \in \Exp^{-\omega }L_p(\R_+),
  \end{align*}
  if and only if the data satisfy the following conditions:
  \begin{enumerate}[label=\emph{(\roman*)}, leftmargin=*]
  \item $f \in \Exp^{-\omega} L_p(\R_+ \times \Omega)$,
  \item $u_0 \in W_p^2(\Omega)$, $u_1 \in W_p^{2-2/p}(\Omega)$,
  \item $g \in \Exp^{-\omega} \W_\nu$, $ \W_\nu = W_p^{3/2-1/2p}(\R_+; L_p(\Gamma)) \cap W_p^1(\R_+; W_p^{1-1/p}(\Gamma))$,
  \item $g |_{t=0}=\partial_\nu u_0|_{\Gamma}$ and if $p>3$ also $g_t |_{t=0}=\partial_\nu u_1|_{\Gamma}$ in the sense of traces,
  \item $\int_0^\infty f(t)\, dt= b \Delta u_0 - u_1$.
  \end{enumerate}  
  Moreover, $w$ satisfies $w_{tt}(t)=\bar f(t)+b\abs\Gamma \abs\Omega^{-1}\bar g_t(t)$ with $w(0) = \bar u_0$ and $w_t(0)= \bar u_1$. 
 \end{lemma}

\begin{proof} 
We start by proving sufficiency. First, note that $\partial_t\colon \W_\nu \rightarrow \Hb_\nu$ is bounded and 
$\|\Exp^{\omega t} g_t\|_{\Hb_\nu} = \|(\Exp^{\omega t} g)_t - \omega \Exp^{\omega t} g\|_{\Hb_\nu} \lesssim \| \Exp^{\omega t} g\|_{\W_\nu}$
which implies $g_t \in \Exp^{-\omega}\Hb_\nu$. Therefore, from Lemma \ref{lem:heatNeumann} we obtain that the heat problem
\begin{equation*}
   \varphi_t - b \Delta \varphi = f \text{ in } \R_+ \times \Omega, \qquad \partial_\nu \varphi = g_t \text{ on } \R_+ \times \Gamma, \qquad \varphi(0)= u_1 \text{ in } \Omega,
\end{equation*}
admits a unique solution of the form $\varphi(t,x) = \varphi_1(t,x) + \varphi_2(t)$ such that $\varphi_1 \in \Exp^{-\omega}  \Hb_{u,0}$ 
and $\del_t \varphi_2 \in \Exp^{-\omega}L_p(\R_+)$. In particular, since $\varphi_1$ has zero mean over $\Omega$, we have $\varphi_2(0)= \bar{\varphi}(0)= \bar{u}_1$ 
and $\varphi_1(0) = u_1- \bar{u}_1$.
For $x \in \Omega$ and $t \in \R_+$ we define $u(t,x)=v(t,x)+w(t)$, where
\begin{equation*}
v(t,x) = - \int_t^\infty \varphi_1(s,x)\, ds \qquad \text{and} \qquad w(t) = - \int_t^\infty \varphi_2(s)\, ds.
\end{equation*}
Clearly $u_t = \varphi$, hence $u_{tt} - b \Delta u_t= f$ in $\Omega$. Integrating the latter with respect to space, multiplying 
with $\abs{\Omega}^{-1}$ and using the identity $\int_\Omega \Delta u \, dx = \int_\Gamma \del_\nu u \, dS$, we deduce that 
$w$ solves the ordinary differential equation 
\begin{equation*}
w_{tt}(t) = \bar{f}(t) + b \abs{\Gamma} \abs{\Omega}^{-1} \bar{g}_t(t), \quad w(0)= \bar{u}_0, \quad w_t(0)=\bar{u}_1. 
\end{equation*}
This implies $v_{tt} - b\Delta v_t = f - \bar{f} - b \abs{\Gamma} \abs{\Omega}^{-1} \bar{g}_t$ in $\Omega$.
In what follows, we abbreviate $v(t) = v(t,\cdot)$, $\varphi_1(t) = \varphi_1(t,\cdot)$ etc. and we let $\chi_{\bbR_-}(t) = 1$ for $t<0$ and $\chi_{\bbR_-}(t) = 0$ for $t>0$. Using $\omega>0$, $v_t = \varphi_1$ and the identity
  \begin{align*}
\Exp^{\omega t} v(t) = - \int_t^\infty \Exp^{\omega (t-s)} \Exp^{\omega s} \varphi_1(s) \, ds = -(\Exp^{\omega \cdot} \chi_{\R_-}) \ast (\Exp^{\omega \cdot} \varphi_1) (t)    
  \end{align*}
together with Young's inequality implies $v\in \Exp^{-\omega} \W_{u,0}$. Moreover, we have
\begin{align*}
\del_\nu v(t) |_\Gamma =\,  &\del_\nu u(t) |_\Gamma = - \int_t^\infty \del_\nu \varphi(s)|_\Gamma \, ds = - \int_t^\infty g_s(s) \, ds= g(t),\\
v_t(t) = \varphi_1(t) = &- \int_t^\infty \del_s \varphi_1(s) \, ds = - \int_t^\infty ( f(s) +b\Delta \varphi_1(s) - \del_s \varphi_2(s) ) \, ds\\
				= & - \int_t^\infty f(s) \, ds + b\Delta v(t) - \varphi_2(t)
\end{align*}
and $v_t(0)=\varphi_2(0)=u_1-\bar{u}_1$. Altogether, $b\Delta v(0) = v_t(0) + \bar{u}_1 + \int_0^\infty f(s)\, ds = b\Delta (u_0 - \bar{u}_0)$
and $\del_\nu v(0) |_\Gamma = g(0) = \del_\nu (u_0 - \bar{u}_0)$ which implies that $v(0)=u_0 - \bar{u}_0$ in $\Omega$. 

To verify necessity of (i)--(v), we assume that $u(t,x) =v(t,x)+w(t)$ with $v \in \Exp^{-\om}  \W_{u,0}$ and $w_{tt} \in \Exp^{-\om} L_p(\R_+)$ 
is a solution of \eqref{prob:inhom_Neumann2}. We have $\Exp^{\om t}f= \Exp^{\om t} v_{tt} + \Exp^{\om t} w_{tt} + b \Delta (\Exp^{\omega t} v_t)$. 
Since $\Exp^{\om t} v_t = (\Exp^{\om t} v)_t - \om \Exp^{\om t} v \in L_p(\R_+; W_p^2(\Om) \cap L_{p,0}(\Om))$ and 
$\Exp^{\om t}v_{tt} = (\Exp^{\om t}v)_{tt} -2 \om \Exp^{\om t} v_t - \om^2 \Exp^{\om t} v \in L_p(\R_+; L_{p,0}(\Om))$ we conclude $\Exp^{\om t} f \in L_p(\R_+ \times \Om)$
and (i) is verified.
Concerning (ii) note that exponential weights do not affect the initial regularity. Due to $W_p^1(\R_+) \hookrightarrow BUC(\R_+)$ 
we infer $v|_{t=0} \in W_p^2(\Om) \cap L_{p,0}(\Om)$, hence $u|_{t=0} = v|_{t=0} + w|_{t=0} \in W_p^2(\Om)$. Furthermore, we have
$\Exp^{\om t} v_t \in W_p^1(\R_+; L_{p,0}(\Om)) \cap L_p(\R_+; W_p^2(\Om)\cap L_{p,0}(\Om))$. Applying the temporal trace \eqref{trace:temporal} 
with $\alpha=1$ and $s=0$ implies $v_t |_{t=0} \in B_{p,p}^{2-2/p}(\Om) = W_p^{2-2/p}(\Om)$, hence $u_t |_{t=0} = v_t |_{t=0} + w_t |_{t=0} \in W_p^{2-2/p}(\Om)$. 
In order to check (iii), we apply the spatial trace \eqref{trace:spatial} with $k=l=1$ to $\Exp^{\om t} v \in  \W_{u,0}$ which gives us $\Exp^{\om t} \del_\nu u |_\Gamma \in\W_\nu$ as claimed.
Next, note that \eqref{trace:spatial2} applied to $u_0 \in W_p^2(\Omega)$ implies $\del_\nu u_0 \in B_{p,p}^{1-1/p}(\Gamma) = W_p^{1-1/p}(\Gamma)$. 
Moreover, we have $g \in W_p^1(\R_+; W_p^{1-1/p}(\Gamma)) \hookrightarrow BUC(\R_+, W_p^{1-1/p}(\Gamma))$, hence $\del_\nu u_0 |_\Gamma = g|_{t=0}$ in $W_p^{1-1/p}(\Gamma)$.
From $u_1 \in W_p^{2-2/p}(\Om)$ we obtain $\del_\nu u_1 \in B_{p,p}^{1-3/p}(\Gamma)= W_p^{1-3/p}(\Gamma)$ for $p>3$ and, from
$g_t \in \mathcal{H}_{\nu}$, using \eqref{trace:temporal} with $\alpha=1/2-1/2p$ and $s=0$, we get
$g_t |_{t=0} \in B_{p,p}^{1-3/p}(\Gamma) = W_p^{1-3/p}(\Gamma)$ if $p>3$. 
Altogether, $\del_\nu u_1|_\Gamma = g_t|_{t=0}$ in $W_p^{1-3/p}(\Om)$ for $p>3$. 
For (v), note that integrating $u_{tt}(t)-b\Delta u_t(t) =f(t)$ with respect to time yields $u_t(t) - b\Delta u(t) = -\int_t^\infty f(s)\, ds$, hence in particular 
$u_t(0) - b\Delta u(0) = -\int_0^\infty f(s)\, ds$. Therewith the proof of necessity is complete. 

Finally, we show that \eqref{prob:inhom_Neumann2} has at most one solution. To this end, suppose we have given two solutions of \eqref{prob:inhom_Neumann2}. 
Their difference $\hat{u}$ solves
\begin{equation*}
\hat{u}_{tt}- b\Delta \hat{u}_t = 0 \text{ in }\R_+ \times \Om, \qquad \del_\nu \hat{u} = 0 \text{ on } \R_+ \times \Gamma, \qquad \hat{u}(0) = \hat{u}_t(0)=0 \text{ in } \Omega.
\end{equation*}
Furthermore, $\check{u} = \hat{u}_t$ solves the heat problem 
\begin{equation*}
\check{u}_{t}- b\Delta \check{u}_t = 0 \text{ in }\R_+ \times \Om, \qquad \del_\nu \check{u} = 0 \text{ on } \R_+ \times \Gamma, \qquad \check{u}=0 \text{ in } \Omega,
\end{equation*}
which implies $\hat{u}_t=0$. Hence $\hat{u}$ is constant which together with $\hat{u}(0)=0$ implies $\hat{u}=0$.
\end{proof}

\begin{lemma}
\label{lem:Westerveltmaximal}
Let $p\in(1,\infty)\setminus\{3\}$ and $\omega_0= \min\{b \lam_1^N /2, c^2/b\}$. Then for every $\omega \in (0, \omega_0)$ the initial boundary value problem
\begin{equation}
\label{Westervelt:Neumannhom}
\begin{cases}
\begin{aligned}
u_{tt}- b \Delta u_t - c^2\Delta u &= f, 	      		&& \text{in } \R_+ \times \Omega,\\  	       		    
			\partial_\nu u &= 0, 				&& \text{on } \R_+ \times \Gamma,\\			
	                 	     (u, u_t) &= (u_0, u_1)		&& \text{on } \{t=0\} \times \Omega,
\end{aligned}
\end{cases}
\end{equation}
has a unique solution 
\begin{equation*}
u \in \Exp^{-\omega}  \W_{u,0}, \qquad  \W_{u,0} = W_p^2(\R_+; L_{p,0}(\Omega)) \cap W_p^1(\R_+; W_p^2(\Omega) \cap L_{p,0}(\Omega))
\end{equation*}
if and only if 
\begin{enumerate}[label=\emph{(\roman*)}, leftmargin=*]
\item $f \in \Exp^{-\omega} L_p(\R_+; L_{p,0}(\Omega))$,
\item $u_0 \in W_p^2(\Omega) \cap L_{p,0}(\Omega)$ and $u_1 \in W_p^{2-2/p}(\Omega) \cap L_{p,0}(\Omega)$ such that $\partial_\nu u_0 |_\Gamma =0$ and if $p>3$ 
additionally $\partial_\nu u_1 |_\Gamma=0$.
\end{enumerate}
\end{lemma}
\begin{proof} In \cite{MeWi11} the result was established with Dirichlet instead of Neumann boundary conditions. 
Here, we just point out the main steps of the proof. We represent \eqref{Westervelt:Neumannhom} as
\begin{equation*}
\begin{pmatrix} u_t \\ u_{tt} \end{pmatrix} + \begin{pmatrix} 0 & -I \\ -c^2\Delta_{N,0} & - b\Delta_{N,0} \end{pmatrix}
 \begin{pmatrix}u \\ u_t\end{pmatrix} = \begin{pmatrix} 0 \\ f \end{pmatrix}, 
\qquad \begin{pmatrix} u(0) \\ u_t(0) \end{pmatrix} = \begin{pmatrix} u_0\\ u_1 \end{pmatrix}.
\end{equation*}
Let us consider the space $\tilde X= \D(\Delta_{N,0}) \times L_{p,0}(\Omega)$ and the operator
$\tilde A\colon \D(\tilde A)\rightarrow \tilde X$ given by
\begin{equation}
\label{def:A}
\tilde A= \begin{pmatrix} 0 & -I \\ -c^2\Delta_{N,0} & - b\Delta_{N,0} \end{pmatrix}, \qquad \D(\tilde A) = \D(\Delta_{N,0}) \times \D(\Delta_{N,0}).
\end{equation}
First we show that there is some $\nu>0$ such that $\nu+\tilde A$ admits maximal $L_p$-regularity on $\R_+$ by a perturbation
argument. Choosing a decomposition of $\tilde A$, $\tilde A= \tilde A_1+\tilde A_2$ with
\begin{equation*}
\tilde A_1= \begin{pmatrix} \alpha I & -I \\ 0 & - b\Delta_{N,0} \end{pmatrix} \qquad \text{and} \qquad \tilde A_2= \begin{pmatrix} -\alpha I & 0\\ -c^2\Delta_{N,0} & 0 \end{pmatrix}
\end{equation*}
for some $\alpha >0$, it turns out that operator $\tilde A_1\colon \D(\tilde A)\rightarrow \tilde X$ admits maximal regularity on $\R_+$ due to Lemma 
\ref{lem:heatNeumann} and the fact that the bounded operator $(\partial_t + \alpha)\colon ~_0W_p^1(\R_+; \D(\Delta_{N,0})) \rightarrow L_p(\R_+; \D(\Delta_{N,0}))$ 
is invertible. 
Moreover, since $\tilde A_2\colon \tilde X \rightarrow \tilde X$ is bounded, on the strength of Proposition 4.3 and Theorem 4.4 in \cite{DHP03} there exists some $\nu >0$ 
such that $\nu+\tilde A_1+\tilde A_2$ admits maximal regularity on $\R_+$. 

Next we claim that the spectral bound $s(-\tilde A)= \sup \{\mathrm{Re}(\lambda)\colon \lambda \in \sigma(-\tilde A)\}$ of $-\tilde A$ is given by $s(-\tilde A) = - \omega_0$. 
This follows analogously to \cite[Lemma 2.4]{MeWi11} if one replaces the Dirichlet eigenvalues by the ones of $-\Delta_{N,0}$.

Since $s(-\tilde A)=-\omega_0<0$, the spectral bound of $-\tilde A+\omega$ equals $\omega - \omega_0$ which is strictly negative as long as $\omega \in [0, \omega_0)$.
Hence, for each $\omega \in [0, \omega_0)$ the operator $\tilde A-\omega$ has maximal $L_p$-regularity on $\R_+$ by \cite[Theorem 2.4]{Dor93}, that is
\begin{equation*}
(\partial_t +\tilde A-\omega, \gamma_t )\colon  W_p^1(\R_+; \tilde X) \cap L_p(\R_+; \D(\tilde A)) \rightarrow L_p(\R_+;\tilde X) \times (\tilde X, D(\tilde A))_{1-1/p,p}
\end{equation*}
is an isomorphism.
Employing Lemma \ref{lem:maximalexp} we conclude that for every $\omega \in [0,\omega_0)$ the operator
\begin{equation*}
(\partial_t +\tilde A, \gamma_t)\colon \Exp^{-\omega} (W_p^1(\R_+; \tilde X) \cap L_p(\R_+; \D(\tilde A))) 
	\rightarrow \Exp^{-\omega}L_p(\R_+;\tilde X) \times (\tilde X, D(\tilde A))_{1-1/p,p}
\end{equation*}
is an isomorphism. 

It is easy to check that $(u, u_t) \in \Exp^{-\omega} (W_p^1(\R_+; \tilde X) \cap L_p(\R_+; \D(\tilde A)))$ implies $u \in \Exp^{-\omega} \W_{u,0}$. Moreover,
we have $f \in L_p(\R_+; L_{p,0}(\Omega))$ and  $u_0 \in \D(\Delta_N)$, i.e. $u_0 \in W_p^2(\Om) \cap L_{p,0}(\Om)$ such that $\del_\nu u_0 |_\Gamma=0$. 
Finally, $u_1 \in (L_{p,0}(\Om), \D(\Delta_N))_{1-1/p,p} = W_p^{2-2/p}(\Om) \cap L_{p,0}(\Om)$ with $\partial_\nu u_1 |_\Gamma =0$ where the trace exists if
$p>3$. This concludes the proof.
\end{proof}

Finally we arrive at our optimal regularity result for the linearized Westervelt equation with inhomogeneous Neumann boundary conditions.
\begin{lemma}
\label{lem:Westervelt:inhom:Neumann} Let $p \in (1,\infty)\setminus \{3\}$ and set $\omega_0=\min\{b\lam_1^N/2, c^2/b\}$. 
Then for every $\omega \in (0,\omega_0)$ the linear initial boundary value problem 
\begin{equation}
\label{Westervelt:Neumanndata}
\begin{cases}
\begin{aligned}
u_{tt}-b\Delta u_t - c^2 \Delta u &= f, 	      			&& \text{in } \R_+ \times \Omega,\\			    
			\partial_\nu u &= g, 				&& \text{on } \R_+ \times \Gamma,\\
				(u, u_t) &= (u_0,u_1)			&& \text{on } \{t=0\} \times \Omega,
				
\end{aligned}
\end{cases}
\end{equation}
admits a unique solution of the form $u(t,x)=v(t,x) + w(t)$, where  
\begin{equation*}
v \in \Exp^{-\omega}\W_{u,0}, \quad \W_{u,0}=W_p^2(\R_+; L_{p,0}(\Omega)) \cap W_p^1(\R_+;W_p^2(\Omega)\cap L_{p,0}(\Om)), \quad  w_{tt} \in \Exp^{-\omega} L_p(\R_+) 
\end{equation*}
if and only if the data satisfy the following conditions:
\begin{enumerate}[label=\emph{(\roman*)}, leftmargin=*]
\item $f \in \Exp^{-\omega} L_p(\R_+ \times \Omega)$,
\item $u_0 \in W_p^2(\Omega)$, $u_1 \in W_p^{2-2/p}(\Omega)$,
\item $g \in \Exp^{-\omega}\W_\nu$, $\W_\nu= W_p^{3/2-1/2p}(\R_+; L_p(\Gamma)) \cap W_p^1(\R_+; W_p^{1-1/p}(\Gamma))$,
\item $g |_{t=0}= \del_\nu u_0|_{\Gamma}$ and if $p>3$ additionally $g_t|_{t=0}=\del_\nu u_1|_{\Gamma}$ in the sense of traces.
\end{enumerate}
\end{lemma}
\begin{proof}
From Lemma \ref{lem:Westerveltmaximal} we obtain uniqueness. In order to show necessity of (i)--(iv) one proceeds as in the proof of Lemma \ref{lem:Neumannint}.
It therefore remains to show sufficiency. 
Let $\delta > \om$. From Lemma \ref{lem:Neumannint} we obtain that 
\begin{equation*}
\begin{cases}
\begin{aligned}
\varphi_{tt}-b\Delta{\varphi}_t&= f-f_\delta	      					&& \text{in } \R_+ \times \Omega,\\			    	
			\partial_\nu \varphi &= g, 						&& \text{on } \R_+ \times \Gamma,\\
			 (\varphi, \varphi_t) &= (u_0,u_1)							&& \text{on } \{t=0\} \times \Omega,			
\end{aligned}
\end{cases}
\end{equation*}  
where $f_\delta= \delta \Exp^{-\delta t} ( \int_0^\infty f(s) \, ds + u_1 - b \Delta u_0)$, admits a unique solution $\varphi(x,t)=\varphi_v(x,t) + \varphi_w(t)$ such that $\varphi_v \in \Exp^{-\om} \W_{u,0}$ and $\partial_t^2\varphi_w \in \Exp^{-\om} L_p(\R_+)$.
Next, Lemma \ref{lem:Westerveltmaximal} implies that
\begin{equation*}
\begin{cases}
\begin{aligned}
\theta_{v,tt}-b\Delta \theta_{v,t} -c^2\Delta \theta_v &= f_\delta - \bar{f}_\delta + c^2 \Delta \varphi_v - c^2 \overline{\Delta \varphi_v}	      	&& \text{in } \R_+ \times \Omega,\\			   
					 \partial_\nu \theta_v &= 0, 															&& \text{on } \R_+ \times \Gamma,\\
				    (\theta_v, \theta_{v,t}) &= (0,0)															&& \text{on } \{t=0\} \times \Omega,		
\end{aligned}
\end{cases}
\end{equation*}
has a unique solution $\theta \in \Exp^{-\om}  \W_{u,0}$. Furthermore, we define $\theta_w$ as the solution of the ordinary differential equation
  \begin{equation*}
   \theta_{w,tt}(t) =  c^2 \overline{\Delta \varphi_v}(t) + \bar{f}_\delta(t) , \qquad \theta_w(0)= 0, \quad \theta_{w,t}(0)=0.
  \end{equation*}  
Then $v=\varphi_v + \theta_v$ and $w= \varphi_w + \theta_w$ satisfy the assertion and we are done.
\end{proof}

\begin{remark}
\label{rem:Westerveltfinite}
If we consider \eqref{Westervelt:Neumanndata} on a finite time interval $J=(0,T)$ instead of $\R_+$, we may set $\om = 0$ and obtain a 
unique solution $u \in W_p^2(J; L_p(\Om)) \cap W_p^1(J; W_p^2(\Om))$ if and only if conditions (i)--(iv) (with $\om=0$) hold. 
\end{remark}

\subsection{Analysis in Banach spaces} 
For later use in the proof of global well-posedness of \eqref{IBVP:Dirichlet} and \eqref{IBVP:Neumann} we will now recall the concept of 
analytic mappings in Banach spaces and the analytic version of the implicit function theorem. The remainder of this section is collected 
from Section 15.1 in \cite{Dei85}.

Let $X$ and $Y$ be Banach spaces over the same field $\mathbb{K}= \R$ or $\mathbb{K}= \C$ and let $U\subset X$ be open. 
Then $F\colon U \rightarrow Y$ is called analytic at $x_0 \in U$ if there is some $r>0$ and continuous symmetric $k$-linear operators 
$F_k\colon X^k = X \times \dots \times X \rightarrow Y$ for $k\geq 1$ such that 
\begin{equation*}
\sum_{k=1}^\infty \|F_k\| \|h\|^k< \infty \qquad \text{and} \qquad  F(x_0+h) = F(x_0) + \sum_{k=1}^\infty F_k(h^k).
\end{equation*}
for $h\in X$, $\|h\| <r$. Here, $\|F_k\| = \sup \{ \|F_k(h^k)\|\colon \|h\| \leq 1\}$. We then necessarily have $F_k = \tfrac{1}{k!} F^{(k)}(x_0)$.
The map $F\colon U \rightarrow Y$ is called analytic if $F$ is analytic at every $x_0 \in U$.
In particular, every bounded linear map $F: X \rightarrow Y$ is analytic.

\begin{theorem}[Implicit Function Theorem, cf. {\cite[15.1]{Dei85}}]
\label{thm:implicit}
Let $X$, $Y$ and $Z$ be Banach spaces over the same field $\mathbb{K}=\R$ and $\mathbb{K}= \C$.
Assume $U\subset X$ and $V \subset Y$ are neighborhoods of $x_0 \in X$ 
and $y_0 \in Y$, respectively. Furthermore, suppose 
\begin{enumerate}[label=\emph{(\roman*)}, leftmargin=*]
\item $F\colon U \times V \rightarrow Z$, $(x,y) \mapsto F(x,y)$ is continuous,
\item the Fr\'echet derivative $F_x\colon U \times V \rightarrow L(X,Z)$ of $F$ with respect to $x$ is continuous,
\item $F(x_0,y_0) = 0$ and $F_x(x_0,y_0) \colon X \rightarrow Z$ is an isomorphism.
\end{enumerate}
Then there exist balls $B_\eta(x_0) \subset U$ and $B_\rho(y_0) \subset V$ and a unique map 
$\varphi\colon B_\rho(y_0) \rightarrow B_\eta(x_0)$ such that $\varphi(y_0)= x_0$ and $F(\varphi(y), y) = 0$ for all $y \in B_\rho(y_0)$. 
The map $\varphi$ is continuous. If furthermore, $F$ is analytic, then $\varphi$ is analytic in some neighborhood of $y_0$, in particular on 
some (possibly smaller) ball $B_{\kappa}(y_0) \subset B_\rho(y_0)$.
\end{theorem}


\section{The Dirichlet boundary value problem}
\label{sec:Dirichlet}
In this section we prove global well-posedness and exponential stability for \eqref{IBVP:Dirichlet}. First of all, we consider the linearized 
version of the inhomogeneous Dirichlet boundary value problem and represent it as an abstract evolution equation.
We show that this abstract equation admits maximal $L_p$-regularity and derive an optimal regularity result for the 
linearized equation associated with \eqref{IBVP:Dirichlet}. Then we use the implicit function theorem to construct a solution of
the nonlinear problem \eqref{IBVP:Dirichlet}. Exponential decay of this solution is an immediate consequence. 

\subsection{Maximal \texorpdfstring{$L_p$}{Lp}-regularity for the linearized equation}
\label{subsec:linear:Dirichlet}
Suppose $J=(0,T)$ or $J=\R_+$. For $f\in L_p(J\times\Omega)$ we consider the initial boundary value problem
\begin{equation}
\label{IBVP:linear:Dirichlet}
\begin{cases}
\begin{aligned}
(a\Delta - \partial_t)(u_{tt}-b\Delta u_t - c^2 \Delta u)&= f			&& \text{in } J \times \Omega,\\
(u, \Delta u)&= (g,h)										&& \text{on }J \times \Gamma,\\
(u,u_t,u_{tt})&=(u_0, u_1, u_2)								&& \text{on } \{t=0\}\times\Omega,
\end{aligned}
\end{cases}
\end{equation}
where $u_0, u_1, u_2\colon \Omega \rightarrow \R$ and $g,h\colon J \times \Gamma \rightarrow \R$ 
are the given initial and boundary data, respectively. 
In order to address the problem of maximal $L_p$-regularity for the linearized equation, we represent 
\eqref{IBVP:linear:Dirichlet} with $g=h=0$ as an abstract Cauchy problem
  \begin{align*}
    \(\del_t+\begin{pmatrix} 0 & -I & 0 \\ -c^2 \Delta_D & -b\Delta_D & -I \\ 0 & 0 & -a_D\Delta \end{pmatrix}\) 
    \begin{pmatrix} u \\ u_t \\ u_{tt}-b \Delta_D u_t - c^2 \Delta u_D \end{pmatrix} = \begin{pmatrix} 0 \\ 0 \\ -f \end{pmatrix}. 
  \end{align*}
This motivates us to consider the Banach space
\begin{equation}
\label{space:X}
X^D=\D((\Delta_D)^2) \times \D(\Delta_D) \times L_p(\Omega)
\end{equation}
and the densely defined linear operator $A^D: \mathcal{D}(A^D) \rightarrow X^D$ given by 
\begin{equation}
\label{operator:A}
A^D= \begin{pmatrix} 0 & -I & 0 \\ -c^2 \Delta_D & -b\Delta_D & -I \\ 0 & 0 & -a\Delta_D \end{pmatrix}, \qquad
\mathcal{D}(A^D)= \D((\Delta_D)^2) \times \D((\Delta_D)^2) \times\D(\Delta_D).
\end{equation} 
Therewith, we may write \eqref{IBVP:linear:Dirichlet} as an abstract evolution equation
\begin{equation*}
\del_t v^D + A v^D = F, \qquad v^D(0)= v_0^D
\end{equation*}
if we define  
\begin{equation}
\label{def:vf}
 v^D = \begin{pmatrix} u \\ u_t \\ u_{tt}-b \Delta_D u_t - c^2 \Delta_D u \end{pmatrix}, 
 \quad  v_0^D = \begin{pmatrix} u_0 \\ u_1 \\ u_2-b \Delta_D u_1 - c^2 \Delta_D u_0 \end{pmatrix},
 \quad F\begin{pmatrix} 0 \\ 0 \\ -f \end{pmatrix}. 
\end{equation}
First of all we will treat the issue of maximal $L_p$-regularity of $A^D: \mathcal{D}(A^D) \rightarrow X^D$ on $\R_+$.

\begin{proposition} 
\label{thm:maximal:Dirichlet}
Let $p\in (1,\infty)$. There is a constant $\mu > 0$ such that $\mu+A^D$ has maximal $L_p$-regularity on $\R_+$.
\end{proposition}
\begin{proof}
Let $\alpha >0$. We decompose $A^D$, $A^D=A_1^D+ A_2^D$, where 
\begin{equation*}
A_1^D= \begin{pmatrix} \alpha I & -I & 0 \\ 0 & -b\Delta_D & -I \\ 0 & 0 & -a\Delta_D \end{pmatrix}
\qquad \text{and} \qquad
A_2^D = \begin{pmatrix} -\alpha I & 0 & 0 \\ -c^2\Delta_D & 0 & 0 \\ 0 & 0 & 0 \end{pmatrix}.
\end{equation*}
First we show that $A_1^D: \mathcal{D}(A^D) \rightarrow X^D$ has maximal $L_p$-regularity. To this end, 
we consider the Cauchy problem $v_t+A_1^Dv=F$, $v_0=0$ and show that for each 
$F \in L_p(\R_+;X^D)$ there exists a unique solution $v\in W_p^1(\R_+; X^D) \cap L_p(\R_+; \mathcal{D}(A^D))$. 
With $v=(v_1, v_2, v_3)^\top$ and $F=(f_1, f_2, f_3)^\top$, we explicitly have
\begin{align*}
\partial_t v_1 + \alpha v_1 - v_2 &= f_1, 		&& v_1(0)=0,\\
\partial_t v_2 - b\Delta_D v_2 - v_3&= f_2, 	&& v_2(0)=0, \\
\partial_t v_3 - a \Delta_D v_3 &= f_3, 			&& v_3(0)=0.
\end{align*}
Since we know from Lemma \ref{lem:heatinhom} that the homogeneous heat equation admits maximal $L_p$-regularity, 
we obtain that for all $f_3 \in L_p(\R_+ \times \Omega)$ there exists a unique solution
\begin{equation*}
v_3 \in W_p^1(\R_+; L_p(\Omega)) \cap L_p(\R_+;\D(\Delta_D)).
\end{equation*}
Moreover, as $f_2+v_3 \in L_p(\R_+; \D(\Delta_D))$, Lemma \ref{lem:interior_higher_regularity} implies that there is a unique solution 
\begin{equation*}
v_2 \in W_p^1(\R_+; \D(\Delta_D)) \cap L_p(\R_+; \D((\Delta_D)^2))
\end{equation*}
Now, note that for $\alpha >0$ the operator $(\partial_t + \alpha)\colon~_0W_p^1(\R_+; \D((\Delta_D)^2)) \rightarrow L_p(\R_+; \D((\Delta_D)^2))$ 
is invertible. Since $f_1+v_2 \in L_p(\R_+; \D((\Delta_D)^2))$ we obtain a unique solution
\begin{equation*}
v_1(t)=\int_0^t \Exp^{-\alpha(t-s)}(f_1(s) + v_2(s)) \, ds,
\end{equation*}
which satisfies $v_1 \in W_p^1(\R_+; \D((\Delta_D)^2))$. Altogether, we conclude that $A_1: \mathcal{D}(A^D) \rightarrow X^D$ 
admits maximal $L_p$-regularity.

Moreover, by the fact that $A_2^D: X^D \rightarrow X^D$ is a bounded linear operator, Proposition 4.3 and Theorem 4.4
in \cite{DHP03} imply that there exists some $\mu > 0$ such that $\mu+A_1^D+A_2^D$
has maximal regularity which concludes the proof.
\end{proof}

In order to show maximal regularity for the operator $A^D:\D(A^D) \rightarrow X^D$, we need the following result on its spectrum.
\begin{lemma}[cf.\ {\cite[Lemma 3.10]{BrKa14}}]
\label{lem:omega}
The spectral bound $s(-A^D)=\sup \{\re(\lam)\colon\lam \in \sigma(-A^D)\}$ of $-A^D$ is given by
$s(-A^D)=-\omega_0^D$, where $\om_0^D=\min \{a\lambda_0^D, b\lambda_0^D /2, c^2/b\}$.
In particular, if $\mathrm{Re}(\lambda) < \omega_0^D$, then $\lambda \in \rho(A^D)$.
\end{lemma}

\begin{theorem} 
\label{thm:maximal:Dirichlet2}
Let $p\in (1,\infty)$ and $\omega \in [0,\omega_0)$ where $\omega_0^D = \min \{a\lambda_0^D, b \lambda_0^D /2, c^2/b\}$. 
Then $A^D: \D(A^D) \rightarrow X^D$ has maximal $L_p$-regularity on $\R_+$ in the sense that
\begin{align*}
(\partial_t +A^D, \gamma_t)\colon &
\Exp^{-\omega} (W_p^1(\R_+; X^D) \cap L_p(\R_+;\D(A^D))) \\
& \qquad \rightarrow \Exp^{-\omega} L_p(\R_+;X^D) \times (X^D, \D(A^D))_{1-1/p,p}
\end{align*}
is an isomorphism.
\end{theorem}
\begin{proof}
We follow the proof of Theorem 2.5 in \cite{MeWi11}.
From Proposition \ref{thm:maximal:Dirichlet} we know $\mu + A^D$ admits maximal regularity on $\R_+$ for some $\mu >0$. 
Multiplying $v_t^D + A^Dv^D=F$ by $\Exp^{-\mu t}$ shows that $A^D$ has maximal $L_p$-regularity on bounded intervals $J=(0,T)$. 
Lemma \ref{lem:omega} tells us that spectral bound $s(-A^D)=-\omega_0$ of $-A^D$ is strictly negative. Hence $s(-A^D+\om) = \om - \om_0^D<0$ as long as 
$\omega \in[0,\om_0^D)$. From \cite[Theorem 2.4]{Dor93} we deduce that $A^D-\om$ admits maximal $L_p$-regularity on $\R_+$
for every $\omega \in [0,\omega_0^D)$, that is,
\begin{equation*}
(\partial_t +A^D-\om, \gamma_t)\colon 
W_p^1(\R_+; X^D) \cap L_p(\R_+;\D(A^D)) \rightarrow L_p(\R_+;X^D) \times (X^D, \D(A^D))_{1-1/p,p}
\end{equation*}
is an isomorphism. Now Lemma \ref{lem:maximalexp} implies the result.
\end{proof}

\begin{corollary} 
\label{cor:Dirichlet:hom}
Let $p\in (1,\infty)\setminus\{3/2\}$ and define $\omega_0^D = \min \{a\lambda_0^D, b\lambda_0^D /2, c^2/b\}$. 
Then for every $\omega \in (0,\omega_0)$ the linear Dirichlet boundary value problem 
\begin{equation}
\label{IBVP:linearhom}
\begin{cases}
\begin{aligned}
(a\Delta- \partial_t)(u_{tt}-b\Delta u_t - c^2 \Delta u)&= f 				&& \text{in } \R_+ \times \Omega,\\
(u, \Delta u)&= (0,0)											&& \text{on } \R_+ \times \Gamma,\\
(u,u_t,u_{tt})&=(u_0, u_1, u_2)									&& \text{on }\{t=0\} \times \Omega,
\end{aligned}
\end{cases}
\end{equation}
admits maximal $L_p$-regularity in the sense that there exists a unique solution 
\begin{equation*}
u \in \Exp^{-\omega} \E_u, \quad \E_u= W_p^3(\R_+; L_p(\Omega)) \cap W_p^1(\R_+;W_p^4(\Omega)),
\end{equation*}
if and only if 
\begin{enumerate}[label=\emph{(\roman*)}, leftmargin=*]
\item $f \in \Exp^{-\omega}L_p(\R_+ \times \Omega)$,
\item $u_0 \in W_p^4(\Omega)$, $u_1 \in W_p^{4-2/p}(\Omega)$, $u_2 \in W_p^{2-2/p}(\Omega)$ 
	    with $u_0 |_\Gamma = \Delta u_0|_\Gamma = u_1 |_\Gamma = 0$ and, if $p>3/2$, also $\Delta u_1 |_\Gamma = u_2 |_\Gamma = 0$ in the sense of traces.
\end{enumerate}
\end{corollary}
\begin{proof} 
Based on the choices of $X^D$ and $\D(A^D)$ in \eqref{space:X} and \eqref{operator:A}, it is straightforward to check that 
the condition $v^D\in \Exp^{-\om} (W_p^1(\R_+; X^D) \cap L_p(\R_+;\D(A^D)))$ where $v^D$ is given by \eqref{def:vf} 
implies $u\in\Exp^{-\omega}(\E_u \cap W_p^2(\R_+; W_p^2(\Omega)))$. 
Since the mixed derivative embedding gives us $\E_u \hookrightarrow W_p^2(\R_+; W_p^2(\Omega))$, we arrive at $u \in \Exp^{-\om} \E_u$. 
Next, we determine $(X, \D(A))_{1-1/p,p}$. It is trivial that $(\D((\Delta_D)^2),\D((\Delta_D)^2))_{1-1/p,p}= \D((\Delta_D)^2)$, 
i.~e.\ we have $u_0 \in W_p^4(\Om)$ with $u|_\Gamma = \Delta u |_\Gamma=0$.
Moreover, since for $p\in(1,\infty)$ we have $2/p \in \R \setminus \N$ unless $p=2$, \eqref{Besov:interpolation}
gives us
\begin{align*}
(W_p^2(\Omega),W_p^4(\Omega))_{1-1/p,p}&= B_{p,p}^{4-2/p}(\Omega) = W_p^{4-2/p}(\Omega),\\
(L_p(\Omega),W_p^2(\Omega))_{1-1/p,p} &= B_{p,p}^{2-2/p}(\Omega) = W_p^{2-2/p}(\Omega).
\end{align*}
Moreover, interpolation with boundary conditions as in \cite[Section 4.9]{Ama09} yields 
$u_1|_\Gamma = \Delta u_1|_\Gamma = u_2 - b\Delta_Du_1 - c^2 \Delta_D u_0 |_\Gamma =0$.
Hence, we have $u_0 \in W_p^4(\Omega)$, $u_1 \in W_p^{4-2/p}(\Omega)$ and  
$u_2 - c^2\Delta_D u_0 - b \Delta_D u_1 \in W_p^{2-2/p}(\Omega)$ which is equivalent to 
$u_0 \in W_p^4(\Omega)$, $u_1 \in W_p^{4-2/p}(\Omega)$ and $u_2\in W_p^{2-2/p}(\Omega)$. 
The result now follows from Theorem \ref{thm:maximal:Dirichlet2}.
\end{proof}

We now arrive at the final result for this section and prove optimal regularity for the linear initial boundary 
value problem \eqref{IBVP:linear:Dirichlet} in the sense that the regularity of the data $f$, $g$, $h$, $u_0$, $u_1$ and $u_2$ 
are necessary and sufficient for the existence of a unique solution $u \in \mathrm{e}^{-\omega} \E_u$.
\begin{proposition}
\label{prop:linearinhom}
Let $p\in(1,\infty)\setminus\{3/2\}$ and define $\omega_0^D = \min \{a\lambda_0^D, b\lambda_0^D/2, c^2/b\}$. 
Then for every $\omega \in (0,\omega_0)$ the linear initial boundary value problem
\begin{equation}
\label{IBVP:linearinhom}
\begin{cases}
\begin{aligned}
(a\Delta- \partial_t)(u_{tt}-b\Delta u_t - c^2 \Delta u)&= f 			&& \text{in } \R_+ \times \Omega,\\
(u, \Delta u)&= (g,h)										&& \text{on } \R_+ \times \Gamma,\\
(u,u_t,u_{tt})&=(u_0, u_1, u_2)								&& \text{on }\{t=0\} \times \Omega,
\end{aligned}
\end{cases}
\end{equation}
has a unique solution
\begin{equation*}
u \in \Exp^{-\omega}\E_u, \quad 
\E_u= W_p^3(\R_+; L_p(\Omega)) \cap W_p^1(\R_+;W_p^4(\Omega)),
\end{equation*}
if and only if the data satisfy the conditions
\begin{enumerate}[label=\emph{(\roman*)}, leftmargin=*]
\item $f \in \Exp^{-\omega}L_p(\R_+ \times \Omega)$,
\item $u_0 \in W_p^4(\Omega)$, $u_1 \in W_p^{4-2/p}(\Omega)$, $u_2 \in W_p^{2-2/p}(\Omega)$,
\item$ g \in \Exp^{-\omega} \F_{g,\Gamma}$, $\F_{g,\Gamma}=W_p^{3-1/2p}(\R_+;L_p(\Gamma)) \cap W_p^1(\R_+;W_p^{4-1/p}(\Gamma))$,\\
$h \in \Exp^{-\omega} \F_{h,\Gamma}$,  $\F_{h,\Gamma}= W_p^{2-1/2p}(\R_+;L_p(\Gamma)) \cap W_p^1(\R_+; W_p^{2-1/p}(\Gamma))$,
\item $u_0 |_\Gamma = g|_{t=0}$, $u_1 |_\Gamma = g_t |_{t=0}$, $\Delta u_0 |_\Gamma = h|_{t=0}$ and, if $p>3/2$, also $\Delta u_1 |_\Gamma = h_t |_{t=0}$
	and $u_2 |_\Gamma = g_{tt} |_{t=0}$ hold in the sense of traces.
\end{enumerate}
Moreover, the solution fulfills the estimate
\begin{equation*}
\|u\|_{\Exp^{-\omega}\E_u} \lesssim \|f\|_{\Exp^{-\omega} L_p} + \|g\|_{\Exp^{-\omega} \F_{g,\Gamma}} 
+  \|h\|_{\Exp^{-\omega} \F_{h,\Gamma}} + \|u_0\|_{W_p^4}+\|u_1\|_{W_p^{4-2/p}}+\|u_2\|_{W_p^{2-2/p}}.
\end{equation*}
\end{proposition}
\begin{proof} First we show necessity of (i)--(iv) for the existence of a unique solution 
$u\in \Exp^{-\omega} \E_u$ of \eqref{IBVP:linearinhom}. In the proof of Corollary \ref{cor:Dirichlet:hom} we already mentioned that
$\E_u \hookrightarrow W_p^2(\R_+ \times \Omega)$. 
Since 
$(\Exp^{\omega t} u)_t  = \omega \Exp^{\omega t} u + \Exp^{\omega t} u_t$, 
$(\Exp^{\omega t} u)_{tt}  = \omega^2 \Exp^{\omega t} u  +2 \omega \Exp^{\omega t} u_t+ \Exp^{\omega t} u_{tt}$ and
$(\Exp^{\omega t} u)_{ttt}  = \omega^3\Exp^{\omega t} u+3 \omega^2 \Exp^{\omega t} u_t+ 3 \omega \Exp^{\omega t} u_{tt}+\Exp^{\omega t} u_{ttt}$,
the assumption that $u\in\Exp^{-\omega} \E_u$ implies
$\Exp^{\omega t} u_t  \in L_p(\R_+;W_p^4(\Omega))\cap W_p^1(\R_+;W_p^2(\Omega)) \cap W_p^2(\R_+;L_p(\Omega))$,
$\Exp^{\omega t} u_{tt}  \in L_p(\R_+;W_p^2(\Omega))\cap W_p^1(\R_+;L_p(\Omega))$ and
$\Exp^{\omega t} u_{ttt}  \in L_p(\R_+\times\Omega)$,
hence $\Exp^{\omega t}f = \Exp^{\omega t} (a\Delta- \partial_t)(u_{tt}-b\Delta u_t - c^2 \Delta u) \in L_p(\R_+\times\Omega)$
and (i) follows. \\
Next, we show (ii). The embedding $W_p^1(J) \hookrightarrow BUC(J)$ implies $u_0 \in W_p^4(\Omega)$ whereas the temporal trace \eqref{trace:temporal} 
with $\alpha=1$, $s=2$ and $\alpha= 1$, $s=0$ gives us the desired regularities of $u_1$ and $u_2$, respectively. \\
For $u \in \Exp^{-\om} \E_u$ the spatial trace \eqref{trace:spatial} with $k=1$ and $l=2$ implies $u |_\Gamma = g \in \Exp^{-\om} \F_{g,\Gamma}$ 
and for $\Delta u \in \Exp^{-\om} \W_u$ the choice $k=l=1$ gives us $\Delta u |_\Gamma = h \in \Exp^{-\om} \F_{h,\Gamma}$. This shows (iii). \\
Using $W_p^1(J) \hookrightarrow BUC(J)$, the spatial trace \eqref{trace:spatial2}, the temporal trace \eqref{trace:temporal} and the mixed derivative embedding 
\eqref{emb:mixed} one shows (iv). We have 
\begin{align*}
&u_0 |_\Gamma = g|_{t=0} \text{ in } W_p^{4-1/p}(\Gamma), &&\Delta u_0 |_\Gamma = h |_{t=0} \text{ in }W_p^{2-1/p}(\Gamma),\\
&u_1|_\Gamma = g_t |_{t=0} \text{ in } B_{p,p}^{4-3/p}(\Gamma),  && \Delta u_1 |_\Gamma = h_t |_{t=0} \text{ in } B_{p,p}^{2-3/p}(\Gamma) \text{ if } p>3/2,\\
&u_2 |_\Gamma = g_{tt} |_{t=0} \text{ in } B_{p,p}^{2-3/p}(\Gamma) \text{ if } p>3/2.
\end{align*}
 

It remains to show that conditions (i)--(iv) imply the existence of a unique solution $u\in \Exp^{-\omega} \E_u$ of 
\eqref{IBVP:linearinhom}. Since we are dealing with a linear partial differential equation with constant coefficients,
we may interchange the order of differentiation on the left-hand side and consider the subproblems 
\begin{equation}
\label{Westervelt:inhom}
\begin{cases}
\begin{aligned}
w_{tt}-b\Delta w_t - c^2 \Delta w &= f 					&& \text{in } \R_+ \times \Omega,\\
w&= a h - g_t										&& \text{on } \R_+ \times \Gamma,\\
(w,w_t)&=(a\Delta u_0-u_1,a \Delta u_1- u_2)			&& \text{on }\{t=0\} \times \Omega,
\end{aligned}
\end{cases}
\end{equation}
and
\begin{equation}
\label{heat:inhom}
\begin{cases}
\begin{aligned}
a\Delta u-u_t &= w						&& \text{in } \R_+ \times \Omega,\\
u&=g								&&\text{on } \R_+ \times \Gamma,\\
u&=u_0								&& \text{on }\{t=0\} \times \Omega.
\end{aligned}
\end{cases}
\end{equation}
From condition (ii) we obtain $a \Delta u_0 - u_1 \in W_p^2(\Omega)$ and $a\Delta u_1- u_2 \in W_p^{2-2/p}(\Omega)$. 
Furthermore, (iii) implies $ah - g_t \in \Exp^{-\omega} \W_\Gamma$. On the strength of Lemma \ref{lem:Westerveltinhom} we obtain that 
\eqref{Westervelt:inhom} admits a unique solution $w\in \Exp^{-\omega}\W_u$. Now we use 
Corollary \ref{cor:higher_regularity_for_the_heat_equation_exponential_weight} with $l=1$ and $k=2$ to solve \eqref{heat:inhom}
and obtain that \eqref{heat:inhom} has a solution $u\in \Exp^{-\omega}\E_u$. This concludes the proof of sufficiency. 
Uniqueness follows from Corollary \ref{cor:Dirichlet:hom}.
\end{proof}

\subsection{Global well-posedness and exponential stability}
Based on Proposition \ref{prop:linearinhom}, we now show that there exists a unique global solution of the nonlinear initial 
boundary value problem \eqref{IBVP:Dirichlet} which depends continuously (in fact, even analytically) on the (sufficiently small) initial 
and boundary data. Moreover, we prove that the equilibrium $u=0$ is exponentially stable. 
\begin{theorem}[Global well-posedness - the Dirichlet case]
\label{thm:global:Dirichlet}
Let $p>\max\{n/4+1/2,n/3\}$, $p\neq 3/2$ and define $\omega_0^D = \min \{a\lambda_0^D, b\lambda_0^D/2, c^2/b\}$. Suppose
\begin{equation}
\label{cond:data}
\begin{aligned}
&u_0 \in W_p^4(\Omega), \qquad u_1 \in W_p^{4-2/p}(\Omega), \qquad u_2 \in W_p^{2-2/p}(\Omega)\\
&g \in \Exp^{-\omega}\F_{g,\Gamma},~  \F_{g,\Gamma}= W_p^{3-1/2p}(\R_+;L_p(\Gamma)) \cap W_p^1(\R_+;W_p^{4-1/p}(\Gamma)),\\
&h \in \Exp^{-\omega} \F_{h,\Gamma},~  \F_{h,\Gamma}= W_p^{2-1/2p}(\R_+;L_p(\Gamma)) \cap W_p^1(\R_+; W_p^{2-1/p}(\Gamma)).
\end{aligned}
\end{equation}
with $u_0 |_\Gamma = g|_{t=0}$, $u_1 |_\Gamma = g_t |_{t=0}$, $\Delta u_0 |_\Gamma = h|_{t=0}$ and, 
if $p>3/2$, also $\Delta u_1 |_\Gamma = h_t |_{t=0}$, $u_2 |_\Gamma = g_{tt} |_{t=0}$. 

Then for every $\omega \in (0, \omega_0)$ there exists some $\rho >0$ such that if
\begin{equation*}
\|g\|_{\Exp^{-\omega}\F_{g,\Gamma}} + \|h\|_{\Exp^{-\omega}\F_{h,\Gamma}} + \|u_0\|_{W_p^4} + \|u_1\|_{W_p^{4-2/p}} + \|u_2\|_{W_p^{2-2/p}} < \rho, 
\end{equation*}
the nonlinear initial boundary value problem \eqref{IBVP:Dirichlet} admits a unique solution
\begin{equation}
\label{sol:nonlinear}
u \in \Exp^{-\omega} \E_u, 
\qquad \E_u =W_p^3(\R_+; L_p(\Omega)) \cap W_p^1(\R_+; W_p^4(\Omega)) 
\end{equation}
which depends analytically on the data \eqref{cond:data} with respect to the 
corresponding topologies. Moreover, conditions \eqref{cond:data} are necessary for the regularity of the solution given 
in \eqref{sol:nonlinear}.
\end{theorem}

\begin{proof} Employing the results on the linearized problem \eqref{IBVP:linear:Dirichlet} from Section \ref{subsec:linear:Dirichlet}, we 
will now construct a solution of the nonlinear initial boundary value problem \eqref{IBVP:Dirichlet} which we linearize at $u=0$. 
Hence, the solution will be of the form $u= u_\star+ u_\bullet$, where $u_\star$ solves the linearized problem \eqref{IBVP:linear:Dirichlet} 
for the data $(f=0, g, h, u_0, u_1, u_2)$ and $u_\bullet$ satisfies homogeneous boundary and initial conditions. 
We will find the (small) deviation $u_\bullet$ from $u_\star$ by application of the implicit function theorem to the map
\begin{equation}
\label{map:G}
\begin{aligned}
G\colon \Exp^{-\omega} \E_{u,h} \times \Exp^{-\omega} \E_u & \rightarrow \Exp^{-\omega}L_p(\R_+ \times \Omega),\\
(u_\bullet, u_\star) & \mapsto D(\partial_t, \Delta) u_\bullet - (k((u_\bullet+ u_\star)_t)^2 -s |\nabla(u_\bullet+u_\star)|^2)_{tt}
\end{aligned}
\end{equation}
where the differential expression $D(\partial_t, \Delta)$ is given by 
$D(\partial_t, \Delta) = (a\Delta - \partial_t)(\partial_t^2 -b\Delta \partial_t - c^2 \Delta)$ and 
$\E_{u,h} = \{u \in \E_u \colon u(0)=u_t(0)=u_{tt}(0) =0,\, u |_\Gamma = \Delta u |_\Gamma =0 \}$. Explicitly, we have 
$D(\partial_t, \Delta)u_\bullet = - u_{\bullet, ttt} + (a+b) \Delta u_{\bullet, tt} + c^2 \Delta u_{\bullet,t} - ab \Delta^2 u_{\bullet,t} - ac^2\Delta^2 u_\bullet$. 
\vspace{2mm}\newline
\textit{Step 1: The implicit function theorem applies.} 
First of all, we will now verify the assumptions of the implicit function theorem (Theorem \ref{thm:implicit}).\vspace{2mm} \newline
\textit{Step 1(a): $G$ is analytic.} The mixed derivative embedding \eqref{eq:mixed_deri_H} implies that the linear maps
\begin{align*}
&u_\bullet \mapsto -u_{\bullet,ttt} \colon 				&& \Exp^{-\omega} \E_{u,h} \rightarrow \Exp^{-\omega} L_p(\R_+ \times \Omega), \\
&u_\bullet \mapsto \Delta u_{\bullet,tt} \colon 		&& \Exp^{-\omega} \E_{u,h} \rightarrow \Exp^{-\omega} L_p(\R_+ \times \Omega), \\
&u_\bullet \mapsto \Delta u_{\bullet,t} \colon 			&& \Exp^{-\omega} \E_{u,h} \rightarrow \Exp^{-\omega} (L_p(\R_+; W_p^2(\Omega)) \cap W_{p}^1(\R_+; L_p(\Omega))) \\
&u_\bullet \mapsto \Delta^2 u_{\bullet,t} \colon 		&& \Exp^{-\omega} \E_{u,h} \rightarrow \Exp^{-\omega} L_p(\R_+ \times \Omega), \\
&u_\bullet \mapsto \Delta^2 u_\bullet \colon 			&& \Exp^{-\omega} \E_{u,h} \rightarrow \Exp^{-\omega} W_p^1(\R_+; L_p(\Omega)), 
\end{align*}
are bounded and therefore analytic. 
Hence $u_\bullet \mapsto D(\partial_t, \Delta)u_\bullet \colon \Exp^{-\omega} \E_{u,h} \rightarrow \Exp^{-\omega}L_p(\R_+ \times \Omega)$ is analytic. 
Next, note that for $p>1/2+n/4$ the embedding 
\begin{align*}
\E_u &\hookrightarrow W_p^1(\R_+; W_p^4(\Om)) \cap W_p^2(\R_+\times \Om)  \\
&\hookrightarrow H_p^{1+1/p+\varepsilon}(\R_+;H_p^{4-2/p-2\varepsilon}(\Om))\hookrightarrow BUC^1(\R_+; BUC(\Omega))
\end{align*}
holds. In particular, it holds if $\varepsilon>0$ is sufficiently small and $4-2/p-2\varepsilon-n/p>0$. Such an $\varepsilon>0$ exists if $p>(n+2)/4$.
Moreover, on the strength of the mixed derivative embedding theorem 
and the Sobolev embedding theorem we conclude similar as in the proof of Lemma 6 in \cite{MeWi13} that
\begin{align*}
\E_u 	& \hookrightarrow W_p^2(\R_+; W_p^2(\Omega)) \cap W_p^3(\R_+; L_p(\Omega))\\
		& \hookrightarrow H_p^{2+\Theta - \varepsilon/2}(\R_+; H_p^{2-2\Theta + \varepsilon}(\Omega)) 	
												&&\text{for } \Theta- \tfrac{\varepsilon}{2} \in [0,1] \text{ and }\varepsilon >0,\\
		& \hookrightarrow W_p^{2+\Theta - \varepsilon}(\R_+; W_p^{2-2\Theta}(\Omega))					
												&&\text{for } \varepsilon >0, \\
		& \hookrightarrow W_{2p}^2(\R_+; W_p^{2-2\Theta}(\Omega))											
												&&\text{for } \Theta \geq \tfrac{1}{2p} + \varepsilon, \\
		& \hookrightarrow W_{2p}^2(\R_+; L_{2p}(\Omega))														
												&&\text{for } \Theta \leq 1- \tfrac{n}{4p},
\end{align*}
provided $\varepsilon >0$ is sufficiently small and $p>1/2+n/4$. 
Furthermore, we observe that $\Exp^{-2\omega} L_p(\R_+ \times \Omega) \hookrightarrow \Exp^{-\omega} L_p(\R_+ \times \Omega)$ 
since $\Exp^{\omega t} \leq \Exp^{2 \omega t}$ for $\omega \geq 0$. 
Prepared like that, we estimate
\begin{align*}
\| f_t g_t \|_{L_p} 			& \leq \|f_t\|_{L_{2p}} \|f_t\|_{L_{2p}} \lesssim \|f\|_{\E_u} \|g\|_{\E_u}, \\
\| (f_t g_t)_t \|_{L_p} 		& \leq \|f_{tt}\|_{L_p} \|g_t \|_{L_\infty} + \|f_{t}\|_{L_\infty} \|g_{tt} \|_{L_p} \lesssim \|f\|_{\E_u} \|g\|_{\E_u}, \\
\| (f_t g_t)_{tt} \|_{L_p} 	& \leq \|f_{ttt}\|_{L_p} \|g_t \|_{L_\infty} + 2 \|f_{tt}\|_{L_{2p}} \|g_{tt} \|_{L_{2p}} + \|f_{t}\|_{L_\infty} \|g_{ttt} \|_{L_p} \lesssim \|f\|_{\E_u} \|g\|_{\E_u},
\end{align*}
and conclude that $(f,g) \mapsto f_t g_t \colon \E_u \times \E_u \rightarrow W_p^2(\R_+; L_p(\Omega))$ is bilinear
and bounded, thus analytic. Setting $w=u_\bullet+ u_\star$ in 
\begin{equation}
\label{w:exp}
\Exp^{2\omega t} ((w_t)^2)_{tt}= \tfrac{1}{2} ((\Exp^{\omega t}w_t)^2)_{tt} - 3 \omega ((\Exp^{\omega t} w_t)^2)_t + 6 \omega^2 (\Exp^{\omega t} w_t)^2
\end{equation}
and choosing $f=\Exp^{\omega t} u_\bullet$ and $g=\Exp^{\omega t} u_\star$ proves that
\begin{equation}
\label{firstterm:analytic}
(u_\bullet, u_\star) \mapsto (( (u_\bullet + u_\star)_t)^2)_{tt} \colon \Exp^{-\omega} \E_{u,h} \times \Exp^{-\omega} \E_u \rightarrow \Exp^{-\omega}L_p(\R_+ \times \Omega)
\end{equation}
is analytic. It remains to show analyticity of the map
\begin{equation*}
(u_\bullet, u_\star) \mapsto (|\nabla(u_\bullet+ u_\star)|^2)_{tt}\colon 
\Exp^{-\omega} \E_{u,h} \times \Exp^{-\omega} \E_u \rightarrow \Exp^{-\omega} L_p(\R_+ \times \Omega).
\end{equation*} 
Note that we have the embeddings  
\begin{align*}
\E_u & \hookrightarrow W_p^1(\R_+; W_p^4(\Omega)) \hookrightarrow BUC(\R_+; W_p^4(\Omega)) \hookrightarrow BUC(\R_+; BUC^1(\Omega)), && p>n/3\\
\E_u &\hookrightarrow W_p^1(\R_+;W_p^4(\Om)) \cap W_p^2(\R_+ \times \Om) \hookrightarrow  H_p^{1+1/2p+\varepsilon}(\R_+;H_p^{4-1/p-2\varepsilon}(\Om))&&\\
	& \hookrightarrow W_{2p}^1(\R_+;H_p^{4-1/p-2\varepsilon}(\Om)) \hookrightarrow W_{2p}^1(\R_+ \times \Om), &&p>n/6+1/3
\end{align*}
Therewith, we obtain the estimates
\begin{align*}
\|\nabla f \cdot \nabla g\|_{L_p} & \leq \| \nabla f \|_{L_{2p}} \| g \|_{L_{2p}} \lesssim \|f\|_{\E_u} \|g\|_{\E_u}, \\
\|(\nabla f \cdot \nabla g)_t\|_{L_p} & \leq \| (\nabla f)_t \|_{L_{2p}} \|\nabla g \|_{L_{2p}} +\| \nabla f \|_{L_{2p}} \| (\nabla g)_t \|_{L_{2p}} \lesssim \|f\|_{\E_u} \|g\|_{\E_u}, \\
\|(\nabla f \cdot \nabla g)_{tt}\|_{L_p} & \leq \|  (\nabla f)_{tt} \|_{L_p} \| \nabla g \|_{L_\infty} + 2 \| (\nabla f)_t \|_{L_{2p}} \| (\nabla g)_t \|_{L_{2p}} + \| \nabla f \|_{L_\infty} \| (\nabla g)_{tt} \|_{L_p}\\
									&\lesssim \|f\|_{\E_u} \|g\|_{\E_u}, 
\end{align*}
and conclude that $(f,g) \mapsto \nabla f \cdot \nabla g\colon  \E_u \times \E_u \rightarrow W_p^2(\R_+; L_p(\Omega))$ is bilinear and bounded, thus analytic. 
Moreover, we have
\begin{equation*}
\Exp^{2\omega t} ((\nabla w)^2)_{tt} = ((\Exp^{\omega t} \nabla w)^2)_{tt} - 4 \omega ((\Exp^{\omega t} \nabla w)^2)_{t} + 4 \omega^2 (\Exp^{\omega t} \nabla w)^2.
\end{equation*}
By setting $w=u_\bullet + u_\star$, $f=\Exp^{\omega t} u_\bullet$ and $g=\Exp^{\omega t} u_\star$ we are done.
Altogether, we have that $G\colon \Exp^{-\omega} \E_{u,h} \times \Exp^{-\omega} \E_u  \rightarrow \Exp^{-\omega}L_p(\R_+ \times \Omega)$ is analytic.
\vspace{2mm}\newline
\textit{Step 1(b): $D_{u_\bullet} G(0,0)\colon \Exp^{-\omega} L_p(\R_+ \times \Omega) \rightarrow \Exp^{-\omega} \E_{u,h}$ is an isomorphism.}
The Fr\'echet derivative of $G$ with respect to $u_\bullet$ at $(0,0)$ is given by
\begin{equation*}
D_{u_\bullet} G(0,0)[\overline{u}]= (a \Delta - \partial_t)(\overline{u}_{tt} - c^2 \Delta \overline{u} - b \Delta \overline{u}_t).
\end{equation*}
The map $D_{u_\bullet} G(0,0)\colon \Exp^{-\omega} L_p(\R_+ \times \Omega) \rightarrow \Exp^{-\omega} \E_{u,h}$ is an isomorphism since,
according to Corollary \ref{cor:Dirichlet:hom}, for every $\overline{f} \in \Exp^{-\omega} L_p(\R_+ \times \Omega)$ the equation 
$(a \Delta - \partial_t)(\overline{u}_{tt} - c^2 \Delta \overline{u} - b \Delta \overline{u}_t)=\overline{f}$ admits a unique solution 
$\overline{u} \in \Exp^{-\omega} \E_{u,h}$. 
\vspace{2mm}\newline
\textit{Step 2: Construction of the solution.}
On the strength of the Implicit Function Theorem there exists a ball $B_\rho(0) \subset \Exp^{-\omega} \E_u$ with sufficiently 
small radius $\rho >0$ and an analytic map $\varphi\colon B_\rho(0) \subset \Exp^{-\omega} \E_u \rightarrow \Exp^{-\omega} \E_{u,h}$, 
$u_\star \mapsto u_\bullet= \varphi(u_\star)$ satisfying $\varphi(0)=0$ and $G(\varphi(u_\star), u_\star) =0$ for all $u_\star \in B_\rho(0)$. 
Hence, whenever $u_\star$ satisfies the boundary conditions $u_\star |_\Gamma = g$, $\Delta u_\star |_\Gamma =h$ and initial 
conditions $u_\star |_{t=0} = u_0$, $u_{\star,t} |_{t=0} = u_1$, $u_{\star,tt} |_{t=0} = u_2$ which is the case if we define 
$u_\star \in \Exp^{-\omega}\E_u$ to be the unique solution of \eqref{IBVP:linear:Dirichlet} with $(f=0, u_0, u_1, u_2, g,h)$, then 
$u_\bullet + u_\star = \varphi(u_\star) + u_\star$ solves \eqref{IBVP:Dirichlet}. 
\vspace{2mm}\newline
\textit{Step 3: Dependence of the solution on the data.} 
It remains to show that the solution $u \in \Exp^{-\omega} \E_u$ depends analytically on $(g,h,u_0, u_1, u_2)$. To this end, 
we define the spaces
\begin{align*}
\overline{\mathbb{D}} &:= \Exp^{-\omega} \E_g \times \Exp^{-\omega} \E_h \times W_p^4(\Omega) \times W_p^{4-2/p}(\Omega) \times W_p^{2-2/p}(\Omega), \\
\mathbb{D}&:= \{(g,h,u_0, u_1, u_2)\in \overline{\mathbb{D}}\colon u_0 |_\Gamma = g|_{t=0}, u_1 |_\Gamma = g_t |_{t=0}, u_2 |_\Gamma = g_{tt} |_{t=0} \text{ if }p>3/2, \\
	 		& \hspace{170pt}  \Delta u_0 |_\Gamma = h|_{t=0}, \Delta u_1 |_\Gamma = h_t |_{t=0} \text{ if }p>3/2\}.
\end{align*}
From Proposition \ref{prop:linearinhom} with $f=0$ we obtain that $u_\star$ depends linearly and continuously and 
thus analytically on $(g,h,u_0, u_1, u_2) \in \mathbb{D}$.
Moreover, $u_\star \mapsto u_\bullet = \varphi(u_\star)$ is analytic on $B_\rho(0)$ and therefore $u_\bullet \in \Exp^{-\om}\E_{u,h}$ depends analytically
on the data $(g,h,u_0, u_1, u_2) \in \mathbb{D}$. Altogether, $u= u_\bullet + u_\star$ enjoys the same property which concludes the proof.
\end{proof}
An immediate consequence of Theorem \ref{thm:global:Dirichlet} is that the global solution $u \in \Exp^{-\om} \E_u$ of \eqref{IBVP:Dirichlet}
decays to zero at an exponential rate.
\begin{theorem}[Exponential stability - the Dirichlet case]
\label{thm:decay:Dirichlet}
Under the same assumptions as in Theorem \ref{thm:global:Dirichlet}, the solution $u$ decays exponentially fast to zero as $t \rightarrow \infty$, 
in the sense that
\begin{equation*}
\|u(t)\|_{W_p^4} + \|u_t(t)\|_{W_p^{4-2/p}} + \|u_{tt}\|_{W_p^{2-2/p}} \leq C \Exp^{-\omega t}, \qquad t \geq 0,
\end{equation*}
for some $C \geq 0$ depending on the boundary and initial data $g$, $h$, $u_0$, $u_1$ and $u_2$.
\end{theorem}
\begin{proof} 
We have $u \in \Exp^{-\omega} W_p^1(\R_+; W_p^4(\Omega)) \hookrightarrow \Exp^{-\omega}BUC(\R_+; W_p^4(\Omega))$, hence
\begin{equation*}
u \in BUC(\R_+, W_p^4(\Omega)), \quad \|u(t)\|_{W_p^4} \leq C_1\, \Exp^{-\omega t} \text{ with } C_1 = \|\Exp^{\omega\cdot} u\|_{BUC(\R_+;W_p^4)}.
\end{equation*}
Furthermore, $\nabla_x^j u_t \in \Hb(\R_+) \hookrightarrow BUC(\R_+; W_p^{2-2/p}(\Omega))$ for $j \in \{0,1,2\}$. Therefore, we obtain
\begin{equation*}
u_t \in BUC(\R_+, W_p^{4-2/p}(\Omega)), \quad \|u_t(t)\|_{W_p^{4-2/p}} \leq C_2\, \Exp^{-\omega t} 
\text{ with }C_2 =\|\Exp^{\omega \cdot} u_t\|_{BUC(\R_+;W_p^{4-2/p})}.
\end{equation*}
Finally, from $u_{tt} \in \Hb(\R_+)$ we deduce that 
\begin{equation*}
u_{tt} \in BUC(\R_+, W_p^{2-2/p}(\Omega)), \quad \|u_{tt}(t)\|_{W_p^{2-2/p}} \leq C_3\, \Exp^{-\omega t} 
\text{ with }C_3 = \|\Exp^{\omega\cdot}u_{tt}\|_{BUC(\R_+;W_p^{2-2/p})}
\end{equation*}
and the claim follows. 
\end{proof}

\section{The Neumann boundary value problem}
\label{sec:Neumann}
In this section we treat the inhomogeneous Neumann boundary value problem \eqref{IBVP:Neumann}. We proceed analogously 
to the Dirichlet case, that is, we first consider the linearized equation and then construct a solution of the nonlinear problem
\eqref{IBVP:Neumann} by means of the implicit function theorem. 

Note that, in the Dirichlet case, the fact that the operator $-A^D\colon \D(A^D) \rightarrow X^D$ defined by
\eqref{operator:A} has a strictly negative spectral bound (Lemma \ref{lem:omega}) was crucial in order to show
that the linearized equation \eqref{IBVP:linear:Dirichlet} admits maximal regularity on $\R_+$, see the proof of Theorem \ref{thm:maximal:Dirichlet2}.
In the Neumann case, due to the zero eigenvalue of $-\Delta_N\colon \D(\Delta_N) \rightarrow L_p(\Om)$ with 
$\D(\Delta_N) = \{u \in W_p^2(\Om)\colon \del_\nu u =0 \text{ on }\Gamma\}$, we cannot expect to obtain maximal regularity on $\R_+$.
For this reason we consider $-\Delta_{N,0}\colon \D(\Delta_{N,0}) \rightarrow L_{p,0}(\Om)$, where $\D(\Delta_{N,0}) = \D(\Delta_N) \cap L_{p,0}(\Om)$. 
The spectrum of $-\Delta_{N,0}$ is contained in $(0,\infty)$, therefore we can prove maximal regularity of the homogeneous linear Neumann boundary problem 
on $\R_+$ analogously to the Dirichlet case.
However, if we restrict ourselves to finite time intervals $J=(0,T)$, then we do not necessarily need to use the realization $-\Delta_{N,0}$ in $L_{p,0}(\Om)$.
In case of finite time intervals we use $-\Delta_N$. 
As a consequence, we will prove global well-posedness of \eqref{IBVP:Neumann} only if the data $u_0$, $u_1$, $u_2$ and $g$, $h$ have zero mean 
whereas local well-posedness holds also for data with non-zero mean. 

\subsection{Maximal \texorpdfstring{$L_p$}{Lp}-regularity for the linearized equation}
\label{subsec:linear:Neumann}
As in Section \ref{subsec:linear:Dirichlet}, let $J=(0,T)$ or $J=\R_+$ and assume $p\in(1,\infty)$.  Here, for we $f\in L_p(J\times\Omega)$ we consider 
\begin{equation}
\label{IBVP:linear:Neumann}
\begin{cases}
\begin{aligned}
(a\Delta- \partial_t)(u_{tt}-b\Delta u_t - c^2 \Delta u)&= f 			&& \text{in } J \times \Omega,\\
(\partial_\nu u, \partial_\nu \Delta u)&= (g,h)					&& \text{on } J \times \Gamma,\\
(u,u_t,u_{tt})&=(u_0, u_1, u_2)								&& \text{on }\{t=0\} \times \Omega,
\end{aligned}
\end{cases}
\end{equation}
where $u_0, u_1, u_2\colon \Omega \rightarrow \R$ and $g,h\colon J \times \Gamma \rightarrow \R$ 
are the given initial and boundary data, respectively. 
Analogously to the Dirichlet case we first represent \eqref{IBVP:linear:Neumann} with $g=h=0$ as an abstract evolution equation of the form
\begin{equation}
\label{eq:parabolic:Neumann}
\del_t v^N + A^N v^N = F, \qquad v^N(0)=v_0^N
\end{equation}
by setting 
\begin{equation*}
v^N= \begin{pmatrix} u \\ u_t \\ u_{tt} - c^2 \Delta_N u - b\Delta_N u_t \end{pmatrix}, \quad 
v_0^N= \begin{pmatrix} u_0 \\ u_1 \\ u_2 - c^2 \Delta_N u_0 - b\Delta_N u_1 \end{pmatrix} \quad \text{and} \quad
F= \begin{pmatrix} 0 \\ 0 \\ -f \end{pmatrix},
\end{equation*}
introducing the Banach space
\begin{equation*}
X^N = \D((\Delta_N)^2) \times \D(\Delta_N) \times L_{p}(\Omega)
\end{equation*} 
and defining the coefficient operator $A^N: \D(A^N) \rightarrow X^N$ via
\begin{equation}
\label{coeff:Neumann}
A^N= \begin{pmatrix} 0 & -I & 0 \\ -c^2\Delta_N & -b\Delta_N & -I \\ 0 & 0 & -a\Delta_N \end{pmatrix}, \qquad 
\D(A^N)=\D((\Delta_N)^2) \times \D((\Delta_N)^2) \times \D(\Delta_N).
\end{equation}
On one hand, in the following we will show maximal regularity of $A^N$ on finite time intervals. On the other hand, as already pointed out,
we are going to use the realization $-\Delta_{N,0}$ of the homogeneous Neumann Laplacian. For this reason we introduce the Banach space 
\begin{equation*}
X^{N,0} = \D((\Delta_{N,0})^2) \times \D(\Delta_{N,0}) \times L_{p,0}(\Omega),
\end{equation*}  
and the densely defined operator $A^{N,0}: \D(A^{N,0}) \rightarrow X^{N,0}$, where $\Delta_N$ has to be replaced by $\Delta_{N,0}$ in \eqref{coeff:Neumann}.

\begin{proposition}
\label{thm:maximal:Neumann1}
Let $p\in(1,\infty)$. There exists some $\nu>0$ such that the operators $\nu+ A^N$ and $\nu+A^{N,0}$ admit maximal regularity on $\R_+$. 
\end{proposition}
\begin{proof}
The result can be proved similarly to Proposition \ref{thm:maximal:Dirichlet}. For some $\alpha>0$ consider 
\begin{equation*}
A_1^N= \begin{pmatrix} \alpha I & -I & 0 \\ 0 & \alpha I -b\Delta_N & -I \\ 0 & 0 & \alpha I -a\Delta_N \end{pmatrix}
\qquad \text{and} \qquad
A_2^N= \begin{pmatrix} -\alpha I & 0 & 0 \\ -c^2\Delta_N & - \alpha I & 0 \\ 0 & 0 & -\alpha I \end{pmatrix},
\end{equation*}
Clearly, the operator $A_2^N\colon X^N \rightarrow X^N$ is bounded. Moreover, $A_1^N\colon \D(A_1^N) \rightarrow X^N$ has maximal $L_p$-regularity on $\R_+$ which is seen as in the proof of Proposition \ref{thm:maximal:Dirichlet} by considering
$v_t+A_1^N v=F$, $v_0=0$ for $v=(v_1, v_2, v_3)^\top$ and $F=(f_1, f_2, f_3)^\top$. Explicitly, we have 
  \begin{align*}
    \partial_t v_1 + \alpha v_1 - v_2 &= f_1, 								&& v_1(0)=0,\\
    \partial_t v_2 + \alpha v_2 - b\Delta_N v_2 - v_3&= f_2, 	&& v_2(0)=0,\\
    \partial_t v_3 + \alpha v_3 - a \Delta_N v_3 &= f_3, 			&& v_3(0)=0.
  \end{align*}
Let $F \in L_p(\R_+; X^N)$. Now one solves stepwise the equations above, starting with the last one, to get a 
unique solution $v_3 \in W_p^1(\R_+; L_p(\Om)) \cap L_p(\R_+; \D(\Delta_N))$.
Then $f_2+v_3 \in L_p(\R_+; \D(\Delta_N))$. Here, we need to employ Lemma \ref{lem:interior_higher_regularity} 
in order to obtain a unique solution $v_2 \in W_p^1(\R_+; \D(\Delta_N)) \cap L_p(\R_+;\D( (\Delta_N)^2))$.  
As in the Dirichlet case, the first equation gives us a unique solution $v_1 \in W_p^1(\R_+; \D((\Delta_N)^2))$.
Altogether, since the condition $F \in L_p(\R_+;X^N)$ implies existence of a unique solution $v \in W_p^1(\R_+; X^N)\cap L_p(\R_+; \D(A^N))$ we conclude that $A_1^N: \D(A^N) \rightarrow X^N$ admits maximal $L_p$-regularity on $\R_+$.
Finally, as in the proof of Proposition \ref{thm:maximal:Dirichlet}, a perturbation argument implies that there exists some $\nu > 0$ 
such that $\nu+A_1^N+A_2^N = \nu+A^N$ has the property of maximal $L_p$-regularity on $\R_+$.

Maximal $L_p$-regularity of $\nu+A^{N,0}$ follows analogously by considering the operators $A_1^{N,0}$ and $A_2^{N,0}$ which are equal to $A_1^N$ and $A_2^N$ upon replacement of $\Delta_N$ by $\Delta_{N,0}$ and proceeding as above.
\end{proof}
Since $\nu + A^N$ has maximal $L_p$-regularity on $\R_+$, multiplication of \eqref{eq:parabolic:Neumann} with $\Exp^{-\nu t}$ 
shows that $A^N$ has maximal $L_p$-regularity on bounded intervals $J=(0,T)$. 
\begin{theorem}
\label{thm:maximal:Neumann:finite}
Suppose $J=(0,T)$ is a finite time interval and let $p\in (1,\infty)$. 
Then $A^N: \D(A^N) \rightarrow X^N$ has maximal $L_p$-regularity on $J$ and therefore
\begin{equation*}
(\partial_t +A^N, \gamma_t)\colon 
W_p^1(J; X^N) \cap L_p(J;\D(A^N)) \rightarrow  L_p(J;X^N) \times (X^N, \D(A^N))_{1-1/p,p}
\end{equation*}
is an isomorphism.
\end{theorem}

Next, observe that $-A^{N,0}$ has a strictly negative spectral bound. This can be shown likewise to Lemma \ref{lem:omega}
since the spectrum of the negative Neumann Laplacian in $L_{p,0}(\Omega)$ is contained in $(0,\infty)$ and consists only of eigenvalues
of finite multiplicity. 
\begin{lemma}
\label{lem:omega:Neumann}
The spectral bound of $-A^{N,0}$ is given by $s(-A^{N,0})=-\omega_0^N$, where $\omega_0^N=\min \{a\lam_1^N, b\lam_1^N/2,c^2/b\}$. 
Here, $\lambda_1^N$ denotes the smallest non-zero eigenvalue of $-\Delta_{N,0}$. 
In particular, if $\mathrm{Re}(\lambda) < \omega_0^N$, then $\lambda \in \rho(A^{N,0})$.
\end{lemma}

By means of Lemma \ref{lem:omega:Neumann} one shows that $A^{N,0}$ has maximal $L_p$-regularity on $\R_+$. For details we refer to the proof of Theorem \ref{thm:maximal:Dirichlet2}.
\begin{theorem} 
\label{thm:maximal:Neumann2}
Let $p\in (1, \infty)$ and $\omega \in [0,\omega_0^N)$. 
Then $A^{N,0}: \D(A^{N,0}) \rightarrow X^{N,0}$ has maximal $L_p$-regularity on $\R_+$ and therefore
\begin{align*}
(\partial_t +A^{N,0}, \gamma_t)\colon &
\Exp^{-\omega} (W_p^1(J; X^{N,0}) \cap L_p(J;\D(A^{N,0})))\\
&\qquad  \rightarrow \Exp^{-\omega} L_p(J;X^{N,0}) \times (X^{N,0}, \D(A^{N,0}))_{1-1/p,p}
\end{align*}
is an isomorphism.
\end{theorem}

Theorems \ref{thm:maximal:Neumann:finite} and \ref{thm:maximal:Neumann2} immediately yield optimal regularity for 
\eqref{IBVP:linear:Neumann} with homogeneous boundary conditions, i.~e.\ $g=h=0$.

\begin{corollary} Let $p\in (1,\infty)\setminus\{3\}$ and consider the homogeneous Neumann boundary value problem
\label{cor:linearhom:Neumann}
\begin{equation}
\label{IBVP:linearhom:Neumann}
\begin{cases}
\begin{aligned}
(a\Delta- \partial_t)(u_{tt}-b\Delta u_t - c^2 \Delta u)&= f 			&& \text{in } J \times \Omega,\\
(\partial_\nu u, \partial_\nu \Delta u)&= (0,0)										&& \text{on } J \times \Gamma,\\
(u,u_t,u_{tt})&=(u_0, u_1, u_2)																		&& \text{on }\{t=0\} \times \Omega.
\end{aligned}
\end{cases}
\end{equation}
\begin{enumerate}[label=\emph{(\roman*)}, leftmargin=*]
\item If $J=(0,T)$ is finite, then \eqref{IBVP:linearhom:Neumann} admits optimal regularity in the sense that there exists a unique solution 
\begin{equation*}
u \in \E_{u}(J), \qquad \E_u(J) = W_p^3(J; L_{p}(\Omega))\cap W_p^1(J;W_{p}^4(\Omega))
\end{equation*} 
if and only if $f \in L_p(J\times \Omega)$, $u_0 \in W_p^4(\Omega)$, $u_1 \in W_{p}^{4-2/p}(\Omega)$, $u_2 \in W_{p}^{2-2/p}(\Omega)$
and the initial and boundary data are compatible, that is, we have $\partial_\nu u_0 |_\Gamma = \partial_\nu \Delta u_0 |_\Gamma = \partial_\nu u_1 |_\Gamma = 0$ and, if $p>3$, also $\partial_\nu  \Delta u_1 |_\Gamma = \partial_\nu u_2 |_\Gamma = 0$ in the sense of traces.
\item If $J=\R_+$, then for every $\omega \in (0, \omega_0^N)$ with $\omega_0^N= \min \{a\lam_1^N, b\lam_1^N/2, c^2/b\}$ we have that 
\eqref{IBVP:linearhom:Neumann} admits optimal regularity in the sense that there exists a unique solution 
\begin{equation*}
u \in \Exp^{-\om} \E_{u,0}, \qquad \E_{u,0}= W_p^3(\R_+; L_{p,0}(\Omega)) \cap W_p^1(\R_+;W_{p}^4(\Omega)\cap L_{p,0}(\Om))
\end{equation*} 
if and only if $f \in \Exp^{-\omega} L_p(\R_+;L_{p,0}(\Om))$, $u_0 \in W_p^4(\Omega)\cap L_{p,0}(\Om)$, $u_1 \in W_{p}^{4-2/p}(\Omega)\cap L_{p,0}(\Om)$, 
$u_2 \in W_{p}^{2-2/p}(\Omega)\cap L_{p,0}(\Om)$ and the initial and boundary data are compatible.
\end{enumerate}
\end{corollary}
\begin{proof}
Assertion (i) follows immediately from Theorem \ref{thm:maximal:Neumann:finite}. Analogously to the proof of Corollary \ref{cor:Dirichlet:hom} one verifies that 
$v^N \in W_p^1(J; X^N) \cap L_p(J;\D(A^N))$, $F \in L_p(J;X^N)$ and $v_0^{N} \in (X^N, \D(A^N))_{1-1/p,p}$ imply $u \in \E_u(J)$, $f \in L_p(J \times \Om)$ 
and the desired regularity of the initial values, respectively. Based on Theorem \ref{thm:maximal:Neumann2}, the second claim follows analogously.
\end{proof}

Finally we arrive at our global optimal regularity result for \eqref{IBVP:linear:Neumann}. As in the Dirichlet case, sufficiency is shown
by a combination of an optimal regularity result for the linearized Westervelt equation and a higher regularity result for the heat equation.
\begin{proposition}
\label{prop:linearinhom:Neumann}
Let $p\in(1,\infty) \setminus \{3\}$ and define $\omega_0 = \min \{a\lam_1^N, b\lam_1^N/2, c^2/b\}$. 
Then for every $\omega \in (0,\omega_0)$ the linear initial boundary value problem
\begin{equation}
\label{IBVP:Neumann:linearinhom}
\begin{cases}
\begin{aligned}
(a\Delta- \partial_t)(u_{tt}-b\Delta u_t - c^2 \Delta u)&= f 			&& \text{in } \R_+ \times \Omega,\\
(\del_\nu u, \del_\nu \Delta u)&= (g,h)						&& \text{on } \R_+ \times \Gamma,\\
(u,u_t,u_{tt})&=(u_0, u_1, u_2)									&& \text{on }\{t=0\} \times \Omega.
\end{aligned}
\end{cases}
\end{equation}
has a unique solution of the form $u(t,x) = v(t,x) + w(t)$, where 
\begin{equation*}
v \in \Exp^{-\omega}\E_{u,0}, \quad \E_{u,0}= W_p^3(\R_+; L_{p,0}(\Omega)) \cap W_p^1(\R_+;W_p^4(\Omega)\cap L_{p,0}(\Om)), \quad \del_t ^3w \in \Exp^{-\om} L_p(\R_+)
\end{equation*}
if and only if the data satisfy the conditions
\begin{enumerate}[label=\emph{(\roman*)}, leftmargin=*]
\item $f \in \Exp^{-\omega}L_p(\R_+ \times \Omega)$,
\item $u_0 \in W_p^4(\Omega)$, $u_1 \in W_p^{4-2/p}(\Omega)$, $u_2 \in W_p^{2-2/p}(\Omega)$,
\item $g \in \Exp^{-\omega} \F_{g,\nu}$, $\F_{g,\nu} = W_p^{5/2-1/2p}(\R_+, L_p(\Gamma)) \cap W_p^1(\R_+, W_p^{3-1/p}(\Gamma))$,\\
$h \in \Exp^{-\omega} \F_{h,\nu}$,  $\F_{h,\nu}= W_p^{3/2-1/2p}(\R_+; L_p(\Gamma)) \cap W_p^{1}(\R_+;W_p^{1-1/p}(\Gamma))$,
\item $\del_\nu u_0 |_\Gamma = g|_{t=0}$, $\del_\nu \Delta u_0 |_\Gamma = h|_{t=0}$, $\del_\nu u_1 |_\Gamma = g_t |_{t=0}$ and, if $p>3$, also $\del_\nu \Delta u_1 |_\Gamma = h_t |_{t=0}$, $\del_\nu u_2 |_\Gamma = g_{tt} |_{t=0}$ in the sense of traces.
\end{enumerate}
Moreover, the solution fulfills the estimate
\begin{equation*}
\|u\|_{\Exp^{-\omega}\E_u} \lesssim \|f\|_{\Exp^{-\omega} L_p} + \|g\|_{\Exp^{-\omega} \F_{g,\nu}} 
+  \|h\|_{\Exp^{-\omega} \F_{h,\nu}} + \|u_0\|_{W_p^4}+\|u_1\|_{W_p^{4-2/p}}+\|u_2\|_{W_p^{2-2/p}}.
\end{equation*}
\end{proposition}
\begin{proof}
It is not surprising that the proof of necessity can be done similarly to Proposition \ref{prop:linearinhom}.  
Assume that $u(t,x) = v(t,x) + w(t,x)$ is a solution of \eqref{IBVP:Neumann:linearinhom} with $v \in \Exp^{-\omega} \E_u$ and 
$\del_t^3 w \in \Exp^{-\om} L_p(\R_+)$. Since, apart from having zero mean, $v$ has the same regularity as $u$ in Proposition \ref{prop:linearinhom}, 
we are conclude that 
$\Exp^{\om t}f= - \Exp^{\om t} (v+w)_{ttt} -ab \Delta^2 (\Exp^{\om t}v_t) -ac^2 \Delta (\Exp^{\om t}v) + (a+b) \Delta (\Exp^{\om t}v_{tt}) + c^2 \Delta^2 (\Exp^{\om t} v_t) \in L_p(\R_+ \times \Omega)$ and (i) is readily checked. Moreover, $w$, $w_t$ and $w_{tt}$ are just time-dependent and thus constant at $t=0$, hence
the regularity of the initial values (ii) can be shown as in the Dirichlet case. Moreover, the regularity of the boundary data (iii) is obtained from the spatial trace 
\eqref{trace:spatial} with the same choices of $k$ and $l$ as in the proof of Proposition \ref{prop:linearinhom}
and setting $j_B=1$. Concerning (iv) it is straightforward to show 
\begin{align*}
&\del_\nu u_0 |_\Gamma = g |_{t=0} \text{ in } W_p^{3-1/p}(\Gamma), 		&&\del_\nu \Delta u_0 |_\Gamma = h |_{t=0} \text{ in } W_p^{1-1/p}(\Gamma),\\
&\del_\nu u_1|_\Gamma = g_t|_{t=0} \text{ in } B_{p,p}^{3-3/p}(\Gamma),		&&\del_\nu \Delta u_1 |_\Gamma = h_t |_{t=0} \text{ in } W_p^{1-3/p}(\Gamma) \text{ if }p>3,\\
&\partial_\nu u_2|_\Gamma = g_{tt} |_{t=0} \text{ in } W_p^{1-3/p}(\Gamma) \text{ if }p>3.
\end{align*}

Next, we show that conditions (i)--(iv) are sufficient for the existence of a unique solution $u(t,x)=v(t,x) +w(t)$ of \eqref{IBVP:Neumann:linearinhom}
such that $v\in \Exp^{-\omega} \E_u$ and $w_{ttt} \in \Exp^{-\om} L_p(\R_+)$. 
As in the Dirichlet case, we interchange the order of differentiation on the left-hand side and consider the subproblems 
\begin{equation}
\label{Westervelt:inhom:Neumann}
\begin{cases}
\begin{aligned}
\varphi_{tt}-b\Delta \varphi_t - c^2 \Delta \varphi &= f 			&& \text{in } \R_+ \times \Omega,\\
\del_\nu \varphi &= a h - g_t							&& \text{on } \R_+ \times \Gamma,\\
(\varphi,\varphi_t)&=(a\Delta u_0-u_1,a \Delta u_1- u_2)		&& \text{on }\{t=0\} \times \Omega,
\end{aligned}
\end{cases}
\end{equation}
and
\begin{equation}
\label{heat:inhom:Neumann}
\begin{cases}
\begin{aligned}
a\Delta u-u_t &= \varphi						&& \text{in } \R_+ \times \Omega,\\
\del_\nu u&=g								&&\text{on } \R_+ \times \Gamma, \\
u&=u_0									&& \text{on }\{t=0\} \times \Omega,
\end{aligned}
\end{cases}
\end{equation}
From condition (ii) we obtain $a \Delta u_0 - u_1 \in W_p^2(\Omega)$ and $a\Delta u_1- u_2 \in W_p^{2-2/p}(\Omega)$. 
Furthermore, (iii) implies $ah - g_t \in \Exp^{-\omega} \W_\nu$. On the strength of Lemma \ref{lem:Westervelt:inhom:Neumann} 
we obtain that \eqref{Westervelt:inhom:Neumann} admits a unique solution of the form $\varphi(t,x)=\varphi_1(t,x)+\varphi_2(t)$ with 
$\varphi_1 \in \Exp^{-\omega} \W_{u,0}$ and $\partial_t^2 \varphi_2 \in \Exp^{-\om}L_p(\R_+)$.
We now make the ansatz $u(x,t)= v(x,t) + w(t)$ such that $\bar{v}(\cdot, t)=0$. 
Applying $\abs\Omega^{-1}\int_\Omega$ to $a\Delta u-u_t=\varphi$ we deduce that $w$ solves the ordinary 
differential equation $w_t = -\varphi_2+a\abs\Gamma\abs\Omega^{-1}\bar g$ with $w(0)=\bar u_0$. Hence $\partial_t^3 w \in \Exp^{-\om}L_p(\R_+)$. 
Moreover, $v$ is a solution of 
\begin{equation}
\label{eq:Neumann:v}
a\Delta v-v_t=\varphi_1+a\abs\Gamma\abs\Omega^{-1}\bar g \text{ in }\Om, \quad \del_\nu v=g \text{ on }\Gamma, \quad v(0)=u_0-\bar u_0. 
\end{equation}
In order to apply Corollary \ref{cor:higher_regularity_for_the_heat_equation_exponential_weight}, we first note that the right-hand side 
$\varphi_1+a\abs\Gamma\abs\Omega^{-1}\bar g$ belongs to $\Exp^{-\om}\W_u$ since $\bar g$ only depends on time and belongs to 
$e^{-\omega}W^2_p(\bbR_+)$.  The rescaled function $v_a(t,x) = av(t/a,x)$ should solve the system
  \begin{align}
\label{eq:Neumann:v_a}
\Delta v_a-\del_t v_a = \varphi_1+a\abs\Gamma\abs\Omega^{-1}\bar g \text{ in }\Om, \quad \del_\nu v_a=ag \text{ on }\Gamma, \quad v_a(0)=au_0-a\bar u_0.     
  \end{align}
Hence the compatibility condition \eqref{eq:150403_1} becomes
\begin{align*}
  -\int_\Omega(\varphi_1+a\abs\Gamma\abs\Omega^{-1}\bar g) \, d x + \int_\Gamma ag \, d S = 0
\end{align*}
and is clearly satisfied.  Therefore Corollary \ref{cor:higher_regularity_for_the_heat_equation_exponential_weight} yields a unique solution 
$v_a\in \Exp^{-\omega}\bbE_{u,0}$ of problem \eqref{eq:Neumann:v_a} and thus $u=v+w$ solves problem \eqref{IBVP:Neumann:linearinhom}. 

Finally, uniqueness follows by considering two solutions of \eqref{IBVP:Neumann:linearinhom}, the difference $\overline{u}$ of which solves 
\eqref{IBVP:linearhom:Neumann} with $f=0$ and 
$u_0=u_1=u_2=0$, hence $\overline{u}=0$ and the proof is complete.
\end{proof}

\begin{proposition}
\label{prop:Neumann:inhom:finite}
Let $p\in(1,\infty) \setminus \{3\}$ and let $J=(0,T)$ be a finite interval. Then the linear initial boundary value problem \eqref{IBVP:linear:Neumann} has a unique solution  
\begin{equation*}
u \in  W_p^3(J; L_{p}(\Omega)) \cap W_p^1(J;W_p^4(\Om))
\end{equation*}
if and only if conditions \emph{(i)}--\emph{(iv)} from Proposition \ref{prop:linearinhom:Neumann} hold with $\R_+$ replaced by $J$.
\end{proposition}
\begin{proof}
Uniqueness and necessity are shown likewise to Proposition \ref{prop:linearinhom:Neumann}. For the proof of sufficiency one considers 
\eqref{Westervelt:inhom:Neumann} and \eqref{heat:inhom:Neumann} on $(0,T)$ instead of $\R_+$. Remark \ref{rem:Westerveltfinite} and Corollary 
\ref{cor:higher_regularity_for_the_heat_equation_exponential_weight} then imply existence of a unique solution 
$u \in  W_p^3(J; L_{p}(\Omega)) \cap W_p^1(J;W_p^4(\Om))$.
\end{proof}

\subsection{Local well-posedness, global well-posedness and exponential stability} We now arrive at our well-posedness results 
for the Neumann problem \eqref{IBVP:Neumann}. If we allow the (sufficiently small) initial and boundary data to have non-zero mean, 
we are only able to prove well-posedness of \eqref{IBVP:Neumann} on finite time intervals $J=(0,T)$. 
On the other hand, if the (sufficiently small) data have zero mean, that is, $\bar{u}_0 = \bar{u}_1 = \bar{u}_2 =0$ and $\bar{g}=\bar{h}=0$, 
then we obtain a globally well-posed solution which decays exponentially fast to zero.

\begin{theorem}[Local well-posedness]
\label{thm:local:Neumann}
Let $J=(0,T)$ for some $T<\infty$ and $p > \max \{n/4 +1/2, n/3\}$, $p\neq 3$. Suppose 
\begin{equation}
\label{cond:data:Neumann:local}
\begin{aligned}
&u_0 \in W_p^4(\Omega), \qquad u_1 \in W_p^{4-2/p}(\Omega), \qquad u_2 \in W_p^{2-2/p}(\Omega),\\
&g \in \Exp^{-\omega} \F_{g,\nu}(J),~ \F_{g,\nu}(J) = W_p^{5/2-1/2p}(J, L_p(\Gamma)) \cap W_p^1(J, W_p^{3-1/p}(\Gamma)),\\
&h \in \Exp^{-\omega} \F_{h,\nu}(J),~  \F_{h,\nu}(J)= W_p^{3/2-1/2p}(J; L_p(\Gamma)) \cap W_p^{1}(J;W_p^{1-1/p}(\Gamma)),
\end{aligned}
\end{equation}
such that $\del_\nu u_0 |_\Gamma = g|_{t=0}$, $\del_\nu \Delta u_0 |_\Gamma = h|_{t=0}$, $\del_\nu u_1 |_\Gamma = g_t |_{t=0}$ 
and, if $p>3$, also $\del_\nu \Delta u_1 |_\Gamma = h_t |_{t=0}$, $\del_\nu u_2 |_\Gamma = g_{tt} |_{t=0}$ in the sense of traces.

There exists some $\rho>0$ such that if 
\begin{equation*}
\|g\|_{\F_{g,\nu}} + \|h\|_{\F_{h,\nu}} + \|u_0\|_{W_p^4} + \|u_1\|_{W_p^{4-2/p}} + \|u_2\|_{W_p^{2-2/p}} < \rho, 
\end{equation*}
the nonlinear Neumann boundary value problem \eqref{IBVP:Neumann} admits a unique solution 
\begin{equation}
\label{sol:nonlinear:Neumann:local}
u \in \E_u(J)=  W_p^3(J; L_{p}(\Omega)) \cap W_p^1(J; W_{p}^4(\Omega))
\end{equation}
which depends analytically (in particular continuously) on the data \eqref{cond:data:Neumann:local} with respect to the 
corresponding topologies. Moreover, conditions \eqref{cond:data:Neumann:local} are necessary for the regularity of the solution given 
in \eqref{sol:nonlinear:Neumann:local}.
\end{theorem}

\begin{theorem}[Global well-posedness - the Neumann case]
\label{thm:global:Neumann}
Let $p > \max \{n/4 +1/2, n/3\}$, $p\neq 3$ and define $\omega_0 = \min \{a \lam_1^N, b\lam_1^N /2, c^2/b\}$. Suppose
\begin{equation}
\label{cond:data:Neumann}
\begin{aligned}
&u_0 \in W_p^4(\Omega)\cap L_{p,0}(\Om),~u_1 \in W_p^{4-2/p}(\Omega)\cap L_{p,0}(\Om),~ u_2 \in W_p^{2-2/p}(\Omega)\cap L_{p,0}(\Om),\\
&g \in \Exp^{-\omega} \F_{g,\nu,0},~ \F_{g,\nu, 0} = W_p^{5/2-1/2p}(\R_+, L_{p,0}(\Gamma)) \cap W_p^1(\R_+, W_p^{3-1/p}(\Gamma)\cap L_{p,0}(\Gamma)),\\
&h \in \Exp^{-\omega} \F_{h,\nu,0},~  \F_{h,\nu,0}= W_p^{3/2-1/2p}(\R_+; L_{p,0}(\Gamma)) \cap W_p^{1}(\R_+;W_p^{1-1/p}(\Gamma)\cap L_{p,0}(\Gamma)),
\end{aligned}
\end{equation}
such that $\del_\nu u_0 |_\Gamma = g|_{t=0}$, $\del_\nu \Delta u_0 |_\Gamma = h|_{t=0}$, $\del_\nu u_1 |_\Gamma = g_t |_{t=0}$ and, if $p>3$, also $\del_\nu \Delta u_1 |_\Gamma = h_t |_{t=0}$, $\del_\nu u_2 |_\Gamma = g_{tt} |_{t=0}$ in the sense of traces.

Then for every $\omega \in (0, \omega_0)$ there exists some $\rho >0$ such that if
\begin{equation*}
\|g\|_{\Exp^{-\omega}\F_{g,\nu,0}} + \|h\|_{\Exp^{-\omega}\F_{h,\nu,0}} + \|u_0\|_{W_p^4} + \|u_1\|_{W_p^{4-2/p}} + \|u_2\|_{W_p^{2-2/p}} < \rho, 
\end{equation*}
then the nonlinear Neumann boundary value problem \eqref{IBVP:Neumann} admits a unique solution \begin{equation}
\label{sol:nonlinear:Neumann}
u \in \Exp^{-\omega} \E_{u,0}, \quad \E_{u,0} = W_p^3(\R_+; L_{p,0}(\Omega)) \cap W_p^1(\R_+; W_{p}^4(\Omega)\cap L_{p,0}(\Om)),
\end{equation}
which depends analytically on the data \eqref{cond:data:Neumann} with respect to the 
corresponding topologies. Moreover, conditions \eqref{cond:data:Neumann} are necessary for the regularity of the solution given 
in \eqref{sol:nonlinear:Neumann}.
\end{theorem}

\begin{remark}
In case of Theorem \ref{thm:local:Neumann} we define $u_\star$ to be the solution according to Proposition \ref{prop:Neumann:inhom:finite}
which satisfies \eqref{IBVP:linear:Neumann} for the data $(f=0, g, h, u_0, u_1, u_2)$ und suppose $u_\bullet$
satisfies homogeneous boundary and initial conditions. The solution is then of the form $u= u_\star + u_\bullet$ and $u_\bullet$ is found by the implicit
function theorem. The claim then follows likewise to the proof of Theorem \ref{thm:global:Dirichlet}.

However, if we want to prove global well-posedness, we need to use Proposition \ref{prop:linearinhom:Neumann} for the linearized equation,
where for given data $(f=0, g, h, u_0, u_1, u_2)$ according to {(ii)}--{(iv)} the solution is of the form $u(t,x) = v(t,x) + w(t)$,
where $v \in \Exp^{-\om} \E_{u,0}$ has zero mean and $w$ is only time-dependent. 
If $w \neq 0$, the term $((u_t)^2)_{tt}$ in the nonlinear right-hand side of \eqref{IBVP:Neumann} causes problems. 
Recall \eqref{w:exp} and note that due to Proposition \ref{prop:linearinhom:Neumann} we in fact have $u_\star = v_\star + w_\star$ with 
$v_\star \in \Exp^{-\om} \E_{u,0}$ and $\partial_t^3 w_\star \in \Exp^{-\om} L_p(\R_+)$. Then $w_\star$, $\partial_t w_\star$ and
$\partial_t^2 w_\star$ are in general not contained in $\Exp^{-\om} L_p(\R_+)$ and thus \eqref{firstterm:analytic} fails.
However, if we assume that the data $(g,h,u_0,u_1,u_2)$ have zero mean then $w_\star = 0$ and, since $u_\star = v_\star$ in this case,
Theorem \ref{thm:global:Neumann} follows analogously to the result on global well-posedness for the Dirichlet boundary value problem 
(Theorem \ref{thm:global:Dirichlet}).
\end{remark}

Finally, provided the data have zero mean, we obtain the following result on exponential stability for the Neumann problem \eqref{IBVP:Neumann}.

\begin{theorem}[Exponential stability - the Neumann case]
\label{thm:decay:Neumann}
Under the same assumptions as in Theorem \ref{thm:global:Neumann}, the solution $u$ decays exponentially fast to zero as $t \rightarrow \infty$ in the sense that
\begin{equation*}
\|u(t)\|_{W_p^4} + \|u_t(t)\|_{W_p^{4-2/p}} + \|u_{tt}(t)\|_{W_p^{2-2/p}} \leq C \Exp^{-\omega t}, \qquad t \geq 0,
\end{equation*}
for some $C \geq 0$ depending on the boundary and initial data $g$, $h$, $u_0$, $u_1$ and $u_2$. 
\end{theorem}
\begin{proof}
Note that we have $u \in \Exp^{-\om} \E_{u,0} \hookrightarrow  \Exp^{-\om} \E_{u}$, therefore the result follows likewise to Theorem \ref{thm:decay:Dirichlet}.
\end{proof}

\appendix 
\section{The Neumann Laplace operator}
\label{app:neumann}
Let $p\in (1, \infty)$ and assume, as always, that $\Omega \subset \R^n$ is a bounded domain with smooth boundary $\Gamma= \partial \Omega$. 
The homogeneous Neumann-Laplacian is given by
\begin{align*}
-\Delta_N\colon \D(\Delta_N) &\rightarrow L_p(\Omega),\\
									  u &\mapsto -\Delta u, 
\end{align*}
where $\D(\Delta_N) = \{u \in W_p^2(\Omega)\colon \partial_\nu u = 0 \text{ on }\Gamma\}$.
It is well-known that $-\Delta_N$ has compact resolvent and that its spectrum $\sigma(-\Delta_N)$ is a discrete
subset of $[0, \infty)$ consisting only of eigenvalues $(\lam_n^N)_{n\geq 0}$ with finite multiplicity. In particular, $0 = \lam_0^N \in \sigma(-\Delta_N)$ is 
an isolated eigenvalue of $-\Delta_N$. We seek for a realization of the Laplace operator with homogeneous 
Neumann boundary conditions such that the spectrum is contained in $[\lam_1^N, \infty)$ where $\lam_1^N>0$ is the smallest non-zero eigenvalue of 
$-\Delta_N$.

In order to remove the zero eigenvalue we will use several results from Appendix A in \cite{Lun95}. In what follows, $A\colon \D(A) \subset X \rightarrow X$ denotes
a linear closed linear operator whose domain $\D(A)$ is dense in the real or complex Banach space $X \neq \{0\}$. 
We say that a subset $\sigma_1 \subset \sigma(A)$ is a spectral set if both, $\sigma_1$ and $\sigma(A)\setminus \sigma_1$ are closed in $\C$.
Let $\sigma_1$ be a bounded spectral set and let $\sigma_2 = \sigma(A) \setminus \sigma_1$. Since $\mathrm{dist}(\sigma_1, \sigma_2) >0$, there
exists a bounded open set $\mathcal{O}$ such that $\sigma_1 \subset \mathcal{O}$ and $\overline{\mathcal{O}} \cap \sigma_2 = \emptyset$. 
We may assume that the boundary $\gamma$ of $\mathcal{O}$  consists of a finite number of rectifiable closed Jordan curves, oriented counterclockwise
and define a linear bounded operator $P$ by
\begin{equation*}
P= \frac{1}{2 \pi \mathrm{i}} \int_\gamma R(\xi, A) \, d\xi.
\end{equation*}
The following result shows how find a realization $A_2$ of $A$ such that $\sigma(A_2) = \sigma(A) \setminus \sigma_1$. 
\begin{proposition}[{\cite[Proposition A.1.2]{Lun95}}]
Let $\sigma_1$ be a bounded spectral set. Then the operator $P$ is a projection and $P(X)$ is contained in $\D(A^n)$ for every $n\in \N$.
Moreover, if we set $X_1=P(X)$, $X_2=(I-P)(X)$ and define the operators
\begin{equation*}
A_1\colon X_1 \rightarrow X_1, u \mapsto Au  \qquad \text{and} \qquad A_2\colon \D(A_2) = \D(A) \cap X_2  \rightarrow X_2, u \mapsto Au										
\end{equation*}
then 
\begin{equation*}
\sigma(A_1)= \sigma_1 \qquad \text{and} \qquad \sigma(A_2) = \sigma_2.
\end{equation*}
\end{proposition}
The crucial point is thus to determine the space $X_2$. In case $\sigma_1=\{\lambda_0\}$ where $\lambda_0$ is an isolated point of $\sigma(A)$ and a 
pole of $R(\cdot, A)$ the following result helps to determine the spaces $X_1$ and $X_2$. 
\begin{proposition}[{\cite[Proposition A.2.2 and Corollary A.2.4]{Lun95}}]
\label{prop:remove}
If $\lambda_0$ is an isolated point of $\sigma(A)$ and a pole of $R(\cdot, A)$, then the following are equivalent:
\begin{enumerate}[label=\emph{(\roman*)}, leftmargin=*]
\item $X_1 = N(\lambda_0 I - A)$
\item $X_2 = R(\lambda_0 I - A)$
\item $\lambda_0$ is a simple pole of $\lambda \mapsto R(\lambda,A)$
\item $R(\lambda_0 I - A)$ is closed and $X= N(\lambda_0 I - A) \oplus R(\lambda_0 - A)$
\item $N(\lambda_0 I - A) = N(\lambda_0 I - A)^2$
\end{enumerate}
\end{proposition}
We now apply the foregoing results to the strong Neumann-Laplacian $-\Delta_N$ and set $\lambda_0 = \lam_0^N = 0$, hence $\sigma_1 = \{0\}$ and 
$\sigma_2 = \sigma(-\Delta_N)\setminus \{0\}$. Then we clearly have $\sigma_2 \subset [\lam_1^N, \infty)$, where $\lam_1^N$ is the smallest 
non-zero eigenvalue of $-\Delta_N$.

\begin{lemma}
The spectrum $\sigma(-\Delta_N)$ of $- \Delta_N$ consists only of poles of $\lambda \mapsto R(\lambda, -\Delta_N)$.
\end{lemma}
\begin{proof}
Since $-\Delta_N$ is closed, densely defined and has compact resolvent, the result is an immediate consequence of Corollary IV.1.19 in \cite{EnNa00}.
\end{proof}
We introduce the space $\K_u = \{u \in L_p(\Omega)\colon u \text{ is constant}\}$
and start with the following observation. 
\begin{lemma}\label{lem:A-4}
We have $N(\Delta_N)=\K_u$. 
\end{lemma}
\begin{proof}
Let $u \in N(\Delta_N)$, i.e. $u \in W_p^2(\Omega)$, $\partial_\nu u = 0$ on $\Gamma$ and $-\Delta u = 0$. For sufficiently large $\mu >0$ the map 
$\mu-\Delta_N : \D(\Delta_N^{j+1}) \to \D(\Delta_N^j)$ is an ismorphism vor every $j \in \N_0$. We write $0 = \Delta_N u = \mu u - (\mu-\Delta_N)u$ 
and obtain $u = (\mu-\Delta_N)^{-1}\mu u$. Hence, if $u\in \D(\Delta_N)$ then we have $u\in \D((\Delta_N)^2)$ and altogether conclude $u\in \D((\Delta_N)^\infty)$.
The Sobolev embedding $W^k_p(\Omega) \hookrightarrow W^2_2(\Omega)$ holds for sufficiently large $k$. Therefore we in fact have $u \in W_2^2(\Om)$ and calculate
\begin{equation*}
0 = - \int_\Omega \Delta u \overline{u} \, dx = \int_\Omega \nabla u \cdot \nabla \overline{u} \, dx - \int_\Gamma \partial_\nu u \overline{u} \, dS = \|\nabla u\|_{L_2}^2,
\end{equation*}
hence $u\in \K_u$. Conversely, every function $u\in \K_u$ trivially satisfies $-\Delta u = 0$. 
\end{proof}

\begin{lemma}
\label{lem:spectrum_of_Neumann_Laplacian}
We have $N(\Delta_N) = N((\Delta_N)^2)$. 
\end{lemma}
\begin{proof}
Let $u \in N((\Delta_N)^2)$, i.e. $u \in \D((\Delta_N)^2)$ and $(-\Delta)^2 u = 0$. Then
\begin{align*}
0 & = \int_\Omega (-\Delta)^2 u\,  \overline{u} \, dx = - \int_\Omega \nabla (\Delta u) \cdot \nabla \overline{u} \, dx + \int_\Gamma \partial_\nu (\Delta u) \overline{u}\, dS\\
   & = \int_\Omega \Delta u \Delta u \, dx = \|\Delta u\|_{L_2}^2,
\end{align*}
hence $-\Delta u = 0$ and thus $u \in N(\Delta_N)$. 
Conversely, let $u \in N(\Delta_N)$. Then $(-\Delta)^2 u = 0$ and $-\Delta u \in \D(\Delta_N)$. 
\end{proof}

Before we proceed, let us recall the space $L_{p,0}(\Omega) = \{u \in L_p(\Omega) \colon \int_\Omega u =0 \}$.

\begin{lemma}
We have $L_p(\Omega) = \K_u \oplus L_{p,0}(\Omega)$ as a topological direct sum.
\end{lemma}
\begin{proof}
We consider the map $P\colon L_p(\Omega) \rightarrow L_{p,0}(\Omega)$, $u \mapsto u - \langle u\rangle_{\Omega}$ where
$\langle u\rangle_\Omega = |\Omega|^{-1} \int_\Omega u\, dx$. It is straightforward to verify that for $u\in L_p(\Omega)$ we indeed 
have $\int_\Omega u- \langle u \rangle_\Omega \, dx = 0$ and hence $Pu \in L_{p,0}(\Omega)$.
Furthermore, $P^2u = Pu- \langle Pu \rangle_\Omega = Pu$ implies that $P$ is a projection. 
Finally $N(P) = \K_u$ since $u-\langle u \rangle_\Omega = 0$ if and only if $u \in L_p(\Omega)$ is constant. 
\end{proof}

Therewith, by means of Proposition \ref{prop:remove} have determined the space $X_2=L_{p,0}(\Om)$ which gives us
the following positive realization of the Neumann Laplacian.

\begin{theorem}\label{thm:spectrum_Neumann_Laplacian}
The spectrum of the closed and densely defined operator
\begin{align*}
-\Delta_{N,0} \colon \D(\Delta_{N,0}) &\rightarrow L_{p,0}(\Omega),\\
u &\mapsto -\Delta u,
\end{align*}
with $\D(\Delta_{N,0}) = \D(\Delta_N) \cap L_{p,0}(\Omega)$ is a discrete subset of $[\lam_1^N, \infty)$, where 
$\lam_1^N>0$ is the smallest non-zero eigenvalue of $-\Delta_N$. Moreover, $\sigma(-\Delta_{N,0})$ consists only of eigenvalues with finite algebraic multiplicity.  
\end{theorem}

\section{Traces and mixed derivatives}
\label{sec:trac-mixed-deriv}

In this section we consider the temporal trace operator in some anisotropic fractional Sobolev spaces.  Furthermore, we present some mixed derivative embeddings for such spaces which are needed for proving suitable mapping properties of differential operators.

A bounded linear operator $r\colon X\to Y$ between Banach spaces $X$ and $Y$ is called a \emph{retraction}, if there is a bounded linear map $r^c\colon  Y \to X$ such that $rr^c = I_Y$.  Thus $r$ is surjective and $r^c$ is a bounded right-inverse for $r$.  
The map $r^c$ is called a co-retraction for $r$.

The following trace theorem can be derived from {\cite[Lemma 11]{diB84}}, {\cite[Section 2.2.1]{Lun95}}, {\cite[Proposition III.4.10.3]{Ama95}}.  

\begin{theorem}\label{thm:trace_semigroup}
  Let $A$ be the generator of a bounded analytic semigroup $(\Exp^{-tA})_{t\geq 0}$ in a Banach space $X$ such that $A\colon \D(A)\to X$ has a bounded inverse, let $p\in(1,\infty)$ and let $\D_A(\alpha,p) := (X,\D(A))_{\alpha,p}$ for $\alpha\in(0,1)$ and $\D_A(1,p) := \D(A)$.  

  Then, for every $\alpha\in(1/p,1]$, the trace operator
  \begin{align*}
   \gamma_t = \cdot|_{t=0} \colon  W^\alpha_p(\bbR_+;X)\cap L_p(\bbR_+;\D_A(\alpha,p)) \to \D_A(\alpha-1/p,p) 
  \end{align*}
  is a retraction, the operator
  \begin{align*}
    R_A \colon  u_0\mapsto \(t\mapsto \Exp^{-t A}u_0\), \quad \D_A(\alpha-1/p,p) \to W^\alpha_p(\bbR_+;X)\cap L_p(\bbR_+;\D_A(\alpha,p))
  \end{align*}
  is a co-retracton for $\gamma_t$ and the following embedding is continuous.
  \begin{align*}
    W^\alpha_p(\bbR_+;X)\cap L_p(\bbR_+;\D_A(\alpha,p)) \hookrightarrow BUC(\bbR_+;\D_A(\alpha-1/p,p)).
  \end{align*}
\end{theorem}

This theorem can be applied to the spaces $W^\alpha_p(\bbR_+;W^s_p(\Omega;E)) \cap L_p(\bbR_+;W^{s+2\alpha}_p(\Omega;E)$ for $\alpha\in(1/p,1]$, $s\in[0,\infty)$, provided that in the case $\alpha<1$ the number $s+2\alpha$ is not an integer and provided that $s+2\alpha\leq 2k$.  Here and in the following, we assume that $E$ is a Banach space of class $\mathcal{HT}$ and has property $(\alpha)$, where we refer to \cite{KuWe04} and \cite{KaSa12} for the definitions of such spaces and additional information.  For instance, any Hilbert space is of class $\mathcal{HT}$ with property $(\alpha)$ and these properties are inherited to closed subspaces and isomorphic spaces.  Moreover, the space $L_q(\Omega,\mathcal A,\mu;E)$ on a $\sigma$-finite measure space $(\Omega,\mathcal A,\mu)$ with $q\in(1,\infty)$ is of class $\mathcal{HT}$ and has property $(\alpha)$.  For $s\in(0,\infty)$ and $p$, $q\in(1,\infty)$, the Sobolev-Slobodecki\u{\i} spaces $W^s_p(\bbR^n;E))$, the Bessel potential spaces $H^s_p(\bbR^n;E)$ and the Besov spaces $B^s_{p,q}(\bbR^n;E)$ and are also of class $\mathcal{HT}$ with property $(\alpha)$.

Let us indicate how Theorem \ref{thm:trace_semigroup} can be applied.  Let first $\calE_\Omega$ be an extension operator from $\Omega$ to $\bbR^n$ which acts as a bounded linear operator $W^t_p(\Omega;E) \to W^t_p(\bbR^n;E)$ for all $t\in[0,2k]$.  Such extension operators are defined in \cite{AdFo03} for $t\in\N_0$ and their boundedness for $t\notin\N_0$ follows from real interpolation.  Then it remains to study $\calE_\Omega u$ in $\bbR_+\times\bbR^n$.  In this situation the operator $A=1-\Delta$ in $X=W^s_p(\bbR^n;E)$ with domain $W^{s+2}_p(\bbR^n;E)$ has the required properties.  Indeed, \cite[Section 5]{DHP03} covers the case $s=0$ and an abstract result of Dore \cite{Dor99} covers the case $s\in(0,2)\setminus\{1\}$.  The remaining cases follow by means of isomorphic mappings, interpolation and taking fractional powers.  Hence $\D_A(\alpha,p) = W^{s+2\alpha}_p(\bbR^n;E)$ and $\D_A(\alpha-1/p,p) = B^{s+2\alpha-2/p}_{p,p}(\bbR^n;E)$.  Then the temporal trace operator can be rewritten as $\gamma_t \colon u \mapsto ((\calE_\Omega u)|_{t=0})|_\Omega$ and acts as a bounded linear operator
\begin{align*}
  W^\alpha_p(\bbR_+;W^s_p(\Omega;E)) \cap L_p(\bbR_+;W^{s+2\alpha}_p(\Omega;E)) \to B^{s+2\alpha-2/p}_{p,p}(\Omega;E).
\end{align*}

For the boundary spaces $W^\alpha_p(\bbR_+;W^s_p(\Gamma;E)) \cap L_p(\bbR_+;W^{s+2\alpha}_p(\Gamma;E))$ we use a common retraction $r \colon  W^t_p(\bbR^{n-1};E)^N \to W^t_p(\Gamma;E)$ for all $t\in[0,2k]$ with some $N\in\N$.   A co-retraction for $r$ can be constructed by means of a partition of unity for $\Gamma$ and local parametrizations of $\Gamma$ over subsets of $\bbR^{n-1}$ as in the proof of Lemma \ref{lem:right_inverse_for_spatial_trace}.  Then the temporal trace operator can be rewritten as $\gamma_t \colon u\mapsto r((r^cu)|_{t=0})$ and maps
\begin{align*}
  W^\alpha_p(\bbR_+;W^s_p(\Gamma;E)) \cap L_p(\bbR_+;W^{s+2\alpha}_p(\Gamma;E)) \to B^{s+2\alpha-2/p}_{p,p}(\Gamma;E).
\end{align*}

In order to construct functions with prescribed initial values, we consider an operator $A\colon \D(A)\subset X\to X$ as in Theorem \ref{thm:trace_semigroup} and define the spaces
\begin{align*}
  \D_A(k+\alpha,p) := A^{-k}\D_A(\alpha,p) = (\D(A^k),\D(A^{k+1}))_{\alpha,p} \quad\text{ for }k\in\bbN_0, \,\alpha\in[0,1],\,p\in(1,\infty).
\end{align*}
Then Theorem \ref{thm:trace_semigroup} and the identity $\del_y\Exp^{-yA}=-A\Exp^{-yA}=\Exp^{-yA}A$ yield the following result.
\begin{corollary}\label{cor:trace_semigroup}
  Let $k\in\bbN_0$, $\alpha\in(1/p,1]$ and $p\in(1,\infty)$.  Then the operator
  \begin{align*}
    R_A \colon  u\mapsto \(t\mapsto \Exp^{-t A}u\), \quad \D_A(k+\alpha-1/p,p) \to W^{k+\alpha}_p(\bbR_+;X)\cap L_p(\bbR_+;\D_A(k+\alpha,p))
  \end{align*}
  is a bounded right-inverse for $\gamma_t$.  
\end{corollary}

We next deal with higher order initial conditions.
\begin{lemma}\label{lem:temp-trace-retraction}
  Let $\gamma_t^j := (\del_t^j\cdot)|_{t=0}$ and let $l\in\N_0$, $m\in\N$ with $m\geq l+1$.  Then the operator
  \begin{align*}
    (\gamma_t^0,\gamma_t^1,\ldots,\gamma_t^l) \colon  W^m_p(\bbR_+;X) \cap L_p(\bbR_+;\D(A^m)) \to {\prod}_{j=0}^l \D_A(m-j-1/p,p)
  \end{align*}
  is a retraction.
\end{lemma}
\begin{proof}
  For $j\in\{0,1,\ldots,l\}$ and $x\in X$ we define
  \begin{align*}
    (S^A_j x)(t,\cdot) := {\sum}_{i=0}^l c_{ij} \Exp^{- t(1+i) A} A^{-j}x \quad\text{ for } t \geq 0,
  \end{align*}
  where, for each $j$, the $l+1$ numbers $c_{ij}$ ($i\in\{0,1,\ldots,l\}$) solve the linear system
  \begin{align*}
    {\sum}_{i=0}^l c_{ij} (-(1+i))^m = \delta_{mj} \quad\text{ for }m \in\{0,1,\ldots,l\}.
  \end{align*}
  By using Vandermonde's matrix
  \begin{align*}
    V = \begin{pmatrix} 1 & 1 & \cdots & 1 \\ 1 & 2 & \cdots & 1+k \\ \vdots & \vdots & & \vdots \\  1^k & 2^k & \cdots & (1+k)^k \end{pmatrix} \quad\text{ with } \quad \det V = (-1)^{\frac{k(k+1)}2}\prod_{j=1}^k j! \neq 0,
  \end{align*}
  the numbers $c_{ij}$ are given by $(c_{0j},\ldots,c_{lj})^\sfT := V^{-1} e_j^\sfT$.  Hence
  \begin{align*}
    (\del^m_t S_j^A x)(0)  = {\sum}_{i=0}^l c_{ij} (-(1+i))^m x = \delta_{mj} x \quad\text{ for } m \in\{0,1,\ldots,l\}.
  \end{align*}
  From Corollary \ref{cor:trace_semigroup} we infer that $S_j^A$ acts as a bounded linear operator
  \begin{align*}
    S_j^A \colon  \D_A(m-j-1/p,p) \to W^m_p(\bbR_+;X) \cap L_p(\bbR_+;\D(A^m)) \quad\text{ for } m\in\bbN, \, m \geq j+1.
  \end{align*}
  Therefore the desired co-retraction is given by
  \begin{align*}
    S^A(x_0,x_1,\ldots,x_l) &:= {\sum}_{j=0}^l S^A_j x_j. \qedhere 
  \end{align*}
\end{proof}

\begin{theorem}[Mixed derivative embeddings]\label{thm:mixed_deri}
  Let $n\in\bbN$, $p\in(1,\infty)$, $t,s\in[0,\infty)$, $\tau,\sigma\in(0,\infty)$, $\theta\in(0,1)$, $E\in\mathcal{HT}$, $J=\bbR$ or $J = (0,T)$ for $T\in(0,\infty]$ and let $\Omega$ be the whole space $\bbR^n$ or a bounded domain with smooth boundary or a compact smooth hypersurface of $\bbR^n$.  

  Then the following embeddings are continuous.
  \begin{subequations}\label{eq:mixed_deri}      
    \begin{align}
      H^{t+\tau}_p(J;H^s_p(\Omega;E)) \cap H^t_p(J;H^{s+\sigma}_p(\Omega;E)) &\hookrightarrow H^{t+\theta\tau}_p(J;H^{s+(1-\theta)\sigma}_p(\Omega;E)) \label{eq:mixed_deri_H},\\
      B^{t+\tau}_{p,p}(J;H^s_p(\Omega;E)) \cap H^t_p(J;B^{s+\sigma}_{p,p}(\Omega;E)) &\hookrightarrow B^{t+\theta\tau}_{p,p}(J;H^{s+(1-\theta)\sigma}_p(\Omega;E)) \label{eq:mixed_deri_BH-BH},\\
      B^{t+\tau}_{p,p}(J;H^s_p(\Omega;E)) \cap H^t_p(J;B^{s+\sigma}_{p,p}(\Omega;E)) &\hookrightarrow H_p^{t+\theta\tau}(J;B_{p,p}^{s+(1-\theta)\sigma}(\Omega;E)) \label{eq:mixed_deri_BH-HB}.
    \end{align}
  \end{subequations}
\end{theorem}
\begin{proof}
  We adapt the proof of \cite[Proposition 3.2]{MeSc12}.  It is sufficient to consider the case $J\times\Omega = \bbR\times\bbR^n$ since the other spaces are retracts of corresponding spaces over $\R\times\R^n$.  

  In the ground space $X=H^t_p(H^s_p) := H^t_p(\R;H^s_p(\R^n;E))$ we consider the operators
  \begin{align*}
    \begin{aligned}
      A &= (1-\del_t^2)^{t/2}, & \D(A) &= H^{t+\tau}_p(\R;H^s_p(\bbR^n;E)),\\
      B &= (1-\Delta_n)^{s/2}, & \D(B) &= H^t_p(\R;H^{s+\sigma}_p(\R^n;E)).        
    \end{aligned}
  \end{align*}
  Here $\Delta_n$ denotes the Laplacian in $\bbR^n$ with respect to the spatial variable.  It follows from \cite[Theorem 5.5]{DHP03} (see also \cite[Lemma 3.1]{MeSc12}) that the operator
  \begin{align*}
    J_r := (1-\Delta_m)^{r/2} \colon  H^{r+\rho}_p(\R^m;F) \to H^r_p(\R^m;F)
  \end{align*}
  is invertible and has a bounded $\calH^\infty$ functional calculus and thus bounded imaginary powers in $H^r_p(\R^m;F)$ for all $m\in\{1,n\}$, $r,\rho\in[0,\infty)$ and all Banach spaces $F$ of class $\calHT$.  The latter property implies that its fractional powers $J_r^\theta$ ($\theta\in(0,1)$) have the domains $\D(J_r^\theta) = [H^r_p(\R^m;F),H^{r+\rho}_p(\R^m;F)]_\theta = H^{r+\theta\rho}_p(\R^m;F)$ by \cite[Theorem 2.5]{DHP03} and complex interpolation.  

  By choosing $(m,r,\rho,F)=(1,t,\tau,H^s_p(\bbR^n;E))$ or $(m,r,\rho,F) = (n,s,\sigma,H^t_p(\R;E))$ and employing Fubini's theorem, we see that $A\colon  \D(A)\to X$ and $B \colon  \D(B)\to X$ are invertible and have bounded $\calH^\infty$ functional calculi.  Sobolevski\u{\i}'s mixed derivative theorem \cite[Theorem 6]{Sob75} implies that $\D(A) \cap \D(B) \hookrightarrow \D(A^\theta B^{1-\theta}) \cap \D(B^{1-\theta}A^\theta)$ which proves \eqref{eq:mixed_deri_H}.  

  For proving \eqref{eq:mixed_deri_BH-BH} we apply \eqref{eq:mixed_deri_H} and real interpolation to the space
  \begin{align*}
    H^{t+\tau\pm\eps\tau}_p(H^s_p) \cap H^t_p(H^{s+\sigma\pm\eps\sigma}_p),
  \end{align*}
  for sufficiently small $\eps>0$.  For $\tau_\pm := (1\pm\eps)\tau$ and $\sigma_\pm := (1\pm\eps)\sigma$ and $\theta_\pm\in(0,1)$, we obtain
  \begin{align*}
    H^{t+\tau\pm\eps\tau}_p(H^s_p) \cap H^t_p(H^{s+\sigma\pm\eps\sigma}_p) =
    H^{t+\tau_\pm}_p(H^s_p) \cap H^t_p(H^{s+\sigma_\pm}_p) \hookrightarrow H^{t+\theta_\pm\tau_\pm}_p(H^{s+(1-\theta_\pm)\sigma_\pm}_p).
  \end{align*}
  We choose $\theta_\pm$ such that $(1-\theta_\pm)\sigma_\pm = (1-\theta)\sigma$, that is, $\theta_\pm := (\theta\pm\eps)/(1\pm\eps)$.  Then $\theta_\pm\eps_\pm = \theta\tau \pm \eps\tau$ and it remains to apply the real interpolation functor $(\cdot,\cdot)_{1/2,p}$ to
  \begin{align*}
    Z_\pm := H^{t+\tau\pm\eps\tau}_p(H^s_p) \cap H^t_p(H^{s+\sigma\pm\eps\sigma}_p) \hookrightarrow H^{t+\theta\tau\pm\eps\tau}_p(H^{s+(1-\theta)\sigma}_p).
  \end{align*}
  Indeed, interpolation of the right-hand side yields
  \begin{align*}
    (Z_-,Z_+)_{1/2,p} \hookrightarrow \(H^{t+\theta\tau-\eps\tau}_p(H^{s+(1-\theta)\sigma}_p),H^{t+\theta\tau+\eps\tau}_p(H^{s+(1-\theta)\sigma}_p)\)_{1/2,p} = B^{t+\theta\tau}_{p,p}(H^{s+(1-\theta)\sigma}_p).
  \end{align*}
  For an interpolation of the left-hand side we write $Z_+ = \D(L)$ and $Z_-=\D(L^{(1-\eps)/(1+\eps)})$ where the operator $L = (1-\del_t^2)^{(1+\eps)\tau/2} + (1-\Delta_n)^{(1+\eps)\sigma/2}$ is considered in the ground space $Z_0 := H^t_p(H^s_p)$.  Then the reiteration theorem (\cite[Remark 1.2.16]{Lun95}) yields
  \begin{align*}
    (Z_-,Z_+)_{1/2,p} &= (\D(L^{(1-\eps)/(1+\eps)}), \D(L))_{1/2,p} \\
    &= \D_L\(1/2+(1-\eps)/(2+2\eps),p\) = B^{t+\tau}_{p,p}(H^s_p) \cap H^t_p(B^{s+\sigma}_{p,p}).
  \end{align*}
  Hence \eqref{eq:mixed_deri_BH-BH} is proved.  The proof of \eqref{eq:mixed_deri_BH-HB} is similar and therefore omitted.  
\end{proof}

\section{Higher regularity for the heat equation}
\label{app:heathigher}

We study the regularity of solutions of the heat problem
\begin{align}\label{eq:ap_heat_problem}
  \left\{
    \begin{aligned}
      (\del_t+\mu_B-\Delta)u &= f &&\text{ in }J\times\Omega,\\
      \gamma_B u &= g &&\text{ on }J\times\Gamma,\\
      u|_{t=0} &= u_0 &&\text{ in }\Omega.     
    \end{aligned}\right.
\end{align}
Here $J$ is a bounded interval $(0,T)$ or the half line $(0,\infty)$ and $\Omega$ is a bounded domain in $\R^n$, $n\in\N$, with smooth boundary $\Gamma$.  For $B\in\{D,N\}$, let $\mu_B$ be a real number and let
\begin{align*}
  \gamma_D := \cdot|_\Gamma, \quad \gamma_N := (\del_\nu\cdot)|_\Gamma = \nu\cdot(\nabla\cdot)|_\Gamma, \quad \gamma_t^j := (\del_t^j \cdot)|_{t=0}, \, \gamma_t := \gamma_t^0
\end{align*}
denote the Dirichlet, the Neumann, and the temporal trace operators, respectively.  Again we let $\lambda_0^D>0$ denote the smallest eigenvalue of $-\Delta_D$ and $\lambda_1^N>0$ denote the smallest non-zero eigenvalue of $-\Delta_N$.  We will prove the following regularity result.  

\begin{theorem}\label{thm:higher_regularity_for_the_heat_equation}
  Let $B\in\{D,N\}$, $j_D=0$, $j_N=1$, $\mu_D\in(-\lambda_0^D,\infty)$, $\mu_N\in(0,\infty)$, $l\in\bbN_0$, $k\in\bbN$ and $p\in(1,\infty)$ such that $j_B/2+3/2p\neq 1$.  Then problem \eqref{eq:ap_heat_problem} has a unique solution
  \begin{align}\label{eq:150322_1}
    u &\in \bbE^{l,k} := W^{l+k}_p(J;L_p(\Omega)) \cap W^l_p(J;W^{2k}_p(\Omega)), 
  \end{align}
  if and only if the data $(f,g,u_0)$ satisfy the regularity conditions
  \begin{subequations}\label{eq:150326}
    \begin{align}
      f \in \bbE^{l,k-1} &= W^{l+k-1}_p(J;L_p(\Omega)) \cap W^l_p(J;W^{2k-2}_p(\Omega)), \\
      g \in \gamma_B \bbE^{l,k} &= W^{l+k-j_B/2-1/2p}_p(J;L_p(\Gamma)) \cap W^l_p(J;W^{2k-j_B-1/p}_p(\Gamma)),\\
      u_0 \in \gamma_t \bbE^{l,k} &= \left\{
        \begin{aligned}
          &W^{2k}_p(\Omega) &&\text{ if }l\geq 1,\\
          &W^{2k-2/p}_p(\Omega)&&\text{ if }l=0,
        \end{aligned}\right.
    \end{align}
  \end{subequations}
  and the compatibility conditions
  \begin{subequations}\label{eq:150323}
      \begin{align}
        &u_j := \gamma_t^{j-1}f + (\Delta-\mu_B)u_{j-1} \in \left\{
          \begin{aligned}
            &W^{2k}_p(\Omega) && \text{ for } j \in \bbN\cap[1,l-1],\\
            &W^{2(l+k-j)-2/p}_p(\Omega) &&\text{ for } j \in\bbN \cap [l,l+k-1],       
          \end{aligned}\right.\label{eq:150323_2}\\          
        &\gamma_t^jg = \gamma_B u_j \quad\text{ for } j \in\bbN_0 , \, j\leq l+k-j_B/2-3/2p\label{eq:150323_3}.
      \end{align}
  \end{subequations}
\end{theorem}

\begin{remark}
  \begin{enumerate}[label={(\roman*)}, leftmargin=*]
  \item The space $\bbE^{0,1}$ ($l=0$, $k=1$) is the standard parabolic solution space.
  \item If $l=0$ and $J\times\Omega$ is the half space $\bbR_+\times\bbR^n$ or the wedge $\bbR_+\times\bbR^n_+$, then $\bbE^{0,k}$ is the anisotropic space $H^{2k/\boldsymbol\nu}_p(J\times\Omega)$ with weight $\boldsymbol\nu=(2,1,\ldots,1)$ in the sense of \cite{Ama09}.  This fact will be used in the construction of functions with prescribed boundary values.
  \item We exclude the case $j_B/2+3/2p=1$ in order to avoid the more complicated trace spaces $\gamma_DW^{2/3}_{3/2}(\Omega)$ and $\gamma_N W^{4/3}_{3}(\Omega)$.
  \item The additional regularity conditions \eqref{eq:150323_2} follow from the non-triangular structure of the space $\bbE^{l,k}$ in the case $l\geq 1$ and are derived in Subsection \ref{ssec:comp-cond}.  For $j\geq l+1$, formula \eqref{eq:150323_2} does not contain additional regularity conditions and should be merely understood as the definition of the functions $u_j$, which appear in \eqref{eq:150323_3}.
  \item Every solution satisfies the higher order boundary conditions
    \begin{align}\label{eq:150324}
      \gamma_B \Delta^{j+1} u = (\del_t+\mu_B)g_j - \gamma_B\Delta^j f =: g_{j+1}\quad\text{ for } j  \in\bbN_0\cap[0,k-2], \quad\text{ with } g_0:=g.
    \end{align}
    With the temporal trace theorem and $u_i = \gamma_t^i u$ we obtain
    \begin{align}\label{eq:150421}
      \gamma_B\Delta^j u_i = \gamma_t^i g_j \quad\text{ for } j \in \bbN_0 \cap [0,k-1], \, i\in\bbN_0 \cap [0,l+k-j-j_B/2-3/2p].
    \end{align}
    These equations are no additional regularity or compatibility conditions but follow from \eqref{eq:150323}, \eqref{eq:150324}, by induction over $j\in\bbN_0$.  Indeed, suppose that $\gamma_B\Delta^j u_i = \gamma_t^i g_j$ for all $i$ and some $j$.  Then \eqref{eq:150324}, the induction hypothesis, and \eqref{eq:150323} yield
    \begin{align*}
      \gamma_t^i g_{j+1} &= \gamma_t^i \( (\del_t+\mu_B)g_j - \gamma_B\Delta^j f\) = \gamma_t^{i+1} g_j + \mu_B \gamma_t^i g_j - \gamma_B \gamma_t^i \Delta^j f\\
      &= \gamma_B\Delta^j u_{i+1} + \mu_B \gamma_B \Delta^j u_i - \gamma_B \gamma_t^i \Delta^j f \\
      &= \gamma_B \Delta^j\(\gamma_t^i f + (\Delta-\mu_B) u_i\) + \mu_B\gamma_B\Delta^j u_i - \gamma_B\gamma_t^i\Delta^j f= \gamma_B\Delta^{j+1} u_i.
    \end{align*}
  \item \label{Ableitung_Lemma_2_3} In the case $l=1$, $k=2$, $B=D$, the compatibility conditions read as
    \begin{align*}
      u_1 &:= f|_{t=0} + (\Delta-\mu_D)u_0 \in W^{2k}_p(\Omega),\\
      u_2 &:= \del_tf|_{t=0} + (\Delta-\mu_D)u_1 \in W^{2k-2/p}_p(\Omega),\\
      g|_{t=0} &= u_0|_{\Gamma},\\
      \del_tg|_{t=0} &= f|_{t=0,\Gamma}+((\Delta-\mu_D)u_0)|_\Gamma,\\
      \del_t^2g|_{t=0} &= \del_tf|_{t=0,\Gamma}+((\Delta-\mu_D)u_1)|_\Gamma \quad\text{ if } p>3/2.
    \end{align*}
  \item \label{Kompatibilitaeten_fur_Neumann} The corresponding result for $B=N$ has the compatibility conditions
    \begin{align*}
      u_1 &:= f|_{t=0} + (\Delta-\mu_N)u_0 \in W^{2k}_p(\Omega),\\
      u_2 &:= \del_tf|_{t=0} + (\Delta-\mu_N)u_1 \in W^{2k-2/p}_p(\Omega),\\
      g|_{t=0} &= \del_\nu u_0|_{\Gamma},\\
      \del_tg|_{t=0} &= \del_\nu f|_{t=0,\Gamma}+\del_\nu((\Delta-\mu_N)u_0)|_\Gamma,\\
      \del_t^2g|_{t=0} &= \del_\nu\del_tf|_{t=0,\Gamma}+\del_\nu((\Delta-\mu_N)u_1)|_\Gamma \quad\text{ if } p>3.
    \end{align*}    
  \end{enumerate}
\end{remark}

Aiming at stability for Neumann boundary conditions, we will also prove the following result, where we consider the subspace $L_{p,0}(\Omega) := \{f\in L_p(\Omega) : \int_\Omega f(x)dx=0\}$.  
\begin{corollary}\label{cor:higher_regularity_for_the_heat_equation_Neumann_stability}
  Let $B=N$, $\mu_N\in(-\lambda_1^N,\infty)$, $l\in\bbN_0$, $k\in\bbN$, $p\in(1,\infty)$, $p\neq 3$.  Then problem \eqref{eq:ap_heat_problem} has a unique solution
  \begin{align*}
    u \in \bbE^{l,k}_0 := \bbE^{l,k} \cap L_p(J;L_{p,0}(\Omega)),
  \end{align*}
  if and only if the data $(f,g,u_0)$ satisfy the regularity conditions \eqref{eq:150326} and the compatibility conditions \eqref{eq:150323} and
  \begin{align}\label{eq:150403_1}
    \int_\Omega u_0(x) \, dx = 0, \quad \int_\Omega f(t,x) \, dx + \int_\Gamma g(t,x) \, dS(x) = 0 \text{ for }t\in J.
  \end{align}
\end{corollary}

Next, we study the original heat problem
\begin{align}\label{eq:ap_heat_problem_without_shift}
  \left\{
    \begin{aligned}
      (\del_t-\Delta)u &= f &&\text{ in }\bbR_+\times\Omega,\\
      \gamma_B u &= g &&\text{ on }\bbR_+\times\Gamma,\\
      u|_{t=0} &= u_0 &&\text{ in }\Omega.     
    \end{aligned}\right.
\end{align}

\begin{corollary}\label{cor:higher_regularity_for_the_heat_equation_exponential_weight}
  \emph{(i)}\quad Let $B\in\{D,N\}$, $\mu_D\in(-\lambda_0^D,\infty)$, $\mu_N\in(0,\infty)$, $l\in\bbN_0$, $k\in\bbN$, $p\in(1,\infty)$ such that $j_B/2+3/2p\neq 1$.  Then problem \eqref{eq:ap_heat_problem_without_shift} has a unique solution $u\in \Exp^{\mu_B}\bbE^{l,k}$ if and only if the data $(f,g,u_0)$ satisfy the regularity conditions
  \begin{align}\label{eq_150326_1}
    (f,g,u_0) &\in \Exp^{\mu_B}\bbE^{l,k-1} \times \Exp^{\mu_B} \gamma_B\bbE^{l,k} \times \left\{
      \begin{aligned}
        &W^{2k}_p(\Omega) &&\text{ if }l\geq 1,\\
        &W^{2k-2/p}_p(\Omega) &&\text{ if }l=0,
      \end{aligned}\right.
  \end{align}
  and the compatibility conditions
  \begin{subequations}\label{eq:150518}
    \begin{align}
      &u_j := \gamma_t^{j-1}f + \Delta u_{j-1} \in \left\{
        \begin{aligned}
          &W^{2k}_p(\Omega) && \text{ for } j\in\bbN\cap[1,l-1],\\
          &W^{2(l+k-j)-2/p}_p(\Omega) &&\text{ for } j\in\bbN\cap[l,l+k-1],       
        \end{aligned}\right.\label{eq:150518_2}\\
      &\gamma_t^jg = \gamma_B u_j \quad\text{ for } j \in\bbN_0 \cap[0,l+k-j_B/2-3/2p]\label{eq:150518_3}.
    \end{align}
  \end{subequations}
    
  \emph{(ii)}\quad Let $B=N$, $\mu_N\in(-\lambda_1^N,\infty)$.  Then problem \eqref{eq:ap_heat_problem_without_shift} has a unique solution $u\in \Exp^{\mu_N}\bbE^{l,k}_0$ if and only if the data $(f,g,u_0)$ satisfy the regularity conditions \eqref{eq_150326_1} and the compatibility conditions \eqref{eq:150518}, \eqref{eq:150403_1}.
\end{corollary}
\begin{proof}
  In problem \eqref{eq:ap_heat_problem} we multiply $f$, $g$ with $\Exp^{\mu_B t}$, so that
  \begin{align*}
    \Exp^{\mu_Bt}f = \Exp^{\mu_Bt}(\del_t+\mu_B-\Delta)u = (\del_t-\Delta)\Exp^{\mu_Bt}u, \quad \Exp^{\mu_Bt}g=\gamma_B\Exp^{\mu_Bt}u.
  \end{align*}
  This shows that $\Exp^{\mu_Bt}u$ solves \eqref{eq:ap_heat_problem_without_shift} for $(\Exp^{\mu_Bt}f,\Exp^{\mu_Bt}g,u_0)$ if and only if $u$ solves the shifted problem \eqref{eq:ap_heat_problem} for $(f,g,u_0)$.  Hence Theorem \ref{thm:higher_regularity_for_the_heat_equation} and Corollary \ref{cor:higher_regularity_for_the_heat_equation_Neumann_stability} yield the assertions.
\end{proof}

\subsection{Compatibility conditions}
\label{ssec:comp-cond}

By means of the results from Section \ref{sec:trac-mixed-deriv} it is not difficult to verify that the regularity conditions \eqref{eq:150326} are indeed necessary for $u\in\bbE^{l,k}$.  Let us now derive the remaining compatibility conditions.

First, any function $u\in\bbE^{l,k}$ satisfies the initial regularity conditions
\begin{align}\label{eq:ap_initial_regularity}
  \gamma_t^j u \in \left\{
    \begin{aligned}
      &W^{2k}_p(\Omega) &&\text{ for } j\in\bbN_0\cap[0,l-1],\\
      &W^{2l+2k-2j-2/p}_p(\Omega) &&\text{ for } j\in\bbN_0\cap[l,l+k-1].
    \end{aligned}\right.
\end{align}
If $u$ solves \eqref{eq:ap_heat_problem} with data $(f,g,u_0)$, then an application of $\del_t^{j-1}$ to the heat equation yields
\begin{align*}
   \del_t^j u = \del_t^{j-1}f + (\Delta-\mu_B) \del_t^{j-1}u.
\end{align*}
In particular, the initial values $u_j := \del_t^j u|_{t=0}$ are given in terms of $f$ and $u_0$ by
\begin{align*}
  u_j = {\sum}_{i=0}^{j-1}(\Delta-\mu_B)^{j-1-i}\gamma_t^if + (\Delta-\mu_B)^{j-1} u_0 \quad\text{ for } j\in\bbN\cap[1,l+k-1].
\end{align*}
Then \eqref{eq:ap_initial_regularity} implies that the data $(f,u_0)$ and $u_j$ must satisfy
\begin{align*}
  u_j = \gamma_t^{j-1} f + (\Delta-\mu_B) u_{j-1} \in \left\{
    \begin{aligned}
      &W^{2k}_p(\Omega) &&\text{ for } j\in\bbN\cap[1,l-1],\\
      &W^{2(l+k-j)-2/p}_p(\Omega) &&\text{ for } j\in\bbN\cap[l,l+k-1].
    \end{aligned}\right.
\end{align*}
For $j\in[1,l]$ this is indeed an additional condition, since $f$ merely satisfies
\begin{align*}
  \gamma_t^{j-1} f \in \left\{
    \begin{aligned}
      &W^{2k-2}_p(\Omega) &&\text{ for } j\in\bbN_0\cap[0,l],\\
      &W^{2l+2k-2j-2/p}_p(\Omega) &&\text{ for } j \in\bbN\cap[l+1,l+k-1].
    \end{aligned}\right.
\end{align*}
The conditions for $l+1\leq j\leq l+k-1$ then follow from the regularity of $f$, $u_0$, \ldots, $u_l$ and could therefore be omitted in \eqref{eq:150323_2}, but we keep them there as a definition of $u_{l+1}$, \ldots, $u_{l+k-1}$.  Indeed, these functions still admit traces on $\Gamma$.  By differentiating the boundary condition $\gamma_B u = g$ with respect to time, we obtain
\begin{align*}
  \gamma_t^j g = \gamma_B u_j \in \left\{
    \begin{aligned}
      &W^{2k-j_B-1/p}_p(\Gamma)&&\text{ for } j\in\bbN_0 \cap [0,l-1],\\
      &W^{2(l+k-j)-j_B-3/p}_p(\Gamma)&&\text{ for } j\in\bbN_0 \cap[l,l+k-j_B/2-3/2p].
    \end{aligned}
  \right.
\end{align*}
This shows that \eqref{eq:150323_3} is a necessary condition.

We conclude that the necessity part of Theorem \ref{thm:higher_regularity_for_the_heat_equation} is true, that is, if problem \eqref{eq:ap_heat_problem} has a solution $u\in\bbE^{l,k}$ with data $(f,g,u_0)$, then \eqref{eq:150326} and \eqref{eq:150323} are satisfied.  We next prepare the proof of the existence part.

\subsection{Interior regularity and initial conditions}
\label{ssec:Interior_regularity_conditions}

From \cite[Theorem 8.2]{DHP03} we deduce that for $B\in\{D,N\}$ and $E\in\calHT$ there exists $\mu_B\geq 0$ such that the realization
\begin{align*}
  \mu_B-\Delta_B \text{ with domain } \D(\Delta_B) &= \{u\in W^2_p(\Omega;E) : \gamma_B u = 0\} \text{ in } L_p(\Omega;E),
\end{align*}
has maximal regularity of type $L_p(\R_+;L_p(\Omega;E))$.  Thus the operator
\begin{align*}
  \mu_B+\del_t-\Delta_B \colon  \0 W^1_p(\R_+;L_p(\Omega;E)) \cap L_p(\R_+;\D(\Delta_B)) \to L_p(\R_+;L_p(\Omega;E))
\end{align*}
is invertible for $B\in\{D,N\}$.  Here and in the following, the space $\0 W^s_p(\bbR_+;E)$ is the closure of $C^\infty_c(\bbR_+;E)$ in $W^s_p(\bbR_+;E)$.  For $k<s-1/p<k+1$ with $k\in\bbN_0$, they consist precisely of those functions with vanishing initial traces $\del_t^j u(0)=0$ for $0\leq j\leq k$, see \cite[Theorem 4.7.1]{Ama09}.  By using \cite[Theorem 2.4]{Dor93} and a spectral theoretic argument as in \cite{MeWi11} we may even allow for $\mu_D\in(-\lambda_0^D,\infty)$, $\mu_N\in(0,\infty)$.

In order to obtain higher regularity results we consider the spaces
\begin{align*}
  X^k_B := (\mu_B-\Delta_B)^{-k} L_p(\Omega;E), \quad \norm{u}_{X^k_B} := \norm{(\mu_B-\Delta_B)^k u}_{L_p(\Omega;E)} \quad\text{ for }k\in\N_0.
\end{align*}
These spaces can be easily characterized by
\begin{align*}
  X_B^k &= \{u\in W^{2k}_p(\Omega;E) : \gamma_B \Delta^j u = 0 \text{ for } 0\leq j\leq k-1\}.
\end{align*}
By commuting the operator $\mu_B+\del_t-\Delta_B$ with $(\mu_B-\Delta_B)^k$ it follows that $\mu_B-\Delta_B$ has maximal regularity of type $L_p(\R_+;X_B^k)$ for every $B\in\{D,N\}$, $k\in\N_0$, that is,
\begin{align*}
  \mu_B+\del_t-\Delta \colon  \0 W^1_p(\bbR_+;X_B^k) \cap L_p(\bbR_+;X_B^{k+1}) \to L_p(\bbR_+;X_B^k)
\end{align*}
is a topological linear isomorphism.  Moreover, the map
\begin{align*}
  \eps+\del_t \colon  \0 W^{l+1}_p(\R_+;E) \to \0 W^l_p(\R_+;E)
\end{align*}
is a topological linear isomorphism for every $\eps>0$ and every $l\in\N_0$, see e.~g.\ \cite{MeSc12}.  Hence, by commuting $\mu_B+\del_t-\Delta_B$ with $\eps+\del_t$, we obtain the following result.
\begin{lemma}\label{lem:interior_higher_regularity}
  Let $\mu_D\in(-\lambda_0^D,\infty)$, $\mu_N\in(0,\infty)$, $l\in\N_0$, $k\in\N_0$, $B\in\{D,N\}$.  Then the map
  \begin{align*}
    \mu_B+\del_t-\Delta_B &\colon  \0 W^{l+1}_p(\R_+;X_B^k) \cap \0 W^l_p(\R_+;X_B^{k+1}) \to \0 W^l_p(\R_+;X_B^k)
  \end{align*}
  is a topological linear isomorphism.  
\end{lemma}

We next comment on function spaces for the initial data.  From \cite[Section 4.9]{Ama09} we derive the characterization
\begin{align*}
  (X_B^k,X_B^{k+1})_{1-1/p,p} &= \{u\in W^{2k+2-2/p}_p(\Omega;E): \gamma_B\Delta^ju = 0 \text{ for }0\leq j\leq k-j_B/2-3/2p\}.
\end{align*}
Then the temporal trace operator
\begin{align*}
  \gamma_t \colon W^1_p(\R_+;X_B^k) \cap L_p(\R_+;X_B^{k+1}) \to (X_B^k,X_B^{k+1})_{1-1/p,p}
\end{align*}
is a bounded and surjective and therefore
\begin{align*}
  \(\mu_B+\del_t-\Delta_B,\gamma_t\) \colon W^1_p(\R_+;X_B^k) \cap L_p(\R_+;X_B^{k+1}) \to L_p(\R_+;X_B^k) \times (X_B^k,X_B^{k+1})_{1-1/p,p}
\end{align*}
is also a topological linear isomorphism for $B\in\{D,N\}$, $k\in\N_0$.

\subsection{Boundary conditions}
\label{ssec:boundary_conditions}




We will use the following result for constructing a function with prescribed boundary conditions \eqref{eq:150324}.

\begin{lemma}\label{lem:right_inverse_for_spatial_trace}
  Let $l\in\bbN_0$, $k\in\bbN$, $p\in(1,\infty)$, let $\gamma_\nu^j := (\del_\nu^j\cdot)|_\Gamma$ in the sense of traces and let
  \begin{align*}
    \0\bbG^{l,m/2}(\bbR_+\times\Gamma) := \0W^{l+m/2-1/2p}_p(\bbR_+;L_p(\Gamma;E)) \cap \0 W^l_p(\bbR_+;W^{m-1/p}_p(\Gamma;E)) \quad\text{ for }m\in\bbN.
  \end{align*}
  Then $\gamma_\nu^j \colon \0\bbE^{l,k}(\bbR_+\times\Omega) \to \0\bbG^{l,k-j/2}(\bbR_+\times\Gamma)$ is a retraction and the operator
  \begin{align*}
    \calB_{l,k} := (\gamma_\nu^0=\gamma_D,\gamma_\nu^1=\gamma_N,\gamma_\nu^2,\ldots,\gamma_\nu^{2k-1}) \colon \0\bbE^{l,k}(\bbR_+\times\Omega) \to {\prod}_{j=0}^{2k-1} \0\bbG^{l,k-j/2} (\bbR_+\times\Gamma)
  \end{align*}
  is a retraction.
\end{lemma}  
\begin{proof}
  In the case $\Omega=\bbR^n_+$, $l=0$ we infer from \cite[Theorem 4.11.6]{Ama09} that
  \begin{align*}
    \overline\calB_{0,k} := ((-1)^j\gamma_y^j)_{j=0}^{2k-1} \colon \0\bbE^{0,k}(\bbR_+\times\bbR^n_+) \to {\prod}_{j=0}^{2k-1}\0\bbG^{0,k-j/2}(\bbR_+\times\bbR^{n-1})
  \end{align*}
  is a retraction.  
  Let $\overline\calB_{0,k}^c$ denote a co-retraction for $\overline\calB_{0,k}$.

  In the case $\Omega=\bbR^n_+$, $l\in\bbN_0$ we use the fact that $(\epsilon+\del_t)^j \colon  \0W^{s+j}_p(\R_+;F) \to \0W^s_p(\R_+;F)$ is invertible for every $\epsilon>0$, $j\in\bbN$, $s\in[0,\infty)$ and every Banach space $F$ of class $\calHT$.  Therefore a co-retraction is given by $\calB^c_{l,k} = (\epsilon+\del_t)^{-l}\overline\calB^c_{0,k}(\epsilon+\del_t)^l$.

  For bounded smooth domains we define such operators by a localization technique.  It is well-known (see e.~g.\ \cite[Section 14.6]{GiTr01}, \cite{PrSi13}) that the tubular neighborhood map
  \begin{align*}
    X \colon  (x,t)\mapsto x+t\nu_\Gamma(x), \quad \Gamma\times(-R,R) \to B_R(\Gamma) := \{x\in\bbR^n : \operatorname{dist}(x,\Gamma) < R\}
  \end{align*}
  is a homeomorphism for some $R>0$.  Let $\{U_j : j\in I\}$ be a finite open covering of $\Gamma$ in $\bbR^n$ and let $\{\varphi_j:j\in I\} \subset C^\infty_c(\Gamma)$ be a partition of unity subordinate to $\{U_j\cap\Gamma:j\in I\}$.  Then there exists $r\in(0,R)$ such that $B_r(\Gamma)$ is covered by $\{U_j:j\in I\}$.  For given $\chi\in C^\infty_c((-r,r))$ with $0\leq\chi\leq 1$ and $\chi(t)=1$ for $\abs t\leq r/2$ we extend $\varphi_j$ to $\bbR^n$ by means of $\varphi_j(X(x,t)) := \varphi_j(x)\chi(t)$ for $(x,t)\in\Gamma\times(-r,r)$ so that $\supp\varphi_j\subset U_j$ and $\del_\nu^m\varphi_j=0$ near $\Gamma$ for all $m\geq 1$.

  In addition, let $U_j = B_r(x^{(j)})$ with $x^{(j)}\in\Gamma$ for some $r\in(0,R)$ and choose rigid transformations $\Xi_j \colon  x\mapsto x^{(j)}+Q_jx$ with $Q_j$ orthogonal such that $Q_j(-e_n) = \nu_\Gamma(x^{(j)})$.  There exist $\omega_j\in C_c^\infty(\bbR^{n-1})$ with $\omega_j(0)=\abs{\nabla\omega_j(0)} = 0$ such that for $\theta_j(x',x_n) := (x',x_n + \omega_j(x'))$ we have $U_j\cap\Omega = U_j \cap \Xi_j(\theta_j(\bbR^n_+))$ and thus $U_j\cap\Gamma = U_j\cap \Xi_j(\theta_j(\Gamma_0))$ with $\Gamma_0 := \bbR^{n-1} \times \{0\}$.  
Let us construct smooth diffeomorphisms $\Theta_j$ of $\R^n$ such that $U_j\cap\Omega = U_j \cap \Theta_j(\bbR^n_+)$ and $U_j\cap \Gamma = U_j \cap \Theta_j(\bbR^{n-1}\times\{0\})$.  Given $r\in(0,R/2)$, $\psi\in C_c^\infty(B_{2r}(0))$ with $\psi = 1$ on $B_r(0)$, let 
  \begin{align*}
    \Theta_j(x) = \left\{
      \begin{aligned}
        & \psi(x) \left[ \Xi_j(\theta_j(x',0)) - x_n \nu_\Gamma (\Xi_j(\theta_j(x',0))) \right] + (1-\psi(x)) \, \Xi_j(x) &\text{ for } \abs x &\leq 2r,\\
        & \Xi_j(x)&\text{ for } \abs x&\geq 2r.
      \end{aligned}\right.
  \end{align*}
  If $r\in(0,R/2)$ is sufficiently small, then $\Theta_j$ is a diffeomorphism since $\del_x\Theta_j(x) \to Q_j$   
  as $r\to 0$, uniformly on $\bbR^n$.  
  Moreover, $\Theta_j$ has the asserted properties and satisfies $-\del_n\Theta_j(x',0) = \nu_\Gamma(\Theta_j(x',0))$ and $\del_n^m\Theta_j(x',0) = 0$ for all $m\geq 2$ and $x'\in B_r(0)$.

  Choose smooth cut-off functions $\psi_j\in C^\infty_c(\Theta_j^{-1}(U_j))$ with $\psi_j=1$ on $\Theta_j^{-1}(\supp\varphi_j)$ and define the multiplication operator $M_j \colon  u \mapsto \psi_j u$.  With the pull-back $\Theta_j^* \colon  u \mapsto u\circ\Theta_j$ and the push-forward $\Theta_{j*} \colon  u\mapsto u\circ\Theta_j^{-1}$ we define a co-retraction for $\calB_{l,k}$ by
  \begin{align*}
    \calB^c_{l,k} g := {\sum}_{j\in I} \Theta_{j*} M_j\overline \calB^c_{l,k} \Theta_j^*(\varphi_j g) \quad\text{ for } g \in {\prod}_{j=0}^{2k-1}\0\bbG^{l,k-j/2}(\bbR_+\times\Gamma).
  \end{align*}

  By means of the chain rule, Hölder's inequality and the mixed derivative embeddings, it can be shown that the linear operators $g\mapsto \varphi_j g$, $\Theta_j^*$, $M_j$ and $\Theta_{j*}$ act continuously in the relevant spaces and the properties of $\Theta_j$ and $\varphi_j$ with respect to the normal direction imply that indeed $\calB_{l,k}\calB^c_{l,k}g = g $.  This concludes the proof of Lemma \ref{lem:right_inverse_for_spatial_trace}.
\end{proof}

\subsection{Proof of Theorem \ref{thm:higher_regularity_for_the_heat_equation}}
\label{proof:thm_higher_regularity_for_the_heat_equation}

We have already discussed the necessity of the regularity conditions and the compatibility conditions on $(f,g,u_0)$.  It remains to prove the uniqueness and existence of a solution $u\in\bbE^{l,k}$ for given data $(f,g,u_0)$ subject to these conditions.

In order to prove \emph{uniqueness}, it suffices to consider the most general case $l=0$, $k=1$, where $\bbE^{l,k} = W^1_p(J;L_p(\Omega)) \cap L_p(J;W^2_p(\Omega))$ and $(f,g,u_0)=0$.  If further $\mu_0$ is sufficiently large, then the general result of \cite{DHP03} implies that $\mu_0-\Delta$ has maximal regularity of type $L_p(\bbR_+;L_p(\Omega))$ and this yields $u=0$ in case $\mu_B\geq\mu_0$.  

Next, we employ spectral theory to cover the case $\mu_B\in(-\lambda_0^B,\infty)$, where $\lambda_0^B=\lambda_0(-\Delta_B)\geq 0$ denotes the smallest eigenvalue of $-\Delta_B$.  It is well known that, since $\D(\Delta_B)$ is compactly embedded into $L_p(\Omega)$, the spectrum of $\Delta_B$ is discrete and consists only of eigenvalues with finite multiplicity.  As for Lemma \ref{lem:A-4}, we see that its eigenfunctions belong to
\begin{align*}
  \D(\Delta_B^m) = \left\{ u \in W^{2m}_p(\Omega) : \gamma_B(\Delta^j u)|_\Gamma = 0 \text{ on }\Gamma\text{ for }j\leq m-1\right\}
\end{align*}
for every $m\in\N$ and hence belong to $W^2_p(\Omega)$ and an integration by parts implies that the spectrum of $\Delta_B$ is contained in $(-\infty,-\lambda_0^B]$.  A result of Dore \cite[Theorem 2.4]{Dor93} implies that $\mu_B-\Delta_B$ has maximal regularity of type $L_p(\bbR_+;L_p(\Omega))$ for each $\mu_B\in(-\lambda_0^B,\infty)$ and this ensures uniqueness.

\emph{Existence.}  We construct a solution $u = u^1 + u^2 + u^3\in\bbE^{l,k}$ such that
\begin{align*}
  \begin{aligned}
    (\del_t+\mu_B-\Delta)u^1 &=: f^1, &&& \gamma_t^i u^1 &= u_i,\\
    (\del_t+\mu_B-\Delta)u^2 &=: f^2, & \gamma_B\Delta^ju^2 &= g_j-\gamma_B\Delta^ju^1, & \gamma_t^i u^2 &= 0,\\   
    (\del_t+\mu_B-\Delta)u^3 &= f-f^2-f^1, & \gamma_B\Delta^ju^3 &= 0, & \gamma_t^iu^3 &= 0,
  \end{aligned}
\end{align*}
for all $i$, $j$ with $0\leq i \leq l+k-1$ and $0\leq j\leq k-1$.  This means that we first construct $u^1\in\bbE^{l,k}$ with prescribed initial data $u_i$.   Then we construct $u^2$ with prescribed boundary data $g_j-\gamma_B\Delta^j u^1$ and we finally we construct $u^3$ with prescribed interior data $f-f^2-f^1$.  Here the functions $u_i$ and $g_j$ are defined according to \eqref{eq:150323_2} and \eqref{eq:150324} by
\begin{alignat}{3}
  u_i &:= \gamma_t^{i-1}f + (\Delta-\mu_B)u_{i-1} & \text{ for } i&\in\bbN\cap[1,l+k-1],&&\\
  g_j &:= -\gamma_B \Delta^{j-1} f + (\del_t+\mu_B) g_{j-1} & \text{ for } j&\in\bbN\cap[1,k-1], & g_0 &:= g.
\end{alignat}

\emph{Construction of $u^1$.}  
Let $r^c_\Omega$ be a common co-retraction for the restriction $r_\Omega \colon  W^t_p(\bbR^n) \to W^t_p(\Omega)$ for all $t\in[0,2k]$ (cf.\ \cite[Theorem 5.22]{AdFo03}).  With the co-retraction $S^A$ for the operator $(\gamma_t^0,\ldots,\gamma_t^{l+k-1})$ from Lemma \ref{lem:temp-trace-retraction} we define
\begin{align*}
  u^1 &:= r_\Omega S^I(r^c_\Omega u_0,r^c_\Omega u_1,\ldots,r^c_\Omega u_{l-1},0,\ldots,0) \\
  &\quad+ r_\Omega S^{I-\Delta}(0,\ldots,0,r^c_\Omega u_l,\ldots,r^c_\Omega u_{l+k-1}).
\end{align*}
Here we consider the identity operator $I\colon \D(I)\to X$ with $X=\D(I) = W^{2k}_p(\bbR^n)$ so that the first summand of $u^1$ belongs to $(t\mapsto\Exp^{-t})BUC^\infty(\bbR_+;W^{2k}_p(\Omega)) \hookrightarrow \bbE^{l,k}$.  In the second summand, we consider the operator $I-\Delta \colon  \D(\Delta)\to X$ in $X=L_p(\bbR^n)$ with domain $\D(\Delta)=W^2_p(\bbR^n)$ so that $r_\Omega^c u_l \in \D_{\Delta}(k-1/p,p)$ and thus $S^{I-\Delta}(0,\ldots,0,r_\Omega^c u_l,\ldots,r_\Omega^c u_{l+k-1}) \in W^{l+k}_p(\bbR_+;L_p(\bbR^n)) \cap L_p(\bbR_+;W^{2l+2k}_p(\bbR^n))$.  Therefore $u^1$ belongs to $\bbE^{l,k}$, satisfies $\gamma_t^iu^1 = u_i$ and depends continuously on the data $(f,u_0)$ with respect to the norms induced by the regularity conditions on $(f,u_0)$ and the compatibility condition \eqref{eq:150323_2}.  

\emph{Construction of $u^2$.}  From \eqref{eq:150421} it follows that $\gamma_t^ig_j-\gamma_B\Delta^j u_i = 0$ for all $i,j\in\bbN_0$ with $j\leq k-1$, $i\leq l+k-j-1$.  Thus $g_j-\gamma_B\Delta^j u^1$ belongs to $\gamma_B\0\bbE^{l,k-j}$.  Near $\Gamma$ we can split the Laplacian into $\Delta = \Delta_\Gamma + H_\Gamma \del_\nu + \del_\nu^2$ with the Laplace-Beltrami operator $\Delta_\Gamma = \operatorname{div}_\Gamma \nabla_\Gamma$ and some $H_\Gamma\in C^\infty(\Gamma)$.  The operator $\Delta_\Gamma$ commutes with $\del_\nu$ since it only depends on tangential derivatives.  Therefore the normal traces $h_j := (\del_\nu^j u^2)|_\Gamma$ ($j\in\{0,\ldots,2k-1\}$) of the desired solution $u^2\in\0\bbE^{l,k}$ are uniquely determined by requiring that $h_{2j+j_B+1} = 0$ and
\begin{align*}
  \gamma_B(\Delta_\Gamma+H_\Gamma\del_\nu+\del_\nu^2)^j u^2 = g_j-\gamma_B\Delta^j u^1 \quad\text{ for } 0\leq j\leq k-1.
\end{align*}
With the co-retraction $\calB_{l,k}^c$ from Lemma \ref{lem:right_inverse_for_spatial_trace} we define $u^2 := \calB_{l,k}^c(h_0,\ldots,h_{2k-1})$ which depends linearly and continuously on $(f,g,u_0)$ with respect to the norms induced by \eqref{eq:150326} and \eqref{eq:150323}.

\emph{Construction of $u^3$.}  Finally, we shall construct a function
\begin{align*}
  u^3 \in \0W^{l+k}_p(\bbR_+;L_p(\Omega)) \cap \0W^l_p(\bbR_+;\D(\Delta_B^k)),
\end{align*}
which solves the equation $(\del_t+\mu_B-\Delta)u^3 = f^3 := f-f^1-f^2$, where $f^m := (\del_t+\mu_B-\Delta)u^m$ for $m\in\{1,2\}$.  The compatibility conditions yield $\gamma_B\Delta^jf = 0$ for $0\leq j\leq k-2$ and thus
\begin{align*}
  f^3 \in \0W^{l+k-1}_p(J;L_p(\Omega)) \cap \0W^l_p(J;\D(\Delta_B^{k-1})).
\end{align*}
Hence $u^3 = (\del_t+\mu_B-\Delta)^{-1}f^3$ is well-defined by Lemma \ref{lem:interior_higher_regularity} and the map $f^3 \mapsto u^3$ is continuous.  The proof of Theorem \ref{thm:higher_regularity_for_the_heat_equation} is complete.

\begin{proof}[Proof of Corollary \ref{cor:higher_regularity_for_the_heat_equation_Neumann_stability}]\label{proof:cor_higher_regularity_for_the_heat_equation_Neumann_stability}
  With the same arguments as above and Theorem \ref{thm:spectrum_Neumann_Laplacian} we see that $\mu_N-\Delta_N$ has maximal regularity of type $L_p(\R_+;L_{p,0}(\Omega))$.  Hence problem \eqref{eq:ap_heat_problem} has at most one solution within the space $\bbE^{l,k}_0$.  Analogously as for Lemma \ref{lem:interior_higher_regularity}, we conclude that
  \begin{align*}
    \del_t+\mu_N-\Delta_N &\colon  \0W^{l+k}_p(\bbR_+;L_{p,0}(\Omega)) \cap \0W^l_p(\bbR_+;\D(\Delta_N^k)) \\
    &\quad \to \0W^{l+k-1}_p(\bbR_+;L_{p,0}(\Omega)) \cap \0W^l_p(\bbR_+;\D(\Delta_N^{k-1}))
  \end{align*}
  is a topological linear isomorphism.  For the proof of existence, we modify the above construction of the solution $u = u^1+u^2+u^3$.  For $i\in\{1,2\}$ we may replace $u^i$ by $u^i-\bar u^i \in \bbE^{l,k}_0$, since $\del_\nu\bar u^i = 0$ and $\bar u^i(0) = \bar u_0 = 0$.  Then we obtain $\bar f^3(t) = \bar f(t) + \abs\Gamma\abs\Omega^{-1}\bar g(t) = 0$.  Hence $u^3 := (\del_t+\mu_N-\Delta)^{-1}f^3$ is well-defined in $\bbE^{l,k}_0$.
\end{proof}

\bigskip

\textbf{Acknowledgements.} R.B. gratefully acknowledges the funding of his research by the Austrian Science Fund (FWF): P24970 as well as the
support of the Karl Popper Kolleg "Modeling-Simulation-Optimization"\ funded by the Alpen-Adria-Universit\"at Klagenfurt and by the Carinthian 
Economic Promotion Fund (KWF).

\bibliography{BrMe14_Lit_edit}
\bibliographystyle{alpha}

\end{document}